\documentclass[12pt]{amsart}
\usepackage{amssymb}
\usepackage{amsmath}
\usepackage{amsthm}
\usepackage{mathrsfs}
\usepackage{graphicx}
\usepackage{geometry}
\usepackage{float}
\usepackage{indentfirst}
\usepackage{color}
\allowdisplaybreaks[4]

\pdfoutput=1

\numberwithin{equation}{section}

\newtheorem{thm}{Theorem}[section]
\newtheorem{defi}[thm]{Definition}
\newtheorem{rem}[thm]{Remark}
\newtheorem{lem}[thm]{Lemma}
\newtheorem{pro}[thm]{Proposition}
\newtheorem{coro}[thm]{Corollary}
\newtheorem{example}[thm]{Example}

\begin{document}
\title[variable martingale Hardy-Lorentz-Karamata spaces]{Variable martingale Hardy-Lorentz-Karamata spaces and their applications in Fourier Analysis}








\author{Zhiwei Hao \\
S\lowercase{chool of} M\lowercase{athematics and} C\lowercase{omputing} S\lowercase{cience}, H\lowercase{unan} U\lowercase{niversity of} S\lowercase{cience and} T\lowercase{echnology}, \\ X\lowercase{iangtan  411201}, P\lowercase{eople's} R\lowercase{epublic of} C\lowercase{hina} \\
\lowercase{E-mail: haozhiwei@hnust.edu.cn}
\\ \\
Xinru Ding \\
S\lowercase{chool of} M\lowercase{athematics and} C\lowercase{omputing} S\lowercase{cience}, H\lowercase{unan} U\lowercase{niversity of} S\lowercase{cience and} T\lowercase{echnology}, \\ X\lowercase{iangtan  411201}, P\lowercase{eople's} R\lowercase{epublic of} C\lowercase{hina} \\
\lowercase{E-mail: dingxinru0724@163.com} \\ \\
Libo Li \\
S\lowercase{chool of} M\lowercase{athematics and} C\lowercase{omputing} S\lowercase{cience}, H\lowercase{unan} U\lowercase{niversity of} S\lowercase{cience and} T\lowercase{echnology}, \\ X\lowercase{iangtan  411201}, P\lowercase{eople's} R\lowercase{epublic of} C\lowercase{hina} \\
\lowercase{E-mail: lilibo@hnust.edu.cn} \\ \\
Ferenc Weisz$^*$
\\
D\lowercase{epartment of} N\lowercase{umerical} A\lowercase{nalysis}, E\lowercase{\"{o}tv\"{o}s} L. U\lowercase{niversity}, \\ H-1117 B\lowercase{udapest},
P\lowercase{\'{a}zm\'{a}ny} P. \lowercase{s\'{e}t\'{a}ny} 1/C, H\lowercase{ungary} \\
\lowercase{E-mail: weisz@inf.elte.hu}
}

\thanks{$^*$Corresponding author, email: weisz@inf.elte.hu}

\begin{abstract}
	In this paper, we introduce a new class of function spaces, which unify and generalize Lorentz-Karamata spaces, variable Lorentz spaces and other several classical function spaces. Based on the new spaces, we develop the theory of variable martingale Hardy-Lorentz-Karamata spaces and apply it to Fourier Analysis.
    To be precise, we discuss the basic properties of Lorentz-Karamata spaces with variable exponents.
    We introduce five variable martingale Hardy-Lorentz-Karamata spaces and characterize them via simple atoms as well as via atoms.
    As applications of the atomic decompositions, dual theorems and the generalized John-Nirenberg theorem for the new framework are presented.
    Moreover, we obtain the boundedness of $\sigma$-sublinear operator defined on variable martingale Hardy-Lorentz-Karamata spaces, which leads to martingale inequalities and the relation of the five variable martingale Hardy-Lorentz-Karamata spaces. Also, we investigate the boundedness of fractional integral operators in this new framework. Finally, we deal with the applications of variable martingale Hardy-Lorentz-Karamata spaces in Fourier analysis by using the previous results. More precisely, we show that the partial sums of the Walsh-Fourier series converge to the function in norm if $f\in L_{p(\cdot),q,b}$ with $1<p_-\le p_+<\infty$. The  Fej\'{e}r summability method is also studied and it is proved that the maximal Fej\'{e}r operator is bounded from variable martingale Hardy-Lorentz-Karamata spaces to variable Lorentz-Karamata spaces. As a consequence, we get conclusions about almost everywhere and norm convergence of Fej\'{e}r means.
    The results obtained in this paper generalize the results for martingale Hardy-Lorentz-Karamata spaces and variable martingale Hardy-Lorentz spaces.
    Especially, we remove the condition that $b$ is nondecreasing in previous literature.
\end{abstract}

\maketitle

\markboth{Z. Hao, X. Ding, L. Li and F. Weisz}
{Variable martingale Hardy-Lorentz-Karamata spaces}


\section{Introduction}

Let $0<p<\infty$, $0<q\le\infty$ and $b$ be a slowly varying function. The Lorentz-Karamata space $L_{p,q,b}(\Omega,\mathcal{F},\mathbb{P})$, or briefly $L_{p,q,b}$, is defined to be the set of all measurable functions $f$ on a probability space $(\Omega,\mathcal{F},\mathbb{P})$ such that $\|f\|_{p,q,b}<\infty$, where
$$
\|f\|_{p,q,b}:=\left\{
\begin{aligned}
	 &\bigg[\int_{0}^{\infty}\big(\mathbb{P}(|f|>t)^{1/p}\gamma_b(\mathbb{P}(|f|>t))t\big)^q\frac{dt}{t}\bigg]^\frac{1}{q},\ \mathrm{if}\ 0<q<\infty,\\
	&\sup_{t>0} t \mathbb{P}(|f|>t)^{1/p}\gamma_b(\mathbb{P}(|f|>t)),\ \ \ \ \ \ \ \ \ \ \ \ \ \ \mathrm{if}\ q=\infty,
\end{aligned}
\right.
$$
where $\gamma_b(t)=b(t^{-1})$, $0<t<1$.
This class of function spaces was introduced in 2000 by Edmunds et al. in \cite{ekp} and it was used to investigate the important problem of the optimal Sobolev embeddings on regular domains in the Euclidean space. By taking different $p$, $q$ and $b$, these spaces generalize the classical Lebesgue spaces, Lorentz spaces, Zygmund spaces, Lorentz-Zygmund spaces and the generalized Lorentz-Zygmund spaces. The theory of Lorentz-Karamata spaces is not only a pursuit of mere generality, but it shows the essential points of the issues while we are less likely  to be overwhelmed by the technical details. Some significant results of the Lorentz-Karamata spaces are studied in \cite{ee,hoa,jxz,lz,n,sd,wzp} and the references therein.

Another interesting topic we shall touch in our work is a kind of function spaces with variable exponents.
The variable Lebesgue space $L_{p(\cdot)}$ is an important class of non-rearrangement invariant spaces and it is defined by the (quasi)-norm
$$
\|f\|_{p(\cdot)}:=\inf \left\{\lambda>0: \int_{\mathbb{R}^n} \left(\frac{|f(x)|}{\lambda} \right)^{p(x)}dx\le 1 \right\}.
$$
Recently, this topic has a great development and appears in the research on endpoint analysis and in different applications, such as elasticity, fluid dynamics, variational calculus and partial differential equations (see \cite{am,al, brr, bds, dr, hhln, p, rs, t} and so on). Such spaces were firstly introduced by Orlicz \cite{o} in 1931. Kov\`{a}\v{c}ik and R\'{a}kosn\'{i}k \cite{kr}, Fan and Zhao \cite{fz} investigated various properties of variable Lebesgue spaces and variable Sobolev spaces.
Since some well known properties do not hold, such as the boundedness of the translation operator and the maximal inequality in a modular form, the theory of variable function spaces is difficult.
A fundamental breakthrough in this topic is due to Diening, who introduced the so-called log-H\"{o}lder continuity condition as follows:
$$|p(x)-p(y)|\lesssim\frac{1}{\log(1/|x-y|)}$$
and
$$|p(x)-p_\infty|\lesssim\frac{1}{\log(e+|x|)}$$
for all $x,y\in\mathbb{R}^n$, see \cite{d1, d2}.
Heavily relying on this condition, the theory of variable Hardy spaces has got a rapid development. One of the most important results states that the classical Hardy-Littlewood maximal operator is bounded on variable Lebesgue spaces (see Cruz-Uribe et al. \cite{CruzUribe2011,cfn}, Nekvinda \cite{nh}, Cruz-Uribe and Fiorenza \cite{cf} and Diening et al. \cite{dhhr}). Nakai and Sawano \cite{nsh} firstly introduced the Hardy space $H_{p(\cdot)}(\mathbb{R}^n)$ and discussed the atomic decompositions. They also studied the duality and the boundedness of singular integral operators as applications of atomic decompositions. Under some weaker conditions on $p(\cdot)$, Cruz-Uribe and Wang \cite{cw} and Sawano \cite{sa} also characterized the variable Hardy spaces $H_{p(\cdot)}(\mathbb{R}^n)$ and extended the conclusions in \cite{nsh}. Later, Yan et al. \cite{yyyz} considered the variable weak Hardy spaces $H_{p(\cdot),\infty}(\mathbb{R}^n)$ and characterized these spaces by the radial maximal functions, atoms and Littlewood-Paley functions. Meanwhile, Jiao et al. \cite{jzzw} investigated the variable Hardy-Lorentz spaces $H_{p(\cdot),q}(\mathbb{R}^n)$. The relevant conclusions on the anisotropic Hardy spaces can be found in Liu et al. \cite{lwyy1,lwyy2} and the references therein. Very recently, Weisz \cite{wn} introduced a new fractional maximal operator and showed that it is bounded from variable Hardy spaces to variable Lebesgue spaces. He also proved that a new type of maximal operators is bounded from variable Hardy-Lorentz spaces to variable Lorentz spaces in \cite{we}.

Motivated by the study of Lorentz-Karamata spaces and variable Lorentz spaces, we introduce a new type of function spaces, the so-called variable Lorentz-Karamata space $L_{p(\cdot),q,b}$ by the quasi-norm
$$
\|f\|_{p(\cdot),q,b}=\left\{
	\begin{aligned}
		 &\bigg[\int_{0}^{\infty}\big(t\|\chi_{\{|f|>t\}}\|_{p(\cdot)}\gamma_b(\|\chi_{\{|f|>t\}}\|_{p(\cdot)})\big)^q\frac{dt}{t}\bigg]^{\frac{1}{q}},\ \mathrm{if}\ 0<q<\infty,\\
		 &\sup_{t>0}t\|\chi_{\{|f|>t\}}\|_{p(\cdot)}\gamma_b(\|\chi_{\{|f|>t\}}\|_{p(\cdot)}),\ \ \ \ \ \ \ \ \ \ \ \ \ \mathrm{if}\ q=\infty.
	\end{aligned}
	\right.
$$
The variable Lorentz-Karamata spaces go back to the Lorentz-Karamata spaces and variable Lorentz spaces when $p(\cdot)\equiv p$ is a constant and $b\equiv1$, respectively. Moreover, we prove that the variable Lorentz-Karamata space is a quasi-Banach space. We also give some fundamental geometric properties for these new spaces.

Inspired by the progress of real theory of variable Hardy spaces, the martingale theory of variable Hardy spaces has gained widespread attentions.
The major difficulty in martingale theory is that
the probability space has no natural metric structure compared with Euclidean space $\mathbb{R}^n$, which means that the above-mentioned log-H\"{o}lder continuity condition can no longer hold. In the last decade, Hao and Jiao \cite{hj} in 2015 (see also \cite{jwzw,jzhc}) overcame the last obstacle, and found a suitable replacement of the above log-H\"{o}lder continuity condition. That is, suppose that there exists a constant $K_{p(\cdot)}\ge1$ depending only on $p(\cdot)$ such that
\begin{eqnarray}\label{gs3}
	\mathbb{P}(A)^{p_-(A)-p_+(A)}\le K_{p(\cdot)},\ \ \ \forall\  A\in\bigcup_{n\ge0}A(\mathcal{F}_n),
\end{eqnarray}
where $\mathcal{F}_n$ is generated by countable many atoms and $A(\mathcal{F}_n)$ denotes the family of all atoms in $\mathcal{F}_n$. Under this condition, Jiao et al. \cite{jwzw, jzhc} proved the Doob maximal inequality and
developed the martingale theory in the framework of variable exponent setting. Since then,
this branch of martingale theory has become a very active area of research,
see \cite{hly,jwwzr,jwwz,jwzw,jzz,FW3,wdoob,wn,hjl}. Martingale Musielak-Orlicz Hardy spaces were considered in Xie et al. \cite{hlxy2023,hpx2023,jwxy,Xie2019,wguangheng,Xie2019a}.


We continue this line of investigation. We introduce a new type of Hardy spaces, the so called variable martingale Hardy-Lorentz-Karamata spaces, which are much more wider than the martingale Hardy-Lorentz-Karamata spaces and variable martingale Hardy-Lorentz spaces.
The purpose of this paper is to make a systematic study of these new spaces. Under the condition (\ref{gs3}), we discuss atomic decompositions, dual theorems, martingale inequalities and the relation of the different martingale Hardy spaces. Compared with the previous literature, we obtain some stronger conclusions even for martingale Hardy-Lorentz-Karamata spaces and variable martingale Hardy-Lorentz spaces. In detail, we remove the restriction that $b$ is a nondecreasing slowly varying function in \cite{jxz,lz,wzp} and extend the range of $q$ in \cite{hoa,jwwz,jxz,kv,lz}. Moreover, we give some applications of variable martingale Hardy-Lorentz-Karamata spaces in Fourier analysis. For example, under some conditions, we show the convergence of the partial sums of the Walsh-Fourier series and the boundedness of the maximal Fej\'{e}r operator. As a consequence, we get conclusions about almost everywhere and norm convergence of Fej\'{e}r means.

This paper is organized as follows. In Section \ref{s2}, we give the definition of variable Lorentz-Karamata spaces and some properties which will be used in the subsequent sections. We also introduce five types of variable martingale Hardy-Lorentz-Karamata spaces, which generalize the variable martingale Hardy-Lorentz spaces and martingale Hardy-Lorentz-Karamata spaces.

The target of Section \ref{s3} is to establish the atomic decompositions of variable martingale Hardy-Lorentz-Karamata spaces. We generalize several known inequalities from martingale Hardy-Lorentz-Karamata spaces and variable martingale Hardy-Lorentz spaces to variable martingale Hardy-Lorentz-Karamata spaces. More exactly, we remove the restriction of $0<p\le1$ and $0<q\le p$ in \cite[Theorem 5.3]{hoa}; the condition that $b$ is nondecreasing in \cite[Theorem 3.1]{jxz}. Moreover, it is worth noting that in \cite{hoa,jxz}, they only considered the atomic decomposition theorems with “$\infty$-atoms”. We discuss the atomic decompositions via “simple $r$-atoms” and “$\infty$-atoms”, respectively. We extend and complete the relevant conclusions in \cite{hoa,jwwz,jwzw,jxz}.

The objective of Section \ref{s4} is to show the martingale inequalities between different variable martingale Hardy-Lorentz-Karamata spaces. If $\{\mathcal{F}_n\}_{n\ge0}$ is regular, the equivalence of different variable martingale Hardy-Lorentz-Karamata spaces is proved. The way we used is to find a sufficient condition for a $\sigma$-sublinear operator to be bounded from variable martingale Hardy-Lorentz-Karamata spaces to variable Lorentz-Karamata spaces. The proof depends on the atomic decompositions of the variable martingale Hardy-Lorentz-Karamata spaces. We point out that, the martingale inequalities extend Theorem $4.11$ in \cite{jwzw} and Theorems $1.2$, $1.3$ in \cite{wzp}.

In Section \ref{s5}, we present the dual theorems of $H_{p(\cdot),q,b}$ as an application of the atomic decomposition theorems. We give the definitions of the generalized martingale spaces $BMO_{r,b}(\alpha(\cdot))$ and $BMO_{r,q,b}(\alpha(\cdot))$, the special case of which are the $BMO$ spaces in \cite{jwwz,jxz,w}. We firstly get that the dual space of $H_{p(\cdot),q,b}^s$ is $BMO_{2,b}(\alpha(\cdot))$ if $0<p_-\le p_+<1$ and $0<q\le1$. The result extends the dual theorems in \cite{hl,hoa,jwwz,jxz,w}. Next we consider the case  $0<p_-\le p_+<2$, $0<q<\infty$ and prove that the dual of $H_{p(\cdot),q,b}^s$ is $BMO_{2,q,b}(\alpha(\cdot))$. For $q=\infty$, the essential difficulty is that $L_p$ is not dense in $L_{p,\infty}$ for $0<p<\infty$. To overcome this difficulty, we introduce a closed subspace of $H_{p(\cdot),q,b}^s$, namely $\mathcal{H}_{p(\cdot),\infty,b}^s$ and verify that $L_2$ is dense in $\mathcal{H}_{p(\cdot),\infty,b}^s$. We show that the dual space of $\mathcal{H}_{p(\cdot),\infty,b}^s$ is $BMO_{2,\infty,b}(\alpha(\cdot))$. Furthermore, we note that in the dual theorems, $b$ is not necessarily nondecreasing. Hence, we further extend the relevant conclusions in \cite{hoa,jwwz,jxz,lz}.

Section \ref{s6} is devoted to the John-Nirenberg theorems on variable Lorentz-Karamata spaces when the stochastic basis $\{\mathcal{F}_n\}_{n\ge0}$ is regular. For a constant exponent, the well known classical John-Nirenberg theorem can be found in \cite{g}. Jiao et al. \cite{jxz} showed the John-Nirenberg theorem for $BMO_{r,q,b}(\alpha)$. Moreover, Jiao et al. \cite{jwyy} discussed the John-Nirenberg theorem for $BMO_{r,q}(\alpha)$. For the generalized $BMO$ martingale spaces associated with variable exponents, the John-Nirenberg theorems were proved in \cite{jwwz,ywj}. In this section, we generalize these theorems to variable Lorentz-Karamata spaces.

In Section \ref{s7}, we deal with the boundedness of fractional integrals in $H_{p(\cdot),q,b}^M$. In martingale theory, Chao and Ombe \cite{co} firstly introduced the fractional integrals for dyadic martingales. The boundedness of fractional integrals on martingale Hardy spaces for $0<p\le1$ was proved by Sadasue \cite{sf}. In \cite{ans,jwyy,lz,nsm}, the notion of fractional integrals was extended to more general martingales. For fractional integrals associated with variable exponents, Hao and Jiao \cite{hj} considered the boundedness on variable martingale Hardy spaces.
The boundedness of fractional integrals studied in Section \ref{s7} goes back to Liu and Zhou \cite[Theorem 4.4]{lz} if $p(\cdot)\equiv p$ is a constant.

In the last section, we investigate some applications of variable martingale Hardy-Lorentz-Karamata spaces in Fourier analysis. Martingale theory has extensive applications in dyadic harmonic analysis and in summability of Walsh-Fourier series, for constant exponents, see \cite{Gat2014d,gg,Gat2009a,Goginava2005,Goginava2006,Goginava2007a,Memic2016,Persson2015,su,ss,swsp,s,Simon2004,sw,ws,www} and the references therein. Using dyadic martingale theory, Weisz \cite{ws} verified that the maximal Fej\'{e}r operator $\sigma_*$ is bounded from $H_{p,q}$ to $L_{p,q}$ under the condition of $p>\frac{1}{2}$. Motivated by this result, Jiao et al. \cite{jwzw} discussed the boundedness of $\sigma_*$ from $H_{p(\cdot)}$ to $L_{p(\cdot)}$ and from $H_{p(\cdot),q}$ to $L_{p(\cdot),q}$. In Section \ref{s8}, we show that the partial sums of the Walsh-Fourier series converges to the function in norm if $f\in L_{p(\cdot),q,b}$ with $1<p_-\le p_+<\infty$. The boundedness of the maximal Fej\'{e}r operator from variable martingale Hardy-Lorentz-Karamata spaces to variable Lorentz-Karamata spaces is also proved. As a consequence, we get conclusions about almost everywhere and norm convergences of Fej\'{e}r means. Furthermore, the conclusions we get here extend the results in \cite{jwzw,ws} when $b\equiv1$. Note that these results are new even if $p(\cdot)$ is a constant and $b\not\equiv1$.

Finally, we end this section by making some conventions on notation. Throughout this paper, we denote by $\mathbb{Z}$ and $\mathbb{N}$ the set of integers and the set of nonnegative integers, respectively. We denote by $c$ a positive constant, which can vary from line to
line. The symbol $A\lesssim B$ stands for the inequality $A\le cB$. If we write $A\approx B$, then it means $A\lesssim B\lesssim A$. We use $\chi_I$ to denote the indicator function of a measurable set $I$.

\section{Preliminaries} \label{s2}
In this section, we introduce a new class of function spaces, namely the variable Lorentz-Karamata spaces. Meanwhile, some fundamental geometric properties of these spaces are given, including completeness and dominated convergence theorems. Based on the new spaces, five types of variable martingale Hardy-Lorentz-Karamata spaces are given in the last subsection. Except where otherwise stated, we always assume that $(\Omega,\mathcal{F},\mathbb{P})$ is a complete probability space.
\subsection{Variable Lebesgue spaces}
A measurable function $p(\cdot):\Omega\rightarrow(0,\infty)$ is called a variable exponent. For a measurable subset $A\subset\Omega$, we write
$$p_+(A):=\mathop{\mathrm{ess}\sup}\limits_{\omega\in A}p(\omega)\ \ \text{and}\ \ \ p_-(A):=\mathop{\mathrm{ess}\inf}\limits_{\omega\in A}p(\omega).$$
For brevity, we use the abbreviations
$$
p_+:=p_+(\Omega) \qquad \text{and} \qquad p_-:=p_-(\Omega).
$$
Denote by $\mathcal{P}(\Omega)$ the collection of all variable exponents $p(\cdot)$ such that $0<p_-\le p_+<\infty$. Throughout the paper, given a variable exponent $p(\cdot)$, the conjugate variable exponent $p'(\cdot)$ of $p(\cdot)$ is defined pointwise by
$$\frac{1}{p(\omega)}+\frac{1}{p'(\omega)}=1,\ \ \quad  \omega\in\Omega.$$

The variable Lebesgue space $L_{p(\cdot)}:=L_{p(\cdot)}(\Omega)$ is the set of all measurable functions $f$ defined on $(\Omega,\mathcal{F},\mathbb{P})$ such that for some $\lambda>0$,
$$\rho(f/\lambda)=\int_{\Omega}\bigg(\frac{|f(\cdot)|}{\lambda}\bigg)^{p(\cdot)}d\mathbb{P}<\infty.$$
The variable Lebesgue space equipped with the (quasi)-norm
$$\|f\|_{p(\cdot)}:=\inf\big\{\lambda>0:\rho(f/\lambda)\le 1 \big\}$$
becomes a (quasi)-Banach function space.
We present some basic properties of functional $\|\cdot\|_{p(\cdot)}$, which can be found in \cite{cf,dhhr,nsh}.

\begin{rem}\label{property}
	For any $f,g\in L_{p(\cdot)}$, the following properties hold:
	
	$1$. $\|f\|_{p(\cdot)}\ge0$, $\|f\|_{p(\cdot)}=0$ if and only if $f=0$.
	
	$2$. $\|cf\|_{p(\cdot)}=|c|\|f\|_{p(\cdot)}$ for any $c\in\mathbb{C}$.
	
	$3$. for $0<l\le \underline{p}:=\min\{p_-,1\}$,
	$$\|f+g\|_{p(\cdot)}^l\le\|f\|_{p(\cdot)}^l+\|g\|_{p(\cdot)}^l \qquad \mbox{and} \qquad \|f+g\|_{p(\cdot)}\le2^{1/l-1}\big(\|f\|_{p(\cdot)}+\|g\|_{p(\cdot)}\big).$$
\end{rem}

Moreover, we have the following lemmas for variable Lebesgue spaces, which will be useful in the whole paper.
\begin{lem} [\cite{jwwz}]\label{fansanjiao}
	Let $p(\cdot)\in\mathcal{P}(\Omega)$ with $p_+\le1$. For any positive functions $f,g\in L_{p(\cdot)}$, we have $$\|f\|_{p(\cdot)}+\|g\|_{p(\cdot)}\le\|f+g\|_{p(\cdot)}.$$
\end{lem}

\begin{lem} [\cite{cw}]\label{lem2.1}
	Let $p(\cdot)\in\mathcal{P}(\Omega)$ and $s>0$. Then for all $f\in L_{sp(\cdot)}$, there is $$\||f|^s\|_{p(\cdot)}=\|f\|_{sp(\cdot)}^s.$$
\end{lem}

\begin{lem}[\cite{cf} or \cite{fz}]
	Let $p(\cdot)\in\mathcal{P}(\Omega)$. For any $f\in L_{p(\cdot)}$, we have
	
	$(1)$ $\|f\|_{p(\cdot)}<1$ $($resp. $=1,>1)$ if and only if $\rho(f)<1$ $($resp. $=1,>1);$
	
	$(2)$ if $\|f\|_{p(\cdot)}>1$, then
	$$\rho(f)^{1/p_+}\le\|f\|_{p(\cdot)}\le\rho(f)^{1/p_-};$$
	
	$(3)$ if $0<\|f\|_{p(\cdot)}\le1$, then
	$$\rho(f)^{1/p_-}\le\|f\|_{p(\cdot)}\le\rho(f)^{1/p_+}.$$
\end{lem}

\begin{lem}[\cite{cf}]\label{2.4}
	Let $p(\cdot),q(\cdot)\in\mathcal{P}(\Omega)$. If $p(\cdot)\le q(\cdot)$, then
	$$\|f\|_{p(\cdot)}\le2\|f\|_{q(\cdot)}\ \ \text{for every}\ f\in L_{q(\cdot)}.$$
\end{lem}

\subsection{Slowly varying functions}

For a function $f:[1,\infty)\rightarrow(0,\infty)$, we say that $f$ is equivalent to a nondecreasing (resp. nonincreasing) function $g$ if $f\approx g$. In order to define the Lorentz-Karamata spaces with variable exponents, we recall the definition of slowly varying functions.
\begin{defi}[\cite{ee}]
	A Lebesgue measurable function $b:[1,\infty)\rightarrow(0,\infty)$ is said to be a slowly varying function if for any given $\varepsilon>0$, the function $t^\varepsilon b(t)$ is equivalent to a nondecreasing function and the function $t^{-\varepsilon}b(t)$ is equivalent to a nonincreasing function on $[1,\infty)$.
\end{defi}
	
\begin{example}
Clearly, $\eta(t)\equiv1$ and $1+\log t$ are slowly varying functions, respectively $($see \cite{n}$)$. Let $0<p<\infty$, $m \in \mathbb{N}$ and $\alpha=(\alpha_1,\alpha_2,\ldots,\alpha_m) \in \mathbb{R}^m$.
Define the family of positive functions $\{\ell_k\}^m_{k=0}$ on $(0,\infty)$ by
$$\ell_0(t) = 1/t\ \ \text{and}\ \ \ell_k(t)=1+\log\big(\ell_{k-1}(t)\big),\quad 0<t\leq1,\;1\leq k\leq m.$$
Moreover, define
$$\Theta_\alpha^m(t)=\prod_{k=1}^m\ell_k^{\alpha_k}(t).$$
It is easy to see that $\Theta_\alpha^m$ is a slowly varying function $($see \cite{ee}$)$.
Moreover, it follows from \cite{ekp} that $(e+\log t)^\alpha(\log(e+\log t))^\beta\ (\alpha,\beta\in\mathbb{R})$ and $\exp(\sqrt{\log t})$ are also slowly varying functions.
\end{example}

Let $b$ be a slowly varying function. We define $\gamma_b$ on $(0,\infty)$ by $$\gamma_b(t)=b\big(\max\{t,t^{-1}\}\big),\quad t>0.$$
This definition is from \cite{ee}. The following proposition shows some properties of the slowly varying functions. We refer to \cite{ee,n,sd} for more information of slowly varying functions.
\begin{pro}\label{b}
	Let $b$ be a slowly varying function on $[1,\infty)$.
		
	$(1)$ If $b$ is a nondecreasing function, then $\gamma_b$ is nonincreasing on $(0,1]$.
		
	$(2)$ For any given $\varepsilon>0$, the function $t^\varepsilon \gamma_b(t)$ is equivalent to a nondecreasing function and the function $t^{-\varepsilon}\gamma_b(t)$ is equivalent to a nonincreasing function on $(0,\infty)$.
		
	$(3)$ If $\varepsilon$ and $r$ are positive numbers, then there exists positive constants $c_\varepsilon$ and $C_{\varepsilon}$ such that $$c_\varepsilon\min\{r^{\varepsilon},r^{-\varepsilon}\}b(t)\le b(rt)\le C_\varepsilon \max\{r^{\varepsilon},r^{-\varepsilon}\}b(t),\quad t>0.$$
		
	$(4)$ For any $a>0$, denote $b_1(t)=b(t^a)$ on $[1,\infty)$. Then $b_1$ also is a slowly varying function.
	
	$(5)$ For any given $r\in\mathbb{R}$, the function $b^r$ is a slowly varying function and $\gamma_{b^r}=\gamma_b^r$.
	
	$(6)$ Let $0<p\le\infty$. For any positive constants $\alpha$ and $\beta$, we have
	\begin{align*}
		 (\alpha+\beta)^p\gamma_b(\alpha+\beta)\lesssim\alpha^p\gamma_b(\alpha)+\beta^p\gamma_b(\beta).
	\end{align*}
\end{pro}

We refer to \cite{jxz,sd} for the proof of $(1);$ $(2)$ and $(3)$ come from \cite{n}$;$ $(4)$ and $(5)$ were showed in \cite{ee}$;$ $(6)$ can be found in \cite{hoa}.

Furthermore, we extend Proposition \ref{b} (6) to more general situations.

\begin{lem}\label{wuqiongheb}
	Let $b$ be a slowly varying function and $\alpha_i\ (i\in\mathbb{N})$ be positive constants. If $1<p<\infty$, then
	 $$\sum_{i\in\mathbb{N}}\alpha_i^p\gamma_b(\alpha_i)\lesssim\Big(\sum_{i\in\mathbb{N}}\alpha_i\Big)^p\gamma_b\Big(\sum_{i\in\mathbb{N}}\alpha_i\Big).$$
	 If $0<p<1$, then
	 $$\Big(\sum_{i\in\mathbb{N}}\alpha_i\Big)^p\gamma_b\Big(\sum_{i\in\mathbb{N}}\alpha_i\Big)\lesssim\sum_{i\in\mathbb{N}}\alpha_i^p\gamma_b(\alpha_i).$$
\end{lem}	
\begin{proof}
	We first consider the case of $1<p<\infty$. Let $0<\varepsilon<p-1$. Then $p-\varepsilon>1$ and we know that
	\begin{align}\label{qq}
		 \sum_{i\in\mathbb{N}}\alpha_i^{p-\varepsilon}\le\Big(\sum_{i\in\mathbb{N}}\alpha_i\Big)^{p-\varepsilon}.
	\end{align}
	Since $t^{\varepsilon}\gamma_b(t)$ is equivalent to a nondecreasing function, there is
	 $$\alpha_k^{p-\varepsilon}\Big(\sum_{i\in\mathbb{N}}\alpha_i\Big)^{\varepsilon}\gamma_b\Big(\sum_{i\in\mathbb{N}}\alpha_i\Big)\gtrsim \alpha_k^p\gamma_b(\alpha_k),\quad\forall\ k\in\mathbb{N}.$$
	Combining (\ref{qq}) with the above inequality, we have
	\begin{align*}
		\sum_{k\in\mathbb{N}}\alpha_k^{p}\gamma_b(\alpha_k)&\ \lesssim\sum_{k\in\mathbb{N}}\alpha_k^{p-\varepsilon}\Big(\sum_{i\in\mathbb{N}}\alpha_i\Big)^\varepsilon\gamma_b\Big(\sum_{i\in\mathbb{N}}\alpha_i\Big)\\
		&\ \le\Big(\sum_{i\in\mathbb{N}}\alpha_i\Big)^{p-\varepsilon}\Big(\sum_{i\in\mathbb{N}}\alpha_i\Big)^\varepsilon\gamma_b\Big(\sum_{i\in\mathbb{N}}\alpha_i\Big).
	\end{align*}

    If $0<p<1$, we set $0<\theta<1-p$. Then we obtain that $\theta+p<1$ and
    $$\Big(\sum_{i\in\mathbb{N}}\alpha_i\Big)^{\theta+p}\le\sum_{i\in\mathbb{N}}\alpha_i^{\theta+p}.$$
    Since $t^{-\theta}\gamma_b(t)$ is equivalent to a nonincreasing function, there is
    $$\alpha_k^{\theta+p}\Big(\sum_{i\in\mathbb{N}}\alpha_i\Big)^{-\theta}\gamma_b\Big(\sum_{i\in\mathbb{N}}\alpha_i\Big)\lesssim \alpha_k^p\gamma_b(\alpha_k),\quad \forall\ k\in\mathbb{N}.$$
    This yields that
    $$\Big(\sum_{k\in\mathbb{N}}\alpha_k^{\theta+p}\Big)\Big(\sum_{i\in\mathbb{N}}\alpha_i\Big)^{-\theta}\gamma_b\Big(\sum_{i\in\mathbb{N}}\alpha_i\Big)\lesssim \sum_{k\in\mathbb{N}}\alpha_k^p\gamma_b(\alpha_k).$$
    Then we conclude that
    \begin{align*}
    	 \Big(\sum_{i\in\mathbb{N}}\alpha_i\Big)^p\gamma_b\Big(\sum_{i\in\mathbb{N}}\alpha_i\Big)=&\ \Big(\sum_{k\in\mathbb{N}}\alpha_k\Big)^{\theta+p}\Big(\sum_{i\in\mathbb{N}}\alpha_i\Big)^{-\theta}\gamma_b\Big(\sum_{i\in\mathbb{N}}\alpha_i\Big)\\
    	\le&\ \Big(\sum_{k\in\mathbb{N}}\alpha_k^{\theta+p}\Big)\Big(\sum_{i\in\mathbb{N}}\alpha_i\Big)^{-\theta}\gamma_b\Big(\sum_{i\in\mathbb{N}}\alpha_i\Big)\\
    	\lesssim&\ \sum_{k\in\mathbb{N}}\alpha_k^p\gamma_b(\alpha_k).
    \end{align*}
\end{proof}

\begin{lem}\label{wuqiongbianb}
	Let $p(\cdot)\in\mathcal{P}(\Omega)$ and $b$ be a slowly varying function. If $0<\theta<\underline{p}$, we have
	\begin{align*}
		 \bigg\|\sum_{i\in\mathbb{N}}\chi_{A_i}\bigg\|_{p(\cdot)}^\theta\gamma_b^\theta\bigg(\bigg\|\sum_{i\in\mathbb{N}}\chi_{A_i}\bigg\|_{p(\cdot)}\bigg)\lesssim\sum_{i\in\mathbb{N}}\|\chi_{A_i}\|_{p(\cdot)}^\theta\gamma_b^\theta\big(\|\chi_{A_i}\|_{p(\cdot)}\big),
	\end{align*}
    where $A_i\in\mathcal{F}$, $i\in\mathbb{N}$ are arbitrary sets.
    If $p_+<\theta<\infty$ and $(A_i)_{i\in\mathbb{N}}$ are disjoint, then we have
	\begin{align}\label{///}
		 \sum_{i\in\mathbb{N}}\|\chi_{A_i}\|_{p(\cdot)}^\theta\gamma_b^\theta\big(\|\chi_{A_i}\|_{p(\cdot)}\big)\lesssim\bigg\|\sum_{i\in\mathbb{N}}\chi_{A_i}\bigg\|_{p(\cdot)}^\theta\gamma_b^\theta\bigg(\bigg\|\sum_{i\in\mathbb{N}}\chi_{A_i}\bigg\|_{p(\cdot)}\bigg).
	\end{align}
    Moreover, if $b$ is nonincreasing and $p_+\le\theta<\infty$, then $(\ref{///})$ also holds.
\end{lem}
\begin{proof}
	For the case of $0<\theta<\underline{p}$, set $\frac{\theta}{\underline{p}}\le\alpha<1$ and $b_1(t)=b(t^\frac{\alpha}{\theta})$ for $t\in[1,\infty)$.
	From Proposition \ref{b} (4), (5) and Lemma \ref{lem2.1}, we know that $b_1^\theta$ is a slowly varying function and
	 $$\gamma_b^\theta\big(\|\chi_{A_i}\|_{p(\cdot)}\big)=\gamma_{b^\theta}\big(\|\chi_{A_i}\|_{p(\cdot)}\big)=\gamma_{b^\theta}\Big(\|\chi_{A_i}\|_{\frac{p(\cdot)\alpha}{\theta}}^\frac{\alpha}{\theta}\Big)=\gamma_{b_1^\theta}\big(\|\chi_{A_i}\|_{\frac{p(\cdot)\alpha}{\theta}}\big).$$
	Then according to Lemma \ref{wuqiongheb}, we get that
	\begin{align*}
		 \sum_{i\in\mathbb{N}}\|\chi_{A_i}\|_{p(\cdot)}^\theta\gamma_b^\theta\big(\|\chi_{A_i}\|_{p(\cdot)}\big)=&\ \sum_{i\in\mathbb{N}}\|\chi_{A_i}\|_{\frac{p(\cdot)\alpha}{\theta}}^\alpha\gamma_{b_1^\theta}\big(\|\chi_{A_i}\|_{\frac{p(\cdot)\alpha}{\theta}}\big)\\
		\gtrsim&\ \bigg(\sum_{i\in\mathbb{N}}\|\chi_{A_i}\|_{\frac{p(\cdot)\alpha}{\theta}}\bigg)^\alpha\gamma_{b_1^\theta}\bigg(\sum_{i\in\mathbb{N}}\|\chi_{A_i}\|_{\frac{p(\cdot)\alpha}{\theta}}\bigg)\\
		=&\ \bigg(\sum_{i\in\mathbb{N}}\|\chi_{A_i}\|_{p(\cdot)}^\frac{\theta}{\alpha}\bigg)^\alpha\gamma_{b_1^\theta}\bigg(\sum_{i\in\mathbb{N}}\|\chi_{A_i}\|_{p(\cdot)}^\frac{\theta}{\alpha}\bigg).
	\end{align*}
	Since $0<\frac{\theta}{\alpha}\le\underline{p}$, we see that $\|\cdot\|_{p(\cdot)}$ is a $\frac{\theta}{\alpha}$-norm and
	 $$\bigg\|\sum_{i\in\mathbb{N}}\chi_{A_i}\bigg\|_{p(\cdot)}^\frac{\theta}{\alpha}\le\sum_{i\in\mathbb{N}}\|\chi_{A_i}\|_{p(\cdot)}^\frac{\theta}{\alpha}.$$
	Then
	\begin{align*}
		 \sum_{i\in\mathbb{N}}\|\chi_{A_i}\|_{p(\cdot)}^\theta\gamma_b^\theta\big(\|\chi_{A_i}\|_{p(\cdot)}\big)\gtrsim&\ \bigg\|\sum_{i\in\mathbb{N}}\chi_{A_i}\bigg\|_{p(\cdot)}^\theta\gamma_{b_1^\theta}\bigg(\bigg\|\sum_{i\in\mathbb{N}}\chi_{A_i}\bigg\|_{p(\cdot)}^\frac{\theta}{\alpha}\bigg)\\
		=&\ \bigg\|\sum_{i\in\mathbb{N}}\chi_{A_i}\bigg\|_{p(\cdot)}^\theta\gamma_b^\theta\bigg(\bigg\|\sum_{i\in\mathbb{N}}\chi_{A_i}\bigg\|_{p(\cdot)}\bigg).
	\end{align*}

    Next we discuss the situation of $p_+<\theta<\infty$. Let $1<\beta\le\frac{\theta}{p_+}$ and $b_2(t)=b(t^\frac{\beta}{\theta})$ for $t\in[1,\infty)$.
    Obviously, $b_2^\theta$ is a slowly varying function and
    $\gamma_b^\theta\big(\|\chi_{A_i}\|_{p(\cdot)}\big)=\gamma_{b_2^\theta}\big(\|\chi_{A_i}\|_{\frac{p(\cdot)\beta}{\theta}}\big).$
    Then it follows from Lemma \ref{wuqiongheb} that
    \begin{align*}
    	 \sum_{i\in\mathbb{N}}\|\chi_{A_i}\|_{p(\cdot)}^\theta\gamma_b^\theta\big(\|\chi_{A_i}\|_{p(\cdot)}\big)=&\ \sum_{i\in\mathbb{N}}\|\chi_{A_i}\|_{\frac{p(\cdot)\beta}{\theta}}^\beta\gamma_{b_2^\theta}\big(\|\chi_{A_i}\|_{\frac{p(\cdot)\beta}{\theta}}\big)\\
    	\lesssim&\ \bigg(\sum_{i\in\mathbb{N}}\|\chi_{A_i}\|_{\frac{p(\cdot)\beta}{\theta}}\bigg)^\beta\gamma_{b_2^\theta}\bigg(\sum_{i\in\mathbb{N}}\|\chi_{A_i}\|_{\frac{p(\cdot)\beta}{\theta}}\bigg).
    \end{align*}
    Since $\big(\frac{p(\cdot)\beta}{\theta}\big)_+\le1$, then we have the following estimation by Lemma \ref{fansanjiao}:
    $$\sum_{i\in\mathbb{N}}\|\chi_{A_i}\|_{\frac{p(\cdot)\beta}{\theta}}\le\bigg\|\sum_{i\in\mathbb{N}}\chi_{A_i}\bigg\|_{\frac{p(\cdot)\beta}{\theta}}.$$
    Hence, we get by the disjointness of $(A_i)_{i\in\mathbb{N}}$ that
    \begin{align*}
    	 \sum_{i\in\mathbb{N}}\|\chi_{A_i}\|_{p(\cdot)}^\theta\gamma_b^\theta\big(\|\chi_{A_i}\|_{p(\cdot)}\big)\lesssim&\ \bigg\|\sum_{i\in\mathbb{N}}\chi_{A_i}\bigg\|_{\frac{p(\cdot)\beta}{\theta}}^\beta\gamma_{b_2^\theta}\bigg(\bigg\|\sum_{i\in\mathbb{N}}\chi_{A_i}\bigg\|_{\frac{p(\cdot)\beta}{\theta}}\bigg)\\
    	=&\ \bigg\|\sum_{i\in\mathbb{N}}\chi_{A_i}\bigg\|_{p(\cdot)}^\theta\gamma_b^\theta\bigg(\bigg\|\sum_{i\in\mathbb{N}}\chi_{A_i}\bigg\|_{p(\cdot)}\bigg).
    \end{align*}

    If $b$ is nonincreasing and $p_+\le\theta<\infty$, it follows from the definition of $\gamma_{b}$ and Lemma \ref{fansanjiao} that
    \begin{align*}
    	 \sum_{i\in\mathbb{N}}\|\chi_{A_i}\|_{p(\cdot)}^\theta\gamma_b^\theta\big(\|\chi_{A_i}\|_{p(\cdot)}\big)\lesssim&\ \sum_{i\in\mathbb{N}}\|\chi_{A_i}\|_{p(\cdot)}^\theta\gamma_b^\theta\bigg(\bigg\|\sum_{i\in\mathbb{N}}\chi_{A_i}\bigg\|_{p(\cdot)}\bigg)\\
    	=&\ \sum_{i\in\mathbb{N}}\|\chi_{A_i}\|_{\frac{p(\cdot)}{\theta}}\gamma_b^\theta\bigg(\bigg\|\sum_{i\in\mathbb{N}}\chi_{A_i}\bigg\|_{p(\cdot)}\bigg)\\
    	\le&\ \bigg\|\sum_{i\in\mathbb{N}}\chi_{A_i}\bigg\|_{\frac{p(\cdot)}{\theta}}\gamma_b^\theta\bigg(\bigg\|\sum_{i\in\mathbb{N}}\chi_{A_i}\bigg\|_{p(\cdot)}\bigg)\\
    	=&\ \bigg\|\sum_{i\in\mathbb{N}}\chi_{A_i}\bigg\|_{p(\cdot)}^\theta\gamma_b^\theta\bigg(\bigg\|\sum_{i\in\mathbb{N}}\chi_{A_i}\bigg\|_{p(\cdot)}\bigg).
    \end{align*}
    Hence, this completes the proof.
\end{proof}

Motivated by Lemma \ref{wuqiongbianb}, we have the following lemma as another deduction of Proposition \ref{b} (6).

\begin{lem}\label{variableb}
	Let $p(\cdot)\in\mathcal{P}(\Omega)$ and $b$ be a slowly varying function.
	For any sets $A,B\in\mathcal{F}$, there is
	$$\|\chi_{A\cup B}\|_{p(\cdot)}\gamma_b\big(\|\chi_{A\cup B}\|_{p(\cdot)}\big)\lesssim\|\chi_A\|_{p(\cdot)}\gamma_b\big(\|\chi_A\|_{p(\cdot)}\big)+\|\chi_B\|_{p(\cdot)}\gamma_b\big(\|\chi_B\|_{p(\cdot)}\big).$$
\end{lem}
\begin{proof}
	According to Remark \ref{property}, we have
	$$\|\chi_{A\cup B}\|_{p(\cdot)}\le\|\chi_{A}+\chi_B\|_{p(\cdot)}\lesssim\|\chi_{A}\|_{p(\cdot)}+\|\chi_{B}\|_{p(\cdot)}.$$
	Since $t\gamma_b(t)$ is equivalent to a nondecreasing function, it follows from Proposition \ref{b} (6) that
	\begin{align*}
		\|\chi_{A\cup B}\|_{p(\cdot)}\gamma_b\big(\|\chi_{A\cup B}\|_{p(\cdot)}\big) \lesssim&\ \big(\|\chi_A\|_{p(\cdot)}+\|\chi_B\|_{p(\cdot)}\big)\cdot\gamma_b\big(\|\chi_A\|_{p(\cdot)}+\|\chi_B\|_{p(\cdot)}\big)\\
		\lesssim&\ \|\chi_A\|_{p(\cdot)}\gamma_b\big(\|\chi_A\|_{p(\cdot)}\big)+\|\chi_B\|_{p(\cdot)}\gamma_b\big(\|\chi_B\|_{p(\cdot)}\big).
	\end{align*}
	Therefore, the proof is finished.
\end{proof}

\subsection{Variable Lorentz-Karamata spaces}
In this subsection, we introduce a new class of function spaces as a generalization of Lorentz-Karamata spaces and variable Lorentz spaces. 
\begin{defi}\label{dingyi}
	Let $p(\cdot)\in\mathcal{P}(\Omega)$, $0<q\le\infty$ and $b$ be a slowly varying function. The variable Lorentz-Karamata space $L_{p(\cdot),q,b}:=L_{p(\cdot),q,b}(\Omega)$ consists of all measurable functions $f$ with a finite functional $\|f\|_{p(\cdot),q,b}$ given by
	$$\|f\|_{p(\cdot),q,b}=\left\{
	\begin{aligned}
		 &\bigg[\int_{0}^{\infty}\big(t\|\chi_{\{|f|>t\}}\|_{p(\cdot)}\gamma_b(\|\chi_{\{|f|>t\}}\|_{p(\cdot)})\big)^q\frac{dt}{t}\bigg]^{\frac{1}{q}},\ \mathrm{if}\ 0<q<\infty,\\
		 &\sup_{t>0}t\|\chi_{\{|f|>t\}}\|_{p(\cdot)}\gamma_b(\|\chi_{\{|f|>t\}}\|_{p(\cdot)}),\ \ \ \ \ \ \ \ \ \ \ \ \ \mathrm{if}\ q=\infty.
	\end{aligned}
	\right.$$
\end{defi}


	
	


\begin{rem}
	These spaces coincide with the Lorentz-Karamata spaces $L_{p,q,b_1}$ when $p(\cdot)\equiv p$ is a positive constant and $b_1(t)=b(t^{1/p})$ for $t\in[1,\infty)$. In fact, $b_1$  is a slowly varying function by Proposition $\ref{b}$ $(4)$. Moreover, when $p(\cdot)\equiv p$, for any $f\in L_{p(\cdot),q,b}$, we have
	\begin{align*}
		 \int_{0}^{\infty}\big(t\|\chi_{\{|f|>t\}}\|_{p}\gamma_b(\|\chi_{\{|f|>t\}}\|_{p})\big)^q\frac{dt}{t}
		=&\ \int_{0}^{\infty}\big[t\mathbb{P}(|f|>t)^{1/p}\gamma_b\big(\mathbb{P}(|f|>t)^{1/p}\big)\big]^q\frac{dt}{t}\\
		=&\ \int_{0}^{\infty}\big[t\|\chi_{\{|f|>t\}}\|_{p}\gamma_{b_1}\big(\mathbb{P}(|f|>t)\big)\big]^q\frac{dt}{t},
	\end{align*}
	which means that $\|f\|_{p(\cdot),q,b}=\|f\|_{p,q,b_1}$ for $0<q<\infty$. Similarly, it is obvious that $\|f\|_{p(\cdot),\infty,b}=\|f\|_{p,\infty,b_1}$.
	For more details of Lorentz-Karamata spaces, we refer the readers to \cite{ee,hol,hoa,n} and the references therein. Moreover, variable Lorentz-Karamata spaces become variable Lorentz spaces when $b\equiv1$. For an introduction to variable Lorentz spaces, we refer to \cite{cf,dhhr,jwzw,kv}.
\end{rem}

In order to discuss whether variable Lorentz-Karamata space is a (quasi)-normed space, we shall show the embedding relationship on variable Lorentz-Karamata spaces. This result extends the embedding relationship on Lorentz-Karamata spaces (see \cite[Theorem $3.4.45$]{ee}).

\begin{lem}\label{baohan}
	Suppose that $p(\cdot)\in\mathcal{P}(\Omega)$, $0<q_1\le q_2\le\infty$ and the slowly varying functions $b_1,b_2$ satisfy $\sup\limits_{1\le t<\infty}\frac{b_2(t)}{b_1(t)}<\infty$. Then $L_{p(\cdot),q_1,b_1}\subset L_{p(\cdot),q_2,b_2}$.
\end{lem}
\begin{proof}
	Since $\sup\limits_{1\le t<\infty}\frac{b_2(t)}{b_1(t)}<\infty$, there exists a constant $C>0$ such that $b_2(t)\le C b_1(t)$ for $1\le t<\infty$. Hence, we have
	\begin{align}\label{22}
	\gamma_{b_2}(t)=b_2(1/t)\lesssim b_1(1/t)=\gamma_{b_1}(t)
	\end{align}
	for $0<t\le1$. When $q_1=\infty$, it yields $q_2=\infty$.
	For $\alpha>0$, it is easy to see that
	 $$\alpha\|\chi_{\{|f|>\alpha\}}\|_{p(\cdot)}\gamma_{b_2}\big(\|\chi_{\{|f|>\alpha\}}\|_{p(\cdot)}\big)\lesssim \alpha\|\chi_{\{|f|>\alpha\}}\|_{p(\cdot)}\gamma_{b_1}\big(\|\chi_{\{|f|>\alpha\}}\|_{p(\cdot)}\big).$$
	Taking the supremum for all $\alpha>0$, we get $L_{p(\cdot),\infty,b_1}\subset L_{p(\cdot),\infty,b_2}$.
	
	Now we consider the case of $0<q_1<q_2=\infty$. Suppose that $f\in L_{p(\cdot),q_1,b_1}$ and $\alpha>0$. We have
	\begin{align*}
		 \Big[\alpha\|\chi_{\{|f|>\alpha\}}\|_{p(\cdot)}\gamma_{b_2}\big(\|\chi_{\{|f|>\alpha\}}\|_{p(\cdot)}\big)\Big]^{q_1}=&\ q_1\int_{0}^{\alpha}s^{q_1-1}ds\cdot\Big[\|\chi_{\{|f|>\alpha\}}\|_{p(\cdot)}\gamma_{b_2}
		\big(\|\chi_{\{|f|>\alpha\}}\|_{p(\cdot)}\big)\Big]^{q_1}\\
		\lesssim&\ \int_{0}^{\alpha}\Big[s\|\chi_{\{|f|>s\}}\|_{p(\cdot)}\gamma_{b_2}\big(\|\chi_{\{|f|>s\}}\|_{p(\cdot)}\big)\Big]^{q_1}\frac{ds}{s}\\
		\lesssim&\ \int_{0}^{\alpha}\Big[s\|\chi_{\{|f|>s\}}\|_{p(\cdot)}\gamma_{b_1}\big(\|\chi_{\{|f|>s\}}\|_{p(\cdot)}\big)\Big]^{q_1}\frac{ds}{s}.
	\end{align*}
	This means
	\begin{align*}
		\|f\|_{p(\cdot),\infty,b_2}=&\ \sup_{\alpha>0}\alpha\|\chi_{\{|f|>\alpha\}}\|_{p(\cdot)}\gamma_{b_2}\big(\|\chi_{\{|f|>\alpha\}}\|_{p(\cdot)}\big)\\
		\lesssim&\ \bigg(\int_{0}^{\infty}\Big[s\|\chi_{\{|f|>s\}}\|_{p(\cdot)}\gamma_{b_1}\big(\|\chi_{\{|f|>s\}}\|_{p(\cdot)}\big)\Big]^{q_1}\frac{ds}{s}\bigg)^{1/q_1}=\|f\|_{p(\cdot),q_1,b_1}.
	\end{align*}
	Hence, we get $L_{p(\cdot),q_1,b_1}\subset L_{p(\cdot),\infty,b_2}$.
	
	Finally, we suppose that $0<q_1\le q_2<\infty$. It follows from (\ref{22}) that
	\begin{align*}
		\|f\|_{p(\cdot),q_2,b_2}^{q_2}=&\ \int_{0}^{\infty}t^{q_2}\|\chi_{\{|f|>t\}}\|_{p(\cdot)}^{q_2}\gamma_{b_2}^{q_2}\big(\|\chi_{\{|f|>t\}}\|_{p(\cdot)}\big)\frac{dt}{t}\\
		=&\ \int_{0}^{\infty}t^{q_2-q_1}\|\chi_{\{|f|>t\}}\|_{p(\cdot)}^{q_2-q_1}\gamma_{b_2}^{q_2-q_1}\big(\|\chi_{\{|f|>t\}}\|_{p(\cdot)}\big)\\
		&\ \ \ \ \ \ \ \ \ \ \ \ \ \ \ \ \ \ \ \ \ \times t^{q_1}\|\chi_{\{|f|>t\}}\|_{p(\cdot)}^{q_1}\gamma_{b_2}^{q_1}\big(\|\chi_{\{|f|>t\}}\|_{p(\cdot)}\big)\frac{dt}{t}\\
		\lesssim&\ \int_{0}^{\infty}t^{q_2-q_1}\|\chi_{\{|f|>t\}}\|_{p(\cdot)}^{q_2-q_1}\gamma_{b_2}^{q_2-q_1}\big(\|\chi_{\{|f|>t\}}\|_{p(\cdot)}\big)\\
		&\ \ \ \ \ \ \ \ \ \ \ \ \ \  \ \  \ \ \ \ \ \times t^{q_1}\|\chi_{\{|f|>t\}}\|_{p(\cdot)}^{q_1}\gamma_{b_1}^{q_1}\big(\|\chi_{\{|f|>t\}}\|_{p(\cdot)}\big)\frac{dt}{t}\\
		\le&\ \|f\|_{p(\cdot),\infty,b_2}^{q_2-q_1}\cdot\|f\|_{p(\cdot),q_1,b_1}^{q_1}\lesssim\|f\|_{p(\cdot),q_1,b_1}^{q_2}.
	\end{align*}
	Therefore, $L_{p(\cdot),q_1,b_1}\subset L_{p(\cdot),q_2,b_2}$ holds again.
\end{proof}

We also recall the next embedding relationship for Lorentz-Karamata spaces.
\begin{lem}[\cite{ee}]\label{3.4.48}
	Let $p_1,p_2\in(0,\infty),q_1,q_2\in(0,\infty]$ with $p_2<p_1$ and $b_1,b_2$ be slowly varying functions. Then $L_{p_1,q_1,b_1}\subset L_{p_2,q_2,b_2}$.
\end{lem}

Based on the above embedding relationship of variable Lorentz-Karamata spaces, we collect the following properties of the functional $\|\cdot\|_{p(\cdot),q,b}$.

\begin{lem}\label{fanshu}
	Let $p(\cdot)\in\mathcal{P}(\Omega)$, $0<q\le\infty$ and $b$ be a slowly varying function. The following properties hold.
	
	$(1)$ $\|f\|_{p(\cdot),q,b}\ge0$, $\|f\|_{p(\cdot),q,b}=0$ if and only if $f\equiv0$.
	
	$(2)$ $\|cf\|_{p(\cdot),q,b}=|c|\|f\|_{p(\cdot),q,b}$ for any $c\in\mathbb{C}$.
	
	$(3)$ For every fixed $\lambda>0$, the functional $\|\cdot\|_{p(\cdot),q,b}$ satisfies the equivalence
	$$\|f\|_{p(\cdot),q,b}\approx\left\{
	\begin{aligned}
		 &\bigg(\sum_{k\in\mathbb{Z}}\Big[2^k\|\chi_{\{|f|>\lambda2^k\}}\|_{p(\cdot)}\gamma_b\big(\|\chi_{\{|f|>\lambda2^k\}}\|_{p(\cdot)}\big)\Big]^q\bigg)^\frac{1}{q},\ \mathrm{if}\ 0<q<\infty,\\
		 &\sup_{k\in\mathbb{Z}}2^k\|\chi_{\{|f|>\lambda2^k\}}\|_{p(\cdot)}\gamma_b\big(\|\chi_{\{|f|>\lambda2^k\}}\|_{p(\cdot)}\big),\ \ \ \ \ \ \ \ \ \ \ \mathrm{if}\ q=\infty.
	\end{aligned}
	\right.$$
	
	$(4)$ For any set $A\in\mathcal{F}$ with $\mathbb{P}(A)>0$, we have
	$$\|\chi_A\|_{p(\cdot),q,b}=\left\{
	\begin{aligned}
		 &\Big(\frac{1}{q}\Big)^{1/q}\|\chi_A\|_{p(\cdot)}\gamma_b(\|\chi_A\|_{p(\cdot)}),\ \ \mathrm{if}\ 0<q<\infty,\\
		&\|\chi_A\|_{p(\cdot)}\gamma_b(\|\chi_A\|_{p(\cdot)}),\ \ \ \ \ \ \ \ \ \ \ \mathrm{if}\ q=\infty.
	\end{aligned}
	\right.$$
	
	$(5)$ Let $s>0$ and $b_1(t)=b(t^s)$ for $t\in[1,\infty)$. For any $f\in L_{sp(\cdot),sq,b_1^{1/s}}$, we have
	\begin{equation*}
		\||f|^s\|_{p(\cdot),q,b}=\left\{
		\begin{aligned}
			&s^{\frac{1}{q}}\|f\|_{sp(\cdot),sq,b_1^{1/s}}^s,\ \ \mathrm{if}\ 0<q<\infty,\\
			&\|f\|_{sp(\cdot),\infty,b_1^{1/s}}^s,\ \ \ \ \ \mathrm{if}\ q=\infty.
		\end{aligned}
		\right.
	\end{equation*}
\end{lem}
\begin{proof}
(1) This fact can be checked by direct calculation.
	
	(2) Fix $c\ne0$. For any $f\in L_{p(\cdot),q,b}$ with $0<q<\infty$, we have
	\begin{align*}
		\|cf\|_{p(\cdot),q,b}^q=&\ \int_{0}^{\infty}t^{q-1}\|\chi_{\{|cf|>t\}}\|_{p(\cdot)}^q\gamma_b^q(\|\chi_{\{|cf|>t\}}\|_{p(\cdot)})dt\\
		=&\ \int_{0}^{\infty}t^{q-1}\|\chi_{\{|f|>t/|c|\}}\|_{p(\cdot)}^q\gamma_b^q(\|\chi_{\{|f|>t/|c|\}}\|_{p(\cdot)})dt\\
		=&\ |c|^{q}\int_{0}^{\infty}\Big(\frac{t}{|c|}\Big)^{q-1}\|\chi_{\{|f|>t/|c|\}}\|_{p(\cdot)}^q\gamma_b^q(\|\chi_{\{|f|>t/|c|\}}\|_{p(\cdot)})d\Big(\frac{t}{|c|}\Big)\\
		=&\ |c|^q\|f\|_{p(\cdot),q,b}^q.
	\end{align*}
    For $f\in L_{p(\cdot),\infty,b}$, it is obvious that
    \begin{align*}
    	 t\|\chi_{\{|cf|>t\}}\|_{p(\cdot)}\gamma_b(\|\chi_{\{|cf|>t\}}\|_{p(\cdot)})
    	=&\ t\|\chi_{\{|f|>t/|c|\}}\|_{p(\cdot)}\gamma_b(\|\chi_{\{|f|>t/|c|\}}\|_{p(\cdot)})\\
    	=&\ |c|\frac{t}{|c|}\|\chi_{\{|f|>t/|c|\}}\|_{p(\cdot)}\gamma_b(\|\chi_{\{|f|>t/|c|\}}\|_{p(\cdot)}).
    \end{align*}
    Taking the supremum for all $t>0$, we have
    $$\|cf\|_{p(\cdot),\infty,b}=|c|\cdot\|f\|_{p(\cdot),\infty,b}.$$

    (3) When $0<q<\infty$, according to Proposition $\ref{b}$ $(2)$, we have
	\begin{align}\label{e11}
		\|f\|_{p(\cdot),q,b}^q=&\ \int_{0}^{\infty}\big(t\|\chi_{\{|f|>t\}}\|_{p(\cdot)}\gamma_b(\|\chi_{\{|f|>t\}}\|_{p(\cdot)})\big)^q\frac{dt}{t} \nonumber\\
		=&\ \lambda^{q}\int_{0}^{\infty}\big(s\|\chi_{\{|f|>\lambda s\}}\|_{p(\cdot)}\gamma_b(\|\chi_{\{|f|>\lambda s\}}\|_{p(\cdot)})\big)^q\frac{ds}{s} \nonumber\\
    	=&\ \lambda^{q}\sum_{k\in\mathbb{Z}}\int_{2^k}^{2^{k+1}}\big(s\|\chi_{\{|f|>\lambda s\}}\|_{p(\cdot)}\gamma_b(\|\chi_{\{|f|>\lambda s\}}\|_{p(\cdot)})\big)^q\frac{ds}{s}\\
	    \lesssim&\ \sum_{k\in\mathbb{Z}}\int_{2^k}^{2^{k+1}}\big(s\|\chi_{\{|f|>\lambda2^k\}}\|_{p(\cdot)}\gamma_b(\|\chi_{\{|f|>\lambda2^k\}}\|_{p(\cdot)})\big)^q\frac{ds}{s} \nonumber\\
	    =&\ \sum_{k\in\mathbb{Z}}\|\chi_{\{|f|>\lambda2^k\}}\|_{p(\cdot)}^q\gamma_b^q(\|\chi_{\{|f|>\lambda2^k\}}\|_{p(\cdot)})\int_{2^k}^{2^{k+1}}s^{q-1}ds \nonumber\\
		=&\ \frac{2^q-1}{q}\sum_{k\in\mathbb{Z}}2^{kq}\|\chi_{\{|f|>\lambda2^k\}}\|_{p(\cdot)}^q\gamma_b^q(\|\chi_{\{|f|>\lambda2^k\}}\|_{p(\cdot)}). \nonumber
	\end{align}
	On the other hand, by \eqref{e11}
	\begin{align*}
		\|f\|_{p(\cdot),q,b}^{q}
		=&\ \lambda^{q}\sum_{k\in\mathbb{Z}}\int_{2^{k-1}}^{2^{k}}\big(s\|\chi_{\{|f|>\lambda s\}}\|_{p(\cdot)}\gamma_b(\|\chi_{\{|f|>\lambda s\}}\|_{p(\cdot)})\big)^q\frac{ds}{s}\\
		\gtrsim&\ \sum_{k\in\mathbb{Z}}\int_{2^{k-1}}^{2^{k}}\big(s\|\chi_{\{|f|>\lambda2^k\}}\|_{p(\cdot)}\gamma_b(\|\chi_{\{|f|>\lambda2^k\}}\|_{p(\cdot)})\big)^q\frac{ds}{s}\\
		=&\ \sum_{k\in\mathbb{Z}}\|\chi_{\{|f|>\lambda2^k\}}\|_{p(\cdot)}^q\gamma_b^q(\|\chi_{\{|f|>\lambda2^k\}}\|_{p(\cdot)})\int_{2^{k-1}}^{2^{k}}s^{q-1}ds\\
		=&\ \frac{1}{q}\Big(1-\frac{1}{2^q}\Big)\sum_{k\in\mathbb{Z}}2^{kq}\|\chi_{\{|f|>\lambda2^k\}}\|_{p(\cdot)}^q\gamma_b^q(\|\chi_{\{|f|>\lambda2^k\}}\|_{p(\cdot)}).
   	\end{align*}
	Hence,
	 $$\|f\|_{p(\cdot),q,b}\approx\bigg(\sum_{k\in\mathbb{Z}}\Big[2^k\|\chi_{\{|f|>\lambda2^k\}}\|_{p(\cdot)}\gamma_b\big(\|\chi_{\{|f|>\lambda2^k\}}\|_{p(\cdot)}\big)\Big]^q\bigg)^\frac{1}{q},\ \ 0<q<\infty.$$
	
	Now, we consider the case of $q=\infty$. For any $t>0$ and $\lambda>0$, there exist $k\in\mathbb{Z}$ such that $\lambda2^k<t\le\lambda2^{k+1}$. Let $t= \lambda s$. It follows from Proposition $\ref{b}$ $(2)$ that
	\begin{align*}
		 t\|\chi_{\{|f|>t\}}\|_{p(\cdot)}\gamma_b\big(\|\chi_{\{|f|>t\}}\|_{p(\cdot)}\big)&=\lambda s\|\chi_{\{|f|>\lambda s\}}\|_{p(\cdot)}\gamma_b\big(\|\chi_{\{|f|>\lambda s\}}\|_{p(\cdot)}\big)\\
		&\lesssim s\|\chi_{\{|f|>\lambda2^k\}}\|_{p(\cdot)}\gamma_b\big(\|\chi_{\{|f|>\lambda2^k\}}\|_{p(\cdot)}\big)\\
		 &\le2^{k+1}\|\chi_{\{|f|>\lambda2^k\}}\|_{p(\cdot)}\gamma_b\big(\|\chi_{\{|f|>\lambda2^k\}}\|_{p(\cdot)}\big).
	\end{align*}
	On the other hand, for any $t>0$ and $\lambda>0$, there exist $k\in\mathbb{Z}$ such that $\lambda2^{k-1}\le t<\lambda2^k$. We have
	\begin{align*}
		 t\|\chi_{\{|f|>t\}}\|_{p(\cdot)}\gamma_b\big(\|\chi_{\{|f|>t\}}\|_{p(\cdot)}\big)&=\lambda s\|\chi_{\{|f|>\lambda s\}}\|_{p(\cdot)}\gamma_b\big(\|\chi_{\{|f|>\lambda s\}}\|_{p(\cdot)}\big)\\
		&\gtrsim s\|\chi_{\{|f|>\lambda2^k\}}\|_{p(\cdot)}\gamma_b\big(\|\chi_{\{|f|>\lambda2^k\}}\|_{p(\cdot)}\big)\\
		&\ge 2^{k-1}\|\chi_{\{|f|>\lambda2^k\}}\|_{p(\cdot)}\gamma_b\big(\|\chi_{\{|f|>\lambda2^k\}}\|_{p(\cdot)}\big).
	\end{align*}
	Hence,
	 $$\|f\|_{p(\cdot),\infty,b}\approx\sup_{k\in\mathbb{Z}}2^k\|\chi_{\{|f|>\lambda2^k\}}\|_{p(\cdot)}\gamma_b\big(\|\chi_{\{|f|>\lambda2^k\}}\|_{p(\cdot)}\big).$$
	
	(4) It is obvious that $\chi_{\{\chi_A>t\}}$ is $\chi_A$ if $0<t<1$ and it is $0$ if $t\ge1$. When $0<q<\infty$, we have
	\begin{align*}
		\|\chi_{A}\|_{p(\cdot),q,b}^q=&\ \int_{0}^{\infty}t^{q-1}\|\chi_{\{\chi_A>t\}}\|_{p(\cdot)}^q\gamma_{b}^q\big(\|\chi_{\{\chi_A>t\}}\|_{p(\cdot)}\big)dt\\
		=&\ \int_{0}^{1}t^{q-1}\|\chi_{\{\chi_A>t\}}\|_{p(\cdot)}^q\gamma_{b}^q\big(\|\chi_{\{\chi_A>t\}}\|_{p(\cdot)}\big)dt\\
		=&\ \int_{0}^{1}t^{q-1}\|\chi_{A}\|_{p(\cdot)}^q\gamma_{b}^q\big(\|\chi_{A}\|_{p(\cdot)}\big)dt\\
		=&\ \frac{1}{q}\|\chi_{A}\|_{p(\cdot)}^q\gamma_{b}^q\big(\|\chi_{A}\|_{p(\cdot)}\big).
	\end{align*}
    The statement can be proved similarly for $q=\infty$.
	
	(5) For any $s>0$ and $0<q<\infty$, it follows from Lemma $\ref{lem2.1}$ that
	\begin{align*}
		\||f|^s\|_{p(\cdot),q,b}^q=&\ \int_{0}^{\infty}\big(t\|\chi_{\{|f|^s>t\}}\|_{p(\cdot)}\gamma_b(\|\chi_{\{|f|^s>t\}}\|_{p(\cdot)})\big)^q\frac{dt}{t}\\
		=&\ \int_{0}^{\infty}t^q\|\chi_{\{|f|>t^{1/s}\}}\|_{p(\cdot)}^q\gamma_b^q(\|\chi_{\{|f|>t^{1/s}\}}\|_{p(\cdot)})\frac{dt}{t}\\
		=&\ \int_{0}^{\infty}t^q\|\chi_{\{|f|>t^{1/s}\}}\|_{sp(\cdot)}^{sq}\gamma_b^q(\|\chi_{\{|f|>t^{1/s}\}}\|_{sp(\cdot)}^s)\frac{dt}{t}\\
		=&\ \int_{0}^{\infty}t^q\|\chi_{\{|f|>t^{1/s}\}}\|_{sp(\cdot)}^{sq}\gamma_{b_1}^q(\|\chi_{\{|f|>t^{1/s}\}}\|_{sp(\cdot)})\frac{dt}{t},
	\end{align*}
    where $b_1(t)=b(t^s)$ for $t\in[1,\infty)$. Notice that $b_1$ is a slowly varying function by Proposition $\ref{b}$ $(4)$. Set $m=t^{1/s}$, by Proposition $\ref{b}$ $(5)$, there is
    \begin{align*}
    	\||f|^s\|_{p(\cdot),q,b}^q=&\ \int_{0}^{\infty}m^{sq}\|\chi_{\{|f|>m\}}\|_{sp(\cdot)}^{sq}\gamma_{b_1}^q(\|\chi_{\{|f|>m\}}\|_{sp(\cdot)})\frac{dm^{s}}{m^s}\\
    	=&\ s\int_{0}^{\infty}m^{sq}\|\chi_{\{|f|>m\}}\|_{sp(\cdot)}^{sq}\gamma_{b_1^{1/s}}^{sq}(\|\chi_{\{|f|>m\}}\|_{sp(\cdot)})\frac{dm}{m}\\
    	=&\ s\|f\|_{sp(\cdot),sq,b_1^{1/s}}^{sq}.
    \end{align*}
    Similarly, we obtain the result for $q=\infty$:
    \begin{align*}
    	\||f|^s\|_{p(\cdot),\infty,b}=&\ \sup_{t>0}t\|\chi_{\{|f|^s>t\}}\|_{p(\cdot)}\gamma_b(\|\chi_{\{|f|^s>t\}}\|_{p(\cdot)})\\
    	=&\ \sup_{t>0}t\|\chi_{\{|f|>t^{1/s}\}}\|_{p(\cdot)}\gamma_b(\|\chi_{\{|f|>t^{1/s}\}}\|_{p(\cdot)})\\
    	=&\ \sup_{m>0}m^s\|\chi_{\{|f|>m\}}\|_{p(\cdot)}\gamma_b(\|\chi_{\{|f|>m\}}\|_{p(\cdot)})\\
    	=&\ \sup_{m>0}m^s\|\chi_{\{|f|>m\}}\|_{sp(\cdot)}^s\gamma_{b^{1/s}}^s(\|\chi_{\{|f|>m\}}\|_{sp(\cdot)}^s)\\
    	=&\ \sup_{m>0}m^s\|\chi_{\{|f|>m\}}\|_{sp(\cdot)}^s\gamma_{b_1^{1/s}}^s(\|\chi_{\{|f|>m\}}\|_{sp(\cdot)})\\
    	=&\ \|f\|_{sp(\cdot),\infty,b_1^{1/s}}^s.
    \end{align*}
    The proof is complete now.
\end{proof}

By virtue of Lemma \ref{variableb}, we state the quasi-triangle inequality for the Lorentz-Karamata spaces with variable exponents.

\begin{lem}\label{fanshusanjiao}
	Let $p(\cdot)\in\mathcal{P}(\Omega)$, $0<q\le\infty$ and $b$ be a slowly varying function.
	For any $f,g\in L_{p(\cdot),q,b}$, there is $$\|f+g\|_{p(\cdot),q,b}\lesssim\|f\|_{p(\cdot),q,b}+\|g\|_{p(\cdot),q,b}.$$
\end{lem}
\begin{proof}
	Suppose that $f,g\in L_{p(\cdot),q,b}$. It is easy to see that
	$$\{|f+g|>t\}\subset\{|f|>t/2\}\cup\{|g|>t/2\},\quad \forall\ t>0.$$
	When $0<q<\infty$, it follows from Proposition \ref{b} (2) and Lemma $\ref{variableb}$ that
	\begin{align*}
		\|f+g\|_{p(\cdot),q,b}^q
		=&\ \int_{0}^{\infty}t^{q-1}\|\chi_{\{|f+g|>t\}}\|_{p(\cdot)}^q\gamma_b^q(\|\chi_{\{|f+g|>t\}}\|_{p(\cdot)})dt\\
		\lesssim&\ \int_{0}^{\infty}t^{q-1}\|\chi_{\{|f|>t/2\}\cup\{|g|>t/2\}}\|_{p(\cdot)}^q\gamma_b^q(\|\chi_{\{|f|>t/2\}\cup\{|g|>t/2\}}\|_{p(\cdot)})dt\\
		\lesssim&\ \int_{0}^{\infty}t^{q-1}\Big(\|\chi_{\{|f|>t/2\}}\|_{p(\cdot)}\gamma_b(\|\chi_{\{|f|>t/2\}}\|_{p(\cdot)})\\
		&\ \ \ \ \ \ \ \ \ \ \ \ \ \ \ \ \ \ \ \ \ \ \ \ \ \ \ \ \ \ +\|\chi_{\{|g|>t/2\}}\|_{p(\cdot)}\gamma_b(\|\chi_{\{|g|>t/2\}}\|_{p(\cdot)})\Big)^qdt\\
		\lesssim&\ \int_{0}^{\infty}t^{q-1}\Big(\|\chi_{\{|f|>t/2\}}\|_{p(\cdot)}^q\gamma_b^q(\|\chi_{\{|f|>t/2\}}\|_{p(\cdot)})\\
		&\ \ \ \ \ \ \ \ \ \ \ \ \ \ \ \ \ \ \ \ \ \ \ \ \ \ \ \ \ \ +\|\chi_{\{|g|>t/2\}}\|_{p(\cdot)}^q\gamma_b^q(\|\chi_{\{|g|>t/2\}}\|_{p(\cdot)})\Big)dt\\
		=&\ 2^q\Big(\|f\|_{p(\cdot),q,b}^q+\|g\|_{p(\cdot),q,b}^q\Big).
	\end{align*}
	When $q=\infty$, we obtain that for $t>0$,
	\begin{align*}
		&\ t\|\chi_{\{|f+g|>t\}}\|_{p(\cdot)}\gamma_b(\|\chi_{\{|f+g|>t\}}\|_{p(\cdot)})\\
		\lesssim&\ t\|\chi_{\{|f|>t/2\}\cup\{|g|>t/2\}}\|_{p(\cdot)}\gamma_b(\|\chi_{\{|f|>t/2\}\cup\{|g|>t/2\}}\|_{p(\cdot)})\\
		\lesssim&\ t\Big(\|\chi_{\{|f|>t/2\}}\|_{p(\cdot)}\gamma_b(\|\chi_{\{|f|>t/2\}}\|_{p(\cdot)})+\|\chi_{\{|g|>t/2\}}\|_{p(\cdot)}\gamma_b(\|\chi_{\{|g|>t/2\}}\|_{p(\cdot)})\Big).
	\end{align*}
	By taking the supremum for all $t>0$, we have
	 $$\|f+g\|_{p(\cdot),\infty,b}\lesssim\|f\|_{p(\cdot),\infty,b}+\|f\|_{p(\cdot),\infty,b}.$$
	The proof is finished.
\end{proof}

By Lemma $\ref{fanshu}\ (1)$, $(2)$ and Lemma $\ref{fanshusanjiao}$, we obtain that the variable Lorentz-Karamata spaces are quasi-normed spaces. Next, we will show that these spaces are also quasi-Banach spaces. Before that, we state the Aoki-Rolewicz theorem, which is proved in \cite{a}.

\begin{lem}[\cite{a}] \label{a-r}
	Let $(X,\|\cdot\|_X)$ be a quasi-normed space, that is, for any $x,y\in X$,
	$$\|x+y\|\le K(\|x\|+\|y\|).$$
	Then for $\alpha\in(0,1]$ defined by $(2K)^\alpha=2$,
	we have
	 $$\|x_1+\cdots+x_n\|_X^\alpha\le4(\|x_1\|_X^\alpha+\cdots+\|x_n\|_X^\alpha)$$
	for all $x_1,\cdots,x_n$ $(n\ge1)$ in $X$.
\end{lem}

\begin{rem}\label{art}
	It follows from Lemmas $\ref{fanshusanjiao}$ and $\ref{a-r}$ that there exists $\alpha\in(0,1]$ such that
	 $$\|f_1+\cdots+f_n\|_{p(\cdot),q,b}^{\alpha}\le4(\|f_1\|_{p(\cdot),q,b}^{\alpha} +\cdots+\|f_n\|_{p(\cdot),q,b}^{\alpha})$$
	for all $f_1,\cdots,f_n$ $(n\ge1)$ in $L_{p(\cdot),q,b}$.
 We denote by $\mathcal{N}\in(0,1]$ the supremum of the $\alpha$'s for which the above inequality holds.
\end{rem}

\begin{lem}
	Let $p(\cdot)\in\mathcal{P}(\Omega)$, $0<q\le\infty$ and $b$ be a slowly varying function. Then $L_{p(\cdot),q,b}$ is a quasi-Banach space.
\end{lem}
\begin{proof}
	Let $(h_n)_{n\geq 1}\subset L_{p(\cdot),q,b}$ be a Cauchy sequence with respect to $\|\cdot\|_{p(\cdot),q,b}$.
	We choose a subsequence $(h_{n_k})_{k\geq 1}$ such that
	$$\|h_{n_{k+1}}-h_{n_k}\|_{p(\cdot),q,b}\le\frac{1}{2^{2k}}, \quad k\geq 1
	$$
	and $h_{n_0}:=0$. Define the function
	 $$g(\omega):=\sum_{k=0}^{\infty}|h_{n_{k+1}}(\omega)-h_{n_k}(\omega)|.$$
	It is obvious that $$\chi_{\{g>\lambda\}}\le\sum_{k=0}^{\infty}\chi_{\{|h_{n_{k+1}}-h_{n_k}|>\frac{\lambda}{2^{k+1}}\}}.$$
	Let $0<\beta<p_-$ and $0<\varepsilon\le\min\{p_--\beta,1\}$. Then
	\begin{align*}
		\|\chi_{\{g>\lambda\}}\|_{p_--\beta}^\varepsilon\le&\ \bigg\|\sum_{k=0}^{\infty}\chi_{\{|h_{n_{k+1}}-h_{n_k}|>\frac{\lambda}{2^{k+1}}\}}\bigg\|_{p_--\beta}^\varepsilon\\
		\le&\ \sum_{k=0}^{\infty}\big\|\chi_{\{|h_{n_{k+1}}-h_{n_k}|>\frac{\lambda}{2^{k+1}}\}}\big\|_{p_--\beta}^\varepsilon\\
		\le&\ \sum_{k=0}^{\infty}\frac{2^{(k+1)\varepsilon}}{\lambda^\varepsilon}\|h_{n_{k+1}}-h_{n_k}\|_{p_--\beta}^\varepsilon.
	\end{align*}
It follows from  Lemma \ref{3.4.48} that
\begin{align*}
	\|f\|_{p_-- \beta} = \|f\|_{p_-- \beta, p_-- \beta,1} \lesssim \|f\|_{p_-, q,b}
\end{align*}
for any measurable function $f$, which implies
	\begin{align*}
		\|\chi_{\{g>\lambda\}}\|_{p_--\beta}^\varepsilon
		\lesssim&\ \sum_{k=0}^{\infty}\frac{2^{(k+1)\varepsilon}}{\lambda^\varepsilon}\|h_{n_{k+1}}-h_{n_k}\|_{p_-,q,b}^\varepsilon\\
		\lesssim&\ \sum_{k=0}^{\infty}\frac{2^{(k+1)\varepsilon}}{\lambda^\varepsilon}\|h_{n_{k+1}}-h_{n_k}\|_{p(\cdot),q,b}^\varepsilon\le\sum_{k=0}^{\infty}\frac{2^{(k+1)\varepsilon}}{\lambda^\varepsilon2^{2k}}.
	\end{align*}
    Hence, $\|\chi_{\{g>\lambda\}}\|_{p_--\beta}\rightarrow0$ as $\lambda\rightarrow\infty$ and $g$ is finite almost everywhere. Set
    $$h(\omega):=\sum_{k=0}^{\infty}\big(h_{n_{k+1}}(\omega)-h_{n_k}(\omega)\big)\ \ \text{and}\ \ \widetilde{h}(\omega):=h(\omega)-h_{n_1}(\omega),\quad \omega\in\Omega.$$
    Then $h$ and $\widetilde{h}$ converge almost everywhere.

    Next we shall verify that $h\in L_{p(\cdot),q,b}$ and $(h_n)_{n\geq 1}$ converges to $h$ in $L_{p(\cdot),q,b}$.
    Let $\mathcal{N}$ be the constant in Remark \ref{art}. We have
    \begin{align*}
    	 \left\|\widetilde{h} \right\|_{p(\cdot),q,b}^{\mathcal{N}}=\bigg\|\sum_{k=1}^{\infty}h_{n_{k+1}}-h_{n_k}\bigg\|_{p(\cdot),q,b}^{\mathcal{N}}
    \le4\sum_{k=1}^{\infty}\|h_{n_{k+1}}-h_{n_k}\|_{p(\cdot),q,b}^{\mathcal{N}}\le4\sum_{k=1}^{\infty}\frac{1}{2^{2k\mathcal{N}}}<\infty,
    \end{align*}
    which means that $\widetilde{h}\in L_{p(\cdot),q,b}$. Combining this with Lemma \ref{fanshusanjiao}, we have $h=\widetilde{h}+ h_{n_1}\in L_{p(\cdot),q,b}$. Since $(h_n)_{n\geq 1}\subset L_{p(\cdot),q,b}$ is a Cauchy sequence, for any $\varepsilon>0$, there exists $N$ such that for any $n,m>N$,
    $$\|h_n-h_m\|_{p(\cdot),q,b}<\varepsilon.$$
    Clearly, there exists $k_0\geq 1$ such that $n_{k_0}>N$ and $2^{-2k_0}<(1-2^{-2\mathcal{N}})^{1/\mathcal{N}}\cdot \varepsilon$.
    It follows from Lemma \ref{fanshusanjiao} and Remark $\ref{art}$ that, for any $n> N$,
    \begin{align*}
    	\|h_{n}-h\|_{p(\cdot),q,b}
    	\lesssim&\ \bigg\|h_n-\sum_{k=0}^{k_0-1}(h_{n_{k+1}}-h_{n_k})\bigg\|_{p(\cdot),q,b}+\bigg\|\sum_{k=k_0}^{\infty}(h_{n_{k+1}}-h_{n_k})\bigg\|_{p(\cdot),q,b}\\
    	=&\ \|h_n-h_{n_{k_0}}\|_{p(\cdot),q,b}+\bigg\|\sum_{k=k_0}^{\infty}(h_{n_{k+1}}-h_{n_k})\bigg\|_{p(\cdot),q,b}^{\mathcal{N}\frac{1}{\mathcal{N}}}\\
    	\lesssim&\ \varepsilon+\bigg(\sum_{k=k_0}^{\infty}\|(h_{n_{k+1}}-h_{n_k})\|_{p(\cdot),q,b}^{\mathcal{N}}\bigg)^{\frac{1}{\mathcal{N}}}\\
    	\le&\ \varepsilon+\bigg(\sum_{k=k_0}^{\infty}2^{-2k\mathcal{N}}\bigg)^{\frac{1}{\mathcal{N}}}= \varepsilon+\bigg(\frac{1}{1-2^{-2\mathcal{N}}}\bigg)^\frac{1}{\mathcal{N}}2^{-2k_0}<2\varepsilon.
    \end{align*}
    Therefore, $h_n\rightarrow h$ as $n\rightarrow\infty$ with respect to $\|\cdot\|_{p(\cdot),q,b}$. Hence, $L_{p(\cdot),q,b}$ is complete.
\end{proof}

At the end of this subsection, we establish dominated convergence theorems in variable Lorentz-Karamata spaces. We begin with the definition of absolutely continuous quasi-norm.
Let $(X,\|\cdot\|_X)$ be a (quasi)-normed space. A function $f\in X$ is said to have absolutely continuous (quasi)-norm in $X$, if
$$\lim\limits_{n\rightarrow\infty}\|f\chi_{A_n}\|_X=0$$
for every sequence $(A_n)_{n\ge0}$ satisfying $\lim\limits_{n\rightarrow\infty}\mathbb{P}(A_n)=0$.

\begin{lem}
	Let $p(\cdot)\in\mathcal{P}(\Omega)$, $0<q<\infty$ and $b$ be a slowly varying function. Then every $f\in L_{p(\cdot),q,b}$ has absolutely continuous quasi-norm.
\end{lem}
\begin{proof}
	Since $f\in L_{p(\cdot),q,b}$, for any $\varepsilon>0$, there exists $N_1\in\mathbb{N}$ such that
	 $$\bigg(\sum_{k=N_1}^{\infty}2^{kq}\|\chi_{\{|f|>2^k\}}\|_{p(\cdot)}^q\gamma_{b}^q\big(\|\chi_{\{|f|>2^k\}}\|_{p(\cdot)}\big)\bigg)^{\frac{1}{q}}<\varepsilon.$$
	Let the sequence $(A_n)_{n\ge0}$ satisfy $\lim\limits_{n\rightarrow\infty}\mathbb{P}(A_n)=0$. There exists $N_2\in\mathbb{N}$ such that for $n\ge N_2$, $\mathbb{P}(A_n)<\big(\frac{\varepsilon}{2^{N_1}\gamma_{b}(1)}\big)^{2p_+}$.
	Now let $n\ge N_2$. By Proposition \ref{b} (2) and Lemma \ref{2.4}, we have
	\begin{align*}
		\|f\chi_{A_n}\|_{p(\cdot),q,b}^q=&\ \sum_{k=N_1}^{\infty}2^{kq}\|\chi_{\{|f\chi_{A_n}|>2^k\}}\|_{p(\cdot)}^q\gamma_{b}^q\big(\|\chi_{\{|f\chi_{A_n}|>2^k\}}\|_{p(\cdot)}\big)\\
		&\ \ \ \ \ \ \ \ \ +\sum_{k=-\infty}^{N_1-1}2^{kq}\|\chi_{\{|f\chi_{A_n}|>2^k\}}\|_{p(\cdot)}^q\gamma_{b}^q\big(\|\chi_{\{|f\chi_{A_n}|>2^k\}}\|_{p(\cdot)}\big)\\
		\lesssim&\ \sum_{k=N_1}^{\infty}2^{kq}\|\chi_{\{|f|>2^k\}}\|_{p(\cdot)}^q\gamma_{b}^q\big(\|\chi_{\{|f|>2^k\}}\|_{p(\cdot)}\big)+\sum_{k=-\infty}^{N_1-1}2^{kq}\|\chi_{A_n}\|_{p(\cdot)}^q\gamma_{b}^q\big(\|\chi_{A_n}\|_{p(\cdot)}\big)\\
		<&\ \varepsilon^q+\sum_{k=-\infty}^{N_1-1}2^{kq}\|\chi_{A_n}\|_{p(\cdot)}^{q/2}\|\chi_{A_n}\|_{p(\cdot)}^{q/2}\gamma_{b}^q\big(\|\chi_{A_n}\|_{p(\cdot)}\big)\\
		\lesssim&\ \varepsilon^q+\sum_{k=-\infty}^{N_1-1}2^{kq}\|\chi_{A_n}\|_{p(\cdot)}^{q/2}\gamma_{b}^q(1)\\
		\le&\ \varepsilon^q+\sum_{k=-\infty}^{N_1-1}2^{kq}(2\|\chi_{A_n}\|_{p_+})^{q/2}\gamma_{b}^q(1)\\
		<&\ \varepsilon^q+2^{q/2}\Big(\frac{\varepsilon}{2^{N_1}\gamma_{b}(1)}\Big)^q\gamma_{b}^q(1)\sum_{k=-\infty}^{N_1-1}2^{kq} =\Big(1+\frac{2^{q/2}}{2^q-1}\Big)\varepsilon^q,
	\end{align*}
    which yields that $\|f\chi_{A_n}\|_{p(\cdot),q,b}\lesssim\varepsilon$. Hence, $f$ has absolutely continuous quasi-norm.
\end{proof}

Next, we give the dominated convergence theorem for two cases: $0<q<\infty$ and $q=\infty$, which extend the results in \cite{jwzw,lz}.
\begin{lem}[Dominated convergence theorem]\label{dct}
	Let $p(\cdot)\in\mathcal{P}(\Omega)$, $0<q<\infty$ and $b$ be a slowly varying function. Suppose that $h_n,f,g\in L_{p(\cdot),q,b}$. If $h_n\rightarrow f$ as $n\rightarrow\infty\ \text{a.e.}$ and $|h_n|\le g\ \text{a.e.}$ for every $n\ge1$, then
	$$\lim\limits_{n\rightarrow\infty}\|h_n-f\|_{p(\cdot),q,b}=0.$$
\end{lem}
\begin{proof}
	Since $g\in L_{p(\cdot),q,b}$ has absolutely continuous quasi-norm, then for any $\varepsilon>0$, there exists $N_1$ such that $\|g\chi_{\{g>N_1\}}\|_{p(\cdot),q,b}<\varepsilon$. 	
	On $\{g\le N_1\}$, it is easy to see that $|h_n-f|\le2N_1$. By Lemma \ref{fanshu} (4), there is
	\begin{align*}
		\|(h_n-f)\chi_{\{g\le N_1\}}\|_{p(\cdot),q,b}=&\ \|(h_n-f)\chi_{\{g\le N_1\}}\chi_{\{h_n\neq f\}}\|_{p(\cdot),q,b}\\
		\le&\ 2N_1\|\chi_{\{h_n\neq f\}}\|_{p(\cdot),q,b}=2N_1\Big(\frac{1}{q}\Big)^{1/q}\|\chi_{\{h_n\neq f\}}\|_{p(\cdot)}\gamma_{b}\big(\|\chi_{\{h_n\neq f\}}\|_{p(\cdot)}\big).
	\end{align*}
    Since $h_n\rightarrow f$ as $n\rightarrow\infty$ a.e., then $\chi_{\{h_n\neq f\}}\rightarrow0$ as $n\rightarrow\infty$. By dominated convergence theorem in $L_{p(\cdot)}$ (see \cite{cf}), it follows from $\chi_{\{h_n\neq f\}}\le1$ that $\|\chi_{\{h_n\neq f\}}\|_{p(\cdot)}\rightarrow0$. Hence there exists $N_2$ such that for $n\ge N_2$,
    $$\|(h_n-f)\chi_{\{g\le N_1\}}\|_{p(\cdot),q,b}<\varepsilon.$$
    Therefore, for $n\ge N_2$, we have
    \begin{align*}
    	\|h_n-f\|_{p(\cdot),q,b}\lesssim&\ \|(h_n-f)\chi_{\{g\le N_1\}}\|_{p(\cdot),q,b}+\|(h_n-f)\chi_{\{g>N_1\}}\|_{p(\cdot),q,b}\\
    	\lesssim&\ \varepsilon+2\|g\chi_{\{g>N_1\}}\|_{p(\cdot),q,b}<3\varepsilon,
    \end{align*}
    which completes the proof.
\end{proof}

For the case $q=\infty$, we shall introduce a subspace of $L_{p(\cdot),\infty,b}$ as follows.
\begin{defi}
	Let $p(\cdot)\in\mathcal{P}(\Omega)$ and $b$ be a slowly varying function. Denote by $\mathcal{L}_{p(\cdot),\infty,b}$ the set of all $f\in L_{p(\cdot),\infty,b}$ having absolutely continuous quasi-norm, that is
	$$\mathcal{L}_{p(\cdot),\infty,b}=\{f\in L_{p(\cdot),\infty,b}:\lim\limits_{n\rightarrow\infty}\|f\chi_{A_n}\|_{p(\cdot),q,b}=0\},$$
	where the sequence $(A_n)_{n\ge0}$ satisfies $\lim\limits_{n\rightarrow\infty}\mathbb{P}(A_n)=0$.
\end{defi}

\begin{lem}
	Let $p(\cdot)\in\mathcal{P}(\Omega)$ and $b$ be a slowly varying function. Then $\mathcal{L}_{p(\cdot),\infty,b}$ is a closed subspace of $L_{p(\cdot),\infty,b}$.
\end{lem}
\begin{proof}
	Let $(g_n)_{n\ge1}\subset\mathcal{L}_{p(\cdot),\infty,b}$ be a Cauchy sequence with respect to quasi-norm $\|\cdot\|_{p(\cdot),\infty,b}$. Then there exists $f\in L_{p(\cdot),\infty,b}$ such that
	$$\lim\limits_{n\rightarrow\infty}\|g_n-f\|_{p(\cdot),\infty,b}=0.$$
	For any given sequence $(A_n)_{n\ge0}$ satisfying $\lim\limits_{n\rightarrow\infty}\mathbb{P}(A_n)=0$, it follows from Lemma \ref{fanshusanjiao} that
	\begin{align*}
		\|f\chi_{A_n}\|_{p(\cdot),\infty,b}\lesssim&\ \|g_n\chi_{A_n}\|_{p(\cdot),\infty,b}+\|(g_n-f)\chi_{A_n}\|_{p(\cdot),\infty,b}\\
		\lesssim&\ \|g_n\chi_{A_n}\|_{p(\cdot),\infty,b}+\|g_n-f\|_{p(\cdot),\infty,b}.
	\end{align*}
	Since $g_n\in\mathcal{L}_{p(\cdot),\infty,b}$, we get
	$$\|f\chi_{A_n}\|_{p(\cdot),\infty,b}\rightarrow0,\ \ \text{as}\ n\rightarrow\infty.$$
	Hence, it implies that $f\in\mathcal{L}_{p(\cdot),\infty,b}$.
\end{proof}

Similarly to the proof of Lemma \ref{dct}, we get the dominated convergence theorem for $q=\infty$, the proof is left to the readers.
\begin{lem}
	Let $p(\cdot)\in\mathcal{P}(\Omega)$ and $b$ be a slowly varying function. Suppose that $h_n,f\in L_{p(\cdot),\infty,b}$ and $g\in\mathcal{L}_{p(\cdot),\infty,b}$. If $h_n\rightarrow f $ as $n\rightarrow\infty$ a.e and $|h_n|\le g\ \ a.e.$ for every $n\ge1$, then
	$$\lim\limits_{n\rightarrow\infty}\|h_n-f\|_{p(\cdot),\infty,b}=0.$$
\end{lem}

\subsection{Martingale and variable martingale Hardy-Lorentz-Karamata spaces}
Let $\{\mathcal{F}_n\}_{n\ge0}$ be a nondecreasing sequence of sub-$\sigma$-algebras of $\mathcal{F}$ such that $\mathcal{F}=\sigma\Big(\bigcup\limits_{n\ge0}\mathcal{F}_n\Big)$.
The expectation operator and the conditional expectation operator relative to $\mathcal{F}_n$ are denoted by $\mathbb{E}$ and $\mathbb{E}_n$, respectively.
A sequence of measurable functions $f=(f_n)_{n\ge0}\subset L_1$ is called a martingale with respect to $\{\mathcal{F}_n\}_{n\ge0}$ if
$$
\mathbb{E}_n(f_{n+1})=f_n \qquad \text{for every} \; n\ge0.
$$
Denote by $\mathcal{M}$ the set of all martingales $f=(f_n)_{n\ge0}$ relative to $\{\mathcal{F}_n\}_{n\ge0}$ such that $f_0=0$.
For $f\in\mathcal{M}$, we define the martingale differences by
$$
d_0f:=0, \qquad d_nf:=f_n-f_{n-1} \qquad (n>0).
$$
Let $\mathcal{T}$ be the set of all stopping times relative to $\{\mathcal{F}_n\}_{n\ge0}$. For $f\in\mathcal{M}$ and $\tau\in\mathcal{T}$, the stopped martingale $f^\tau=(f_n^\tau)_{n\ge0}$ is defined by
$$f_n^\tau:=\sum_{m=0}^{n}\chi_{\{\tau\ge m\}}d_mf.$$
	
Define the maximal function, the square function and the conditional square function of a martingale $f$, respectively, as follows:
$$M_m(f)=\sup_{n\le m}|f_n|,\ \ M(f)=\sup_{n\ge0}|f_n|;$$
$$S_m(f)=\bigg(\sum_{n=0}^{m}|d_nf|^2\bigg)^\frac{1}{2},\ \ S(f)=\bigg(\sum_{n=0}^{\infty}|d_nf|^2\bigg)^\frac{1}{2};$$
$$s_m(f)=\bigg(\sum_{n=0}^{m}\mathbb{E}_{n-1}|d_nf|^2\bigg)^\frac{1}{2},\ \ s(f)=\bigg(\sum_{n=0}^{\infty}\mathbb{E}_{n-1}|d_nf|^2\bigg)^\frac{1}{2}.$$
	
Let $p(\cdot)\in\mathcal{P}(\Omega)$, $0<q\le\infty$ and $b$ be a slowly varying function. Denote by $\Lambda_{p(\cdot),q,b}$ the collection of all sequences $\lambda=(\lambda_n)_{n\ge0}$ of nondecreasing, nonnegative and adapted functions with $\lambda_\infty:=\lim\limits_{n\rightarrow\infty}\lambda_n\in L_{p(\cdot),q,b}$.
	
\begin{defi}\label{hardy}
	Let $p(\cdot)\in\mathcal{P}(\Omega)$, $0<q\le\infty$ and $b$ be a slowly varying function. The variable martingale Hardy spaces associated with variable Lorentz-Karamata spaces are defined by
	 $$H_{p(\cdot),q,b}^M=\big\{f\in\mathcal{M}:\|f\|_{H_{p(\cdot),q,b}^M}=\|M(f)\|_{p(\cdot),q,b}<\infty\big\};$$
	 $$H_{p(\cdot),q,b}^s=\big\{f\in\mathcal{M}:\|f\|_{H_{p(\cdot),q,b}^s}=\|s(f)\|_{p(\cdot),q,b}<\infty\big\};$$
	 $$H_{p(\cdot),q,b}^S=\big\{f\in\mathcal{M}:\|f\|_{H_{p(\cdot),q,b}^S}=\|S(f)\|_{p(\cdot),q,b}<\infty\big\};$$
	$$\mathcal{Q}_{p(\cdot),q,b}=\big\{f\in\mathcal{M}:\exists\ \lambda=(\lambda_n)_{n\ge0}\in\Lambda_{p(\cdot),q,b}\ \text{s.t.}\ S_n(f)\le\lambda_{n-1}\big\}$$
	$$with\ \|f\|_{\mathcal{Q}_{p(\cdot),q,b}}=\inf_{\lambda\in\Lambda_{p(\cdot),q,b}}\|\lambda_\infty\|_{p(\cdot),q,b};$$
	$$\mathcal{P}_{p(\cdot),q,b}=\big\{f\in\mathcal{M}:\exists\ \lambda=(\lambda_n)_{n\ge0}\in\Lambda_{p(\cdot),q,b}\ \text{s.t.}\ |f_n|\le\lambda_{n-1}\big\}$$
	$$with\ \|f\|_{\mathcal{P}_{p(\cdot),q,b}}=\inf_{\lambda\in\Lambda_{p(\cdot),q,b}}\|\lambda_\infty\|_{p(\cdot),q,b}.$$
	We define $\mathcal{H}_{p(\cdot),\infty,b}^M$ as the space of all martingales such that $M(f)\in\mathcal{L}_{p(\cdot),\infty,b}$. Replacing $M(f)$ by $s(f)$ or $S(f)$, we can define $\mathcal{H}_{p(\cdot),\infty,b}^s$ and $\mathcal{H}_{p(\cdot),\infty,b}^S$, respectively.
\end{defi}

\begin{rem}\label{huaH}
	$(1)$ It is easy to check that $\mathcal{H}_{p(\cdot),\infty,b}^M$, $\mathcal{H}_{p(\cdot),\infty,b}^s$ and $\mathcal{H}_{p(\cdot),\infty,b}^S$ are closed subspaces of $H_{p(\cdot),\infty,b}^M$, $H_{p(\cdot),\infty,b}^s$ and $H_{p(\cdot),\infty,b}^S$, respectively.
	
	$(2)$ If we take $p(\cdot)\equiv p$ in Definition $\ref{hardy}$, we get back the martingale Hardy-Lorentz-Karamata spaces introduced in \cite{hoa,jxz}. Moreover, the variable martingale Hardy-Lorentz space defined in \cite{jwwz,jwzw} is also a special case of the variable martingale Hardy-Lorentz-Karamata space when $b\equiv1$.
\end{rem}

We recall the definition of regularity. The stochastic basis $\{\mathcal{F}_n\}_{n\ge0}$ is said to be regular, if for $n\ge0$ and $A\in\mathcal{F}_n$, there exists $B\in\mathcal{F}_{n-1}$ such that $A\subset B$ and $\mathbb{P}(B)\le \mathcal{R}\mathbb{P}(A)$, where $\mathcal{R}$ is a positive constant independent of $n$. Equivalently, there exists a constant $\mathcal{R}>0$ such that
\begin{align}\label{R}
	f_n\le \mathcal{R}f_{n-1}
\end{align}
for all nonnegative martingales $(f_n)_{n\ge0}$ adapted to the stochastic basis $\{\mathcal{F}_n\}_{n\ge0}$. We refer to \cite[Chapter 7]{l} for more details
about regular basis.

Recall that $B\in\mathcal{F}_n$ is an atom if $A\subset B$ with $A\in\mathcal{F}_n$ satisfies $\mathbb{P}(A)<\mathbb{P}(B)$, then $\mathbb{P}(A)=0$. Denote by $A(\mathcal{F}_n)$ the set of all atoms in $\mathcal{F}_n$.
In the sequel of this paper, we always assume that every $\sigma$-algebra $\mathcal{F}_n$ is generated by countable many atoms. There are many examples for $\sigma$-algebras generated by countable atoms, see \cite{jwzw,w}.


We need the following Doob maximal inequality on variable Lebesgue spaces, which was proved by Jiao et al. \cite{jwzw}.
\begin{lem}
	Let $p(\cdot)\in\mathcal{P}(\Omega)$ satisfy $(\ref{gs3})$ with $1<p_-\le p_+<\infty$. Then for any $f\in L_{p(\cdot)}$, we have
	$$\|M(f)\|_{p(\cdot)}\lesssim\|f\|_{p(\cdot)}.$$
\end{lem}

Moreover, by applying the inequality above and interpolation results for variable Lorentz-Karamata spaces, we obtain the next generalization of Doob's maximal inequality.
\begin{lem}[\cite{hdlw}] \label{Doob}
	Let $p(\cdot)\in\mathcal{P}(\Omega)$ satisfy $(\ref{gs3})$ with $1<p_-\le p_+<\infty$, $0<q\le\infty$ and let $b$ be a slowly varying function. Then
	$$\|M(f)\|_{p(\cdot),q,b}\lesssim\|f\|_{p(\cdot),q,b} \quad \text{for any}\ f\in L_{p(\cdot),q,b}.$$
\end{lem}

\section{Atomic Decompositions} \label{s3}

The atomic decomposition is a powerful tool for dealing with duality theorems and some fundamental inequalities
both in martingale theory and harmonic analysis. In this section, we characterize the five variable martingale Hardy-Lorentz-Karamata spaces by atomic decompositions. Note that there are two notions of atoms in this paper, one is a measurable set described in Section 2, the other is a measurable function defined in this section (see Definitions \ref{atom} and \ref{1defi1}).

\subsection{Atomic Decompositions via Simple Atoms}
In this subsection, we use simple atoms.

\begin{defi}[\cite{jwwz}]\label{atom}
	Let $p(\cdot)\in\mathcal{P}(\Omega)$ and $1<r\le\infty$. A function $a$ is called a simple $(p(\cdot),r)^s$-atom $($resp. simple $(p(\cdot),r)^S$-atom or simple $(p(\cdot),r)^M$-atom$)$ if there exists $I\in A(\mathcal{F}_i)$ such that
	
	$(1)$ $supp(a)\subset I;$
	
	$(2)$ $\|s(a)\|_r\ (\text{resp.}\ \|S(a)\|_r\ \text{or}\ \|M(a)\|_r)\le\frac{\mathbb{P}(I)^{1/r}}{\|\chi_{I}\|_{p(\cdot)}};$
		
	$(3)$ $\mathbb{E}_i(a)=0.$
\end{defi}

\begin{rem}
	$(i)$ If we take $p(\cdot)\equiv p$, then the simple $(p(\cdot),r)^*$-atom $(*=s,S,M)$ is the same as the classical definition in Weisz \cite{w}.
	
	$(ii)$ Let $p(\cdot)\in\mathcal{P}(\Omega)$ and $1<r\le\infty$. If $a$ is a simple $(p(\cdot),r)^*$-atom $(*=s,S,M)$ associated with $I\in A(\mathcal{F}_i)$ for some $i\in\mathbb{N}$, then
	$$s(a)\chi_{I}=s(a),\ \ \ S(a)\chi_{I}=S(a)\ \ \ \text{and}\ \ \ M(a)\chi_{I}=M(a).$$
	We refer to \cite{jwwz} for the proof.
\end{rem}


\begin{defi}
Let $p(\cdot)\in\mathcal{P}(\Omega)$, $0<q,r\leq \infty$ and let $b$ be a slowly varying function.
The atomic martingale Lorentz-Karamata space with variable exponent $H_{p(\cdot),q,b}^{s-at,r,1}$
$\big($resp. $H_{p(\cdot),q,b}^{s-at,r,2}$ or $H_{p(\cdot),q,b}^{s-at,r,3}$$\big)$
is defined to be the set of all $f=(f_n)_{n\geq 0} \in \mathcal{M} $
such that for each $n\geq0$,
\begin{eqnarray}\label{deco-f}
f_n= \sum_{ k \in \mathbb{Z} }\sum_{i=0}^{n-1}\sum_{j} \mu_{k,i,j}\mathbb{E}_n(a^{k,i,j}),
\end{eqnarray}
where $(a^{k,i,j})_{k\in\mathbb{Z},i\in\mathbb{N},j}$ is a sequence of simple $(p(\cdot),r)^s$-atoms $($resp. simple $(p(\cdot),r)^S$-atoms or simple $(p(\cdot),r)^M$-atoms$)$ associated with $(I_{k,i,j})_{k\in\mathbb{Z},i\in\mathbb{N},j}$, the sets $I_{k,i,j}$ are disjoint for any fixed $k$, $\mu_{k,i,j}=3\cdot2^{k}\|\chi_{I_{k,i,j}}\|_{p(\cdot)}$ and
$$\Bigg\|\Bigg\{\bigg\|\sum_{i=0}^{\infty}\sum_{j}\mu_{k,i,j}\|\chi_{I_{k,i,j}}\|_{p(\cdot)}^{-1}\chi_{I_{k,i,j}}\bigg\|_{p(\cdot)}\gamma_b\bigg(\bigg\|\sum_{i=0}^{\infty}\sum_{j}\chi_{I_{k,i,j}}\bigg\|_{p(\cdot)}\bigg)\Bigg\}_{k\in\mathbb{Z}}\Bigg\|_{l_q}<\infty.$$
Endow $H_{p(\cdot),q,b}^{s-at,r,1}$ $\big($resp. $H_{p(\cdot),q,b}^{s-at,r,2}$ or $H_{p(\cdot),q,b}^{s-at,r,3}$$\big)$  with the $($quasi$)$-norm
\begin{align*}
&\ \|f\|_{H_{p(\cdot),q,b}^{s-at,r,1}}\; \big(\text{resp.}\; \|f\|_{H_{p(\cdot),q,b}^{s-at,r,2}}\; \text{or}\; \|f\|_{H_{p(\cdot),q,b}^{s-at,r,3}}\big)\\
=&\ \inf
\Bigg\|\Bigg\{\bigg\|\sum_{i=0}^{\infty}\sum_{j}\mu_{k,i,j}\|\chi_{I_{k,i,j}}\|_{p(\cdot)}^{-1}\chi_{I_{k,i,j}}
\bigg\|_{p(\cdot)}\gamma_b\bigg(\bigg\|\sum_{i=0}^{\infty}\sum_{j}\chi_{I_{k,i,j}}\bigg\|_{p(\cdot)}\bigg)
\Bigg\}_{k\in\mathbb{Z}}\Bigg\|_{l_q},
\end{align*}
where the infimum is taken over all decompositions of $f$ as above.
\end{defi}

In order to present the atomic decomposition theorems for variable Lorentz-Karamata spaces, we need the following lemmas.
First of all, we shall show a useful tool to verify whether a function belongs to variable Lorentz-Karamata spaces. Its original idea comes from Abu-Shammala and Torchinsky \cite{at}.

\begin{lem}\label{belong}
	Let $p(\cdot)\in\mathcal{P}(\Omega)$, $0<q\le\infty$ and $b$ be a slowly varying function. Assume that the nonnegative sequence $\{2^k\mu_k\}_{k\in\mathbb{Z}}\in l_q$. Suppose that the nonnegative function $\varphi$ verifies the following property: there exist $0<m<\infty$ and $0<\varepsilon<1$ such that, given an arbitrary integer $k_0$, we have $\varphi\le\psi_{k_0}+\eta_{k_0}$, where $\psi_{k_0}$ is essentially bounded and satisfies $\|\psi_{k_0}\|_\infty\lesssim2^{k_0}$ and $\eta_{k_0}$ satisfies
	\begin{align}\label{gs19}
		 2^{k_0m\varepsilon}\|\chi_{\{\eta_{k_0}>2^{k_0}\}}\|_{p(\cdot)}^m\gamma_b^ m\big(\|\chi_{\{\eta_{k_0}>2^{k_0}\}}\|_{p(\cdot)}\big)\lesssim\sum_{k=k_0}^{\infty}(2^{k\varepsilon}\mu_k)^m.
	\end{align}
	Then $\varphi\in L_{p(\cdot),q,b}$ and
	 $$\|\varphi\|_{p(\cdot),q,b}\lesssim\|\{2^k\mu_k\}_{k\in\mathbb{Z}}\|_{l_q}.$$
\end{lem}
\begin{proof}
	For given $k_0\in\mathbb{Z}$, let $\psi_{k_0}$ and $\eta_{k_0}$ be as above. Since $\|\psi_{k_0}\|_\infty\lesssim2^{k_0}$, there exists a positive constant $C$ such that $\mathbb{P}(\psi_{k_0}>C2^{k_0})=0$. Then we have
	 $$\{\varphi>\lambda2^{k_0}\}\subset\{\psi_{k_0}>C2^{k_0}\}\cup\{\eta_{k_0}>2^{k_0}\}=\{\eta_{k_0}>2^{k_0}\},$$
	where $\lambda=C+1$.
	Thus, it clearly suffices to verify that $$\big\|\big\{2^{k_0}\|\chi_{\{\eta_{k_0}>2^{k_0}\}}\|_{p(\cdot)}\gamma_b\big(\|\chi_{\{\eta_{k_0}>2^{k_0}\}}\|_{p(\cdot)}\big)\big\}_{k_0\in\mathbb{Z}}\big\|_{l_q}\lesssim\|\{2^k\mu_k\}_{k\in\mathbb{Z}}\|_{l_q}.$$
	
	When $q=\infty$, the hypothesis (\ref{gs19}) gives that
	\begin{align*}
		 2^{k_0}\|\chi_{\{\eta_{k_0}>2^{k_0}\}}\|_{p(\cdot)}\gamma_b\big(\|\chi_{\{\eta_{k_0}>2^{k_0}\}}\|_{p(\cdot)}\big)\lesssim&\ 2^{k_0}\bigg(2^{-k_0m\varepsilon
		 }\sum_{k=k_0}^{\infty}\big(2^{k\varepsilon}\mu_k\big)^m\bigg)^\frac{1}{m}\\\notag
		=&\ 2^{k_0(1-\varepsilon)}\bigg(\sum_{k=k_0}^{\infty}2^{km(\varepsilon-1)}(2^k\mu_k)^m\bigg)^\frac{1}{m}\\
		\lesssim&\ 2^{k_0(1-\varepsilon)}2^{k_0(\varepsilon-1)}\sup_{k\ge k_0}2^k\mu_k=\sup_{k\ge k_0}2^k\mu_k.
	\end{align*}
	It follows that
	 $$\big\|\big\{2^{k_0}\|\chi_{\{\eta_{k_0}>2^{k_0}\}}\|_{p(\cdot)}\gamma_b\big(\|\chi_{\{\eta_{k_0}>2^{k_0}\}}\|_{p(\cdot)}\big)\big\}_{k_0\in\mathbb{Z}}\big\|_{l_\infty}\lesssim\|\{2^k\mu_k\}_{k\in\mathbb{Z}}\|_{l_\infty}.$$
	
	When $0<q<\infty$, let $0<m<\infty$, $0<\varepsilon<1$ and $\theta=\frac{1-\varepsilon}{2}$. Then
	\begin{align}\label{gs22}
		 \sum_{k=k_0}^{\infty}(2^{k\varepsilon}\mu_k)^m=\sum_{k=k_0}^{\infty}2^{-k\theta m}\big(2^{k(1-\theta)}\mu_k\big)^m.
	\end{align}
	Setting $r=\frac{q}{m}$, we consider the right side of (\ref{gs22}) for two cases: $0<r<1$ and $1\le r<\infty$. If $0<m\le q$, by using H\"{o}lder's inequality with $\frac{1}{r}+\frac{1}{r'}=1$, we have
	\begin{align*}
		\sum_{k=k_0}^{\infty}2^{-k\theta m}\big(2^{k(1-\theta)}\mu_k\big)^m\le&\ \bigg(\sum_{k=k_0}^{\infty}2^{-k\theta mr'}\bigg)^\frac{1}{r'}\bigg(\sum_{k=k_0}^{\infty}\big(2^{k(1-\theta)}\mu_k\big)^q\bigg)^\frac{1}{r}\\
		=&\ \bigg(\frac{1}{1-2^{\theta mr'}}\bigg)^{1/r'}2^{-k_0\theta m}\bigg(\sum_{k=k_0}^{\infty}\big(2^{k(1-\theta)}\mu_k\big)^q\bigg)^\frac{m}{q}.
	\end{align*}
	When $0<q\le m$, we get a similar estimation by simply observing that
	\begin{align*}
		\sum_{k=k_0}^{\infty}2^{-k\theta m}\big(2^{k(1-\theta)}\mu_k\big)^m\le2^{-k_0\theta m}\bigg(\sum_{k=k_0}^{\infty}\big(2^{k(1-\theta)}\mu_k\big)^m\bigg)^{r\cdot\frac{1}{r}}\le&\ 2^{-k_0\theta m}\bigg(\sum_{k=k_0}^{\infty}\big(2^{k(1-\theta)}\mu_k\big)^q\bigg)^\frac{m}{q}.
	\end{align*}
	Combining with (\ref{gs19}) and (\ref{gs22}), we deduce that
	 $$2^{k_0m\varepsilon}\|\chi_{\{\eta_{k_0}>2^{k_0}\}}\|_{p(\cdot)}^m\gamma_b^ m\big(\|\chi_{\{\eta_{k_0}>2^{k_0}\}}\|_{p(\cdot)}\big)\lesssim2^{-k_0\theta m}\bigg(\sum_{k=k_0}^{\infty}\big(2^{k(1-\theta)}\mu_k\big)^q\bigg)^\frac{m}{q}.$$
	It yields that
	\begin{align*}
		&\ 2^{k_0q}\|\chi_{\{\eta_{k_0}>2^{k_0}\}}\|_{p(\cdot)}^q\gamma_b^ q\big(\|\chi_{\{\eta_{k_0}>2^{k_0}\}}\|_{p(\cdot)}\big)\\
		\lesssim&\ 2^{k_0(1-\varepsilon-\theta) q}\sum_{k=k_0}^{\infty}\big(2^{k(1-\theta)}\mu_k\big)^q=2^{k_0\theta q}\sum_{k=k_0}^{\infty}\big(2^{k(1-\theta)}\mu_k\big)^q.
	\end{align*}
	Moreover, according to Abel's transformation, there is
	\begin{align*}
		&\ \sum_{k_0\in\mathbb{Z}}2^{k_0q}\|\chi_{\{\eta_{k_0}>2^{k_0}\}}\|_{p(\cdot)}^q\gamma_b^ q\big(\|\chi_{\{\eta_{k_0}>2^{k_0}\}}\|_{p(\cdot)}\big)\\\notag
		\lesssim&\ \sum_{k_0\in\mathbb{Z}}2^{k_0\theta q}\sum_{k=k_0}^{\infty}\big(2^{k(1-\theta)}\mu_k\big)^q=\sum_{k\in\mathbb{Z}}\big(2^{k(1-\theta)}\mu_k\big)^q\sum_{k_0=-\infty}^{k}2^{k_0\theta q}\\\notag
		=&\ \frac{1}{1-2^{-\varepsilon(1-\varepsilon-\theta)}}\sum_{k\in\mathbb{Z}}\big(2^{k(1-\theta)}\mu_k\big)^q2^{k\theta q}=\sum_{k\in\mathbb{Z}}2^{kq}\mu_k^q\notag.
	\end{align*}
	Hence, we have
	\begin{align*}
		\|\varphi\|_{p(\cdot),q,b}\approx&\ \big\|\big\{2^{k_0}\|\chi_{\{\varphi_{k_0}>\lambda2^{ k_0}\}}\|_{p(\cdot)}\gamma_b\big(\|\chi_{\{\varphi_{k_0}>\lambda2^{ k_0}\}}\|_{p(\cdot)}\big)\big\}_{k_0\in\mathbb{Z}}\big\|_{l_q}\\
		\lesssim&\ \big\|\big\{2^{k_0}\|\chi_{\{\eta_{k_0}>2^{k_0}\}}\|_{p(\cdot)}\gamma_b\big(\|\chi_{\{\eta_{k_0}>2^{k_0}\}}\|_{p(\cdot)}\big)\big\}_{k_0\in\mathbb{Z}}\big\|_{l_q}\\
		\lesssim&\ \|\{2^k\mu_k\}_{k\in\mathbb{Z}}\|_{l_q},
	\end{align*}
	which completes the proof.
\end{proof}

Also, we need the next lemma, its proof can be found in \cite{jwwz}.

\begin{lem} \label{3.3}
	Let $p(\cdot)\in\mathcal{P}(\Omega)$ satisfy $(\ref{gs3})$ and $\max\{p_+,1\}<r<\infty$. Take $0<\varepsilon<\underline{p}$ and $1<L<\min\{\frac{r}{p_+},\frac{1}{\varepsilon}\}$. If for a sublinear operator $T$ and all simple $(p(\cdot),r)^*$-atoms $a^{k,i,j}$ $(*=s,S,M)$, there is
	 $$\|T(a^{k,i,j})\|_r\lesssim\frac{\|\chi_{I_{k,i,j}}\|_r}{\|\chi_{I_{k,i,j}}\|_{p(\cdot)}},$$
	then $$\bigg\|\sum_{i=0}^{\infty}\sum_{j}\Big[\|\chi_{I_{k,i,j}}\|_{p(\cdot)}T(a^{k,i,j})\chi_{I_{k,i,j}}\Big]^{L\varepsilon}\bigg\|_{\frac{p(\cdot)}{\varepsilon}}\lesssim\bigg\|\sum_{i=0}^{\infty}\sum_{j}\chi_{I_{k,i,j}}\bigg\|_{\frac{p(\cdot)}{\varepsilon}}.$$
\end{lem}

\begin{thm}\label{ad1}
	Let $p(\cdot)\in\mathcal{P}(\Omega)$ satisfy $(\ref{gs3})$, $0<q\le\infty$, $\max\{p_+,1\}<r\le\infty$ and let $b$ be a slowly varying function. Then
    $$H_{p(\cdot),q,b}^s=H_{p(\cdot),q,b}^{s-at,r,1}$$
    with equivalent $($quasi$)$-norms.
\end{thm}
\begin{proof}
	Let $f\in H_{p(\cdot),q,b}^s$. For $k\in\mathbb{Z}$ and $n\in\mathbb{N}$, define
	$$\tau_k:=\inf\{n\in\mathbb{N}:s_{n+1}(f)>2^k\}.$$
	It is easy to check that $\{\tau_k\}_{k\in\mathbb{Z}}$ is a nondecreasing sequence of stopping times. For each $n\in\mathbb{N}$, there is
	$$f_n=\sum_{k\in\mathbb{Z}}(f_n^{\tau_{k+1}}-f_n^{\tau_k})\ \ \ \ a.e.$$
	It is obvious that for fixed $k,i$, there exist disjoint atoms $(I_{k,i,j})_j\subset A(\mathcal{F}_i)$ such that
	$$\bigcup_{j}I_{k,i,j}=\{\tau_k=i\}\in\mathcal{F}_i.$$
	Thus the sets $I_{k,i,j}$ are disjoint for each fixed $k$.
	Then for each $n\in\mathbb{N}$, we have
	\begin{align*}
		 f_n&=\sum_{k\in\mathbb{Z}}(f_n^{\tau_{k+1}}-f_n^{\tau_k})\chi_{\{\tau_k<n\}}\\
		 &=\sum_{k\in\mathbb{Z}}\sum_{i=0}^{n-1}(f_n^{\tau_{k+1}}-f_n^{\tau_k})\chi_{\{\tau_k=i\}}\\
		 &=\sum_{k\in\mathbb{Z}}\sum_{i=0}^{n-1}\sum_{j}(f_n^{\tau_{k+1}}-f_n^{\tau_k})\chi_{I_{k,i,j}}
	\end{align*}
	and $\bigcup\limits_{i=0}^{\infty}\bigcup\limits_{j}I_{k,i,j}=\{\tau_k<\infty\}.$
	Set
	$$\mu_{k,i,j}:=3\cdot2^{k}\|\chi_{I_{k,i,j}}\|_{p(\cdot)}\quad \text{and}\quad a_n^{k,i,j}:=\frac{f_n^{\tau_{k+1}}-f_n^{\tau_k}}{\mu_{k,i,j}}\chi_{I_{k,i,j}}$$
	(if $\mu_{k,i,j}=0$, then let $a_n^{k,i,j}=0$).
	Since $f_n^{\tau_{k+1}}=\sum\limits_{m=0}^{n}\chi_{\{\tau_{k+1}\ge m\}}d_mf$ and $I_{k,i,j}\subset\{\tau_k=i\}$, we have
	\begin{align}\label{e12}
		(f_n^{\tau_{k+1}}-f_n^{\tau_k})\chi_{I_{k,i,j}}=&\ \chi_{I_{k,i,j}}\sum_{m=0}^{n}\chi_{\{\tau_{k+1}\ge m>\tau_k\}}d_mf\\
		=&\ \notag
\chi_{I_{k,i,j}}\sum_{m=i+1}^{n}\chi_{\{\tau_{k+1}\ge m>\tau_k\}}d_mf.
	\end{align}
	Hence
	$$\mathbb{E}_i(a_n^{k,i,j})=0,\ \ \int_{I_{k.i.j}}a_n^{k,i,j}=0,$$
	and for fixed $k,i,j$, $(a_n^{k,i,j})_{n\ge0}$ is a martingale. By the definition of $\tau_k$, we have
	 $$s\big((a_n^{k,i,j})_n\big)\le\frac{s(f_n^{\tau_{k+1}})+s(f_n^{\tau_k})}{\mu_{k,i,j}}\le\frac{2^{k+1}+2^k}{\mu_{k,i,j}}=\frac{1}{\|\chi_{I_{k,i,j}}\|_{p(\cdot)}}.$$
	Thus $(a_n^{k,i,j})_{n\ge0}$ is an $L_2$-bounded martingale. Moreover, there exists $a^{k,i,j}\in L_2$ such that $\mathbb{E}_n(a^{k,i,j})=a_n^{k,i,j}$ and
	 $$\|s(a^{k,i,j})\|_r\le\mathbb{P}(I_{k,i,j})^{1/r}\frac{1}{\|\chi_{I_{k,i,j}}\|_{p(\cdot)}}.$$
	Therefore, $a^{k,i,j}$ is a simple $(p(\cdot),r)^s$-atom and (\ref{deco-f}) holds.
	
	For the case of $0<q<\infty$, it follows from $\bigcup\limits_{i=0}^{\infty}\bigcup\limits_{j}I_{k,i,j}=\{\tau_k<\infty\}=\{s(f)>2^k\}$ and Proposition \ref{b} (2) that
	\begin{align*}
		&\ \sum_{k\in\mathbb{Z}}\bigg\|\sum_{i=0}^{\infty}\sum_{j}\mu_{k,i,j}\|\chi_{I_{k,i,j}}\|_{p(\cdot)}^{-1}\chi_{I_{k,i,j}}\bigg\|_{p(\cdot)}^q\gamma_b^q\bigg(\bigg\|\sum_{i=0}^{\infty}\sum_{j}\chi_{I_{k,i,j}}\bigg\|_{p(\cdot)}\bigg)\\
		=&\ \sum_{k\in\mathbb{Z}}\bigg\|\sum_{i=0}^{\infty}\sum_{j}3\cdot2^{k}\chi_{I_{k,i,j}}\bigg\|_{p(\cdot)}^q\gamma_b^q\bigg(\bigg\|\sum_{i=0}^{\infty}\sum_{j}\chi_{I_{k,i,j}}\bigg\|_{p(\cdot)}\bigg)\\
		=&\ 3^q\sum_{k\in\mathbb{Z}}2^{kq}\bigg\|\sum_{i=0}^{\infty}\sum_{j}\chi_{I_{k,i,j}}\bigg\|_{p(\cdot)}^q\gamma_b^q\bigg(\bigg\|\sum_{i=0}^{\infty}\sum_{j}\chi_{I_{k,i,j}}\bigg\|_{p(\cdot)}\bigg)\\
		=&\ 3^q\sum_{k\in\mathbb{Z}}2^{kq}\|\chi_{\{s(f)>2^k\}}\|_{p(\cdot)}^q\gamma_b^q\big(\|\chi_{\{s(f)>2^k\}}\|_{p(\cdot)}\big)\\
		\lesssim&\ \sum_{k\in\mathbb{Z}}\int_{2^{k-1}}^{2^k}\|\chi_{\{s(f)>2^k\}}\|_{p(\cdot)}^q\gamma_b^q\big(\|\chi_{\{s(f)>2^k\}}\|_{p(\cdot)}\big) t^{q-1}dt\\
		\lesssim&\ \sum_{k\in\mathbb{Z}}\int_{2^{k-1}}^{2^k}\|\chi_{\{s(f)>t\}}\|_{p(\cdot)}^q\gamma_b^q\big(\|\chi_{\{s(f)>t\}}\|_{p(\cdot)}\big) t^{q-1}dt\\
		=&\ \int_{0}^{\infty}\|\chi_{\{s(f)>t\}}\|_{p(\cdot)}^q\gamma_b^q\big(\|\chi_{\{s(f)>t\}}\|_{p(\cdot)}\big) t^{q-1}dt\\
		=&\ \|f\|_{H_{p(\cdot),q,b}^s}^q.
	\end{align*}
	Similarly, when $q=\infty$, we have
	\begin{align*}
		 &\sup_{k\in\mathbb{Z}}\bigg\|\sum_{i=0}^{\infty}\sum_{j}\mu_{k,i,j}\|\chi_{I_{k,i,j}}\|_{p(\cdot)}^{-1}\chi_{I_{k,i,j}}\bigg\|_{p(\cdot)}\gamma_b\bigg(\bigg\|\sum_{i=0}^{\infty}\sum_{j}\chi_{I_{k,i,j}}\bigg\|_{p(\cdot)}\bigg)\\
		=&\ 3\sup_{k\in\mathbb{Z}}2^{k}\bigg\|\sum_{i=0}^{\infty}\sum_{j}\chi_{I_{k,i,j}}\bigg\|_{p(\cdot)}\gamma_b\bigg(\bigg\|\sum_{i=0}^{\infty}\sum_{j}\chi_{I_{k,i,j}}\bigg\|_{p(\cdot)}\bigg)\\
		=&\ 3\sup_{k\in\mathbb{Z}}2^{k}\|\chi_{\{s(f)>2^k\}}\|_{p(\cdot)}\gamma_b\big(\|\chi_{\{s(f)>2^k\}}\|_{p(\cdot)}\big)\\
		=&\ 3\sup_{k\in\mathbb{Z}}\sup_{t>0}\chi_{\{2^{k-1}<t\le2^k\}}t\|\chi_{\{s(f)>2^k\}}\|_{p(\cdot)}\gamma_b\big(\|\chi_{\{s(f)>2^k\}}\|_{p(\cdot)}\big)\\
		\lesssim&\ \sup_{k\in\mathbb{Z}}\sup_{t>0}\chi_{\{2^{k-1}<t\le2^k\}}t\|\chi_{\{s(f)>t\}}\|_{p(\cdot)}\gamma_b\big(\|\chi_{\{s(f)>t\}}\|_{p(\cdot)}\big)\\
		=&\ \sup_{t>0}t\|\chi_{\{s(f)>t\}}\|_{p(\cdot)}\gamma_b\big(\|\chi_{\{s(f)>t\}}\|_{p(\cdot)}\big)\\
		=&\ \|s(f)\|_{p(\cdot),\infty,b}.
	\end{align*}
    Hence, we obtain that
	 $$\|f\|_{H_{p(\cdot),q,b}^{s-at,r,1}}\lesssim\|f\|_{H_{p(\cdot),q,b}^s}.$$

	For the converse part, assume that the martingale $f$ has a decomposition (\ref{deco-f}). For an arbitrary $k_0$, let
	 $$T_1=\sum_{k=-\infty}^{k_0-1}\sum_{i=0}^{\infty}\sum_j\mu_{k,i,j}s(a^{k,i,j})\ \ \text{and}\ \ T_2=\sum_{k=k_0}^{\infty}\sum_{i=0}^{\infty}\sum_j\mu_{k,i,j}s(a^{k,i,j}).$$
	It follows from the sublinearity of the conditional square operator $s$ that
	 $$s(f)\le\sum_{k\in\mathbb{Z}}\sum_{i=0}^{\infty}\sum_j\mu_{k,i,j}s(a^{k,i,j})=T_1+T_2.$$
    Let $0<\varepsilon<\min\{\underline{p},q\}$, $1<L<\min\{\frac{r}{p_+},\frac{1}{\varepsilon}\}$ and $0<\sigma<1-\frac{1}{L}$. According to H\"{o}lder's inequality, we have
    \begin{align*}
    	T_1=&\ \sum_{k=-\infty}^{k_0-1}\sum_{i=0}^{\infty}\sum_j\mu_{k,i,j}s(a^{k,i,j})\\
    	=&\ \sum_{k=-\infty}^{k_0-1}2^{k\sigma}2^{-k\sigma}\sum_{i=0}^{\infty}\sum_j\mu_{k,i,j}s(a^{k,i,j})\\
    	\le&\ \bigg(\sum_{k=-\infty}^{k_0-1}2^{k\sigma L'}\bigg)^\frac{1}{L'}\bigg[\sum_{k=-\infty}^{k_0-1}2^{-k\sigma L}\bigg(\sum_{i=0}^{\infty}\sum_{j}\mu_{k,i,j}s(a^{k,i,j})\bigg)^L\bigg]^{\frac{1}{L}}\\
    	\approx&\ 2^{k_0\sigma}\bigg[\sum_{k=-\infty}^{k_0-1}2^{-k\sigma L}\bigg(\sum_{i=0}^{\infty}\sum_{j}\mu_{k,i,j}s(a^{k,i,j})\bigg)^L\bigg]^{\frac{1}{L}},
    \end{align*}
    where $L'$ is the conjugate of $L$. According to Lemma \ref{lem2.1} and Remark \ref{property} (3), we have
    \begin{align*}
    	\|\chi_{\{T_1>2^{k_0}\}}\|_{p(\cdot)}\le & 2^{-k_0L}\|T_1^L\|_{p(\cdot)}\\
    	\lesssim&\ 2^{-k_0L}\Bigg\|2^{k_0\sigma L}\sum_{k=-\infty}^{k_0-1}2^{-k\sigma L}\bigg(\sum_{i=0}^{\infty}\sum_{j}\mu_{k,i,j}s(a^{k,i,j})\bigg)^L\Bigg\|_{p(\cdot)}\\
    	\approx&\ 2^{k_0L(\sigma-1)}\Bigg\|\sum_{k=-\infty}^{k_0-1}2^{kL(1-\sigma)}\bigg(\sum_{i=0}^{\infty}\sum_{j}\|\chi_{I_{k,i,j}}\|_{p(\cdot)}s(a^{k,i,j})\bigg)^L\Bigg\|_{p(\cdot)}\\
    	=&\ 2^{k_0L(\sigma-1)}\Bigg\|\bigg[\sum_{k=-\infty}^{k_0-1}2^{kL(1-\sigma)}\bigg(\sum_{i=0}^{\infty}\sum_{j}\|\chi_{I_{k,i,j}}\|_{p(\cdot)}s(a^{k,i,j})\bigg)^L\bigg]^\varepsilon\Bigg\|_{\frac{p(\cdot)}{\varepsilon}}^{\frac{1}{\varepsilon}}\\
    	\le&\ 2^{k_0L(\sigma-1)}\Bigg\|\sum_{k=-\infty}^{k_0-1}2^{kL\varepsilon(1-\sigma)}\bigg(\sum_{i=0}^{\infty}\sum_{j}\|\chi_{I_{k,i,j}}\|_{p(\cdot)}s(a^{k,i,j})\bigg)^{L\varepsilon}\Bigg\|_{\frac{p(\cdot)}{\varepsilon}}^{\frac{1}{\varepsilon}}\\
    	\le&\ 2^{k_0L(\sigma-1)}\Bigg\|\sum_{k=-\infty}^{k_0-1}2^{kL\varepsilon(1-\sigma)}\sum_{i=0}^{\infty}\sum_{j}\Big(\|\chi_{I_{k,i,j}}\|_{p(\cdot)}s(a^{k,i,j})\Big)^{L\varepsilon}\Bigg\|_{\frac{p(\cdot)}{\varepsilon}}^{\frac{1}{\varepsilon}}\\
    	\lesssim&\ 2^{k_0L(\sigma-1)}\Bigg(\sum_{k=-\infty}^{k_0-1}2^{kL\varepsilon(1-\sigma)}\bigg\|\sum_{i=0}^{\infty}\sum_{j}\Big(\|\chi_{I_{k,i,j}}\|_{p(\cdot)}s(a^{k,i,j})\Big)^{L\varepsilon}\bigg\|_{\frac{p(\cdot)}{\varepsilon}}\Bigg)^{\frac{1}{\varepsilon}}.
    \end{align*}
    Since $(a^{k,i,j})_{k\in\mathbb{Z},i\in\mathbb{N},j}$ are simple $(p(\cdot),r)^s$-atoms, it follows from Lemma \ref{3.3} that
    \begin{align*}
    	 \bigg\|\sum_{i=0}^{\infty}\sum_{j}\Big(\|\chi_{I_{k,i,j}}\|_{p(\cdot)}s(a^{k,i,j})\Big)^{L\varepsilon}\bigg\|_{\frac{p(\cdot)}{\varepsilon}}\lesssim\bigg\|\sum_{i=0}^{\infty}\sum_{j}\chi_{I_{k,i,j}}\bigg\|_{\frac{p(\cdot)}{\varepsilon}}.
    \end{align*}
    Then we have
    \begin{align}\label{gs25}
    	\|\chi_{\{T_1>2^{k_0}\}}\|_{p(\cdot)}\lesssim&\ 2^{k_0L(\sigma-1)}\Bigg(\sum_{k=-\infty}^{k_0-1}2^{kL\varepsilon(1-\sigma)}\bigg\|\sum_{i=0}^{\infty}\sum_{j}\chi_{I_{k,i,j}}\bigg\|_{\frac{p(\cdot)}{\varepsilon}}\Bigg)^{\frac{1}{\varepsilon}}\\\notag
    	=&\ 2^{k_0L(\sigma-1)}\Bigg(\sum_{k=-\infty}^{k_0-1}2^{k\varepsilon(L(1-\sigma)-\delta)}2^{k\varepsilon\delta}\bigg\|\sum_{i=0}^{\infty}\sum_{j}\chi_{I_{k,i,j}}\bigg\|_{p(\cdot)}^{\varepsilon}\Bigg)^{\frac{1}{\varepsilon}},\notag
    \end{align}
    where $\delta>0$ satisfying $1<\delta<L(1-\sigma)$.

    Firstly, we consider the case of $0<q<\infty$. By H\"{o}lder's inequality with $\frac{\varepsilon}{q}+\frac{q-\varepsilon}{q}=1$, we obtain
    \begin{align*}
    	\|\chi_{\{T_1>2^{k_0}\}}\|_{p(\cdot)}\lesssim&\ 2^{k_0L(\sigma-1)}\Bigg(\sum_{k=-\infty}^{k_0-1}2^{k\varepsilon(L(1-\sigma)-\delta)\frac{q}{q-\varepsilon}}\Bigg)^{\frac{q-\varepsilon}{q\varepsilon}}\Bigg(\sum_{k=-\infty}^{k_0-1}2^{kq\delta}\bigg\|\sum_{i=0}^{\infty}\sum_{j}\chi_{I_{k,i,j}}\bigg\|_{p(\cdot)}^{q}\Bigg)^{\frac{1}{q}}\\
    	\lesssim&\ 2^{-k_0\delta}\Bigg(\sum_{k=-\infty}^{k_0-1}2^{k\delta q}\bigg\|\sum_{i=0}^{\infty}\sum_{j}\chi_{I_{k,i,j}}\bigg\|_{p(\cdot)}^{q}\Bigg)^{\frac{1}{q}}\\
    	=&\ \Bigg(\sum_{k=-\infty}^{k_0-1}2^{(k-k_0)\delta q}\bigg\|\sum_{i=0}^{\infty}\sum_{j}\chi_{I_{k,i,j}}\bigg\|_{p(\cdot)}^{q}\Bigg)^{\frac{1}{q}}.
    \end{align*}
    Then this yields that
    \begin{align*}
    	&\ \sum_{k_0\in\mathbb{Z}}2^{k_0q}\|\chi_{\{T_1>2^{k_0}\}}\|_{p(\cdot)}^{q}\gamma_b^{q}\big(\|\chi_{\{T_1>2^{k_0}\}}\|_{p(\cdot)}\big)\\
    	\lesssim&\ \sum_{k_0\in\mathbb{Z}}2^{k_0q}\sum_{k=-\infty}^{k_0-1}2^{(k-k_0)\delta q}\bigg\|\sum_{i=0}^{\infty}\sum_{j}\chi_{I_{k,i,j}}\bigg\|_{p(\cdot)}^{q}\gamma_b^{q}\Bigg[\Bigg(\sum_{k=-\infty}^{k_0-1}2^{(k-k_0)\delta q}\bigg\|\sum_{i=0}^{\infty}\sum_{j}\chi_{I_{k,i,j}}\bigg\|_{p(\cdot)}^{q}\Bigg)^{\frac{1}{q}}\Bigg].
    \end{align*}
    Define $b_1(t)=b\big(t^{\frac{1}{q}}\big)$ for $t\in[1,\infty)$. Set $0<\theta<1$, it follows from Lemma \ref{wuqiongheb} that
    \begin{align*}
    	&\ \sum_{k_0\in\mathbb{Z}}2^{k_0q}\|\chi_{\{T_1>2^{k_0}\}}\|_{p(\cdot)}^{q}\gamma_b^{q}\big(\|\chi_{\{T_1>2^{k_0}\}}\|_{p(\cdot)}\big)\\
    	\lesssim&\ \sum_{k_0\in\mathbb{Z}}2^{k_0q}\sum_{k=-\infty}^{k_0-1}2^{(k-k_0)\delta q}\bigg\|\sum_{i=0}^{\infty}\sum_{j}\chi_{I_{k,i,j}}\bigg\|_{p(\cdot)}^{q}\gamma_{b_1^{q}}\bigg(\sum_{k=-\infty}^{k_0-1}2^{(k-k_0)\delta q}\bigg\|\sum_{i=0}^{\infty}\sum_{j}\chi_{I_{k,i,j}}\bigg\|_{p(\cdot)}^{q}\bigg)\\
    	=&\ \sum_{k_0\in\mathbb{Z}}2^{k_0q}\Bigg[\bigg(\sum_{k=-\infty}^{k_0-1}2^{(k-k_0)\delta q}\bigg\|\sum_{i=0}^{\infty}\sum_{j}\chi_{I_{k,i,j}}\bigg\|_{p(\cdot)}^{q}\bigg)^{\theta}\\
    	&\ \ \ \ \ \ \ \ \ \ \ \ \ \ \ \ \ \ \ \ \ \ \ \ \ \ \ \ \ \ \ \ \ \ \ \ \ \ \ \ \ \times\gamma_{b_1^{\theta q}}\bigg(\sum_{k=-\infty}^{k_0-1}2^{(k-k_0)\delta q}\bigg\|\sum_{i=0}^{\infty}\sum_{j}\chi_{I_{k,i,j}}\bigg\|_{p(\cdot)}^{q}\bigg)\Bigg]^{\frac{1}{\theta}}\\
    	\lesssim&\ \sum_{k_0\in\mathbb{Z}}2^{k_0q}\Bigg[\sum_{k=-\infty}^{k_0-1}2^{(k-k_0)\delta \theta q}\bigg\|\sum_{i=0}^{\infty}\sum_{j}\chi_{I_{k,i,j}}\bigg\|_{p(\cdot)}^{\theta q}\gamma_{b_1^{\theta q}}\bigg(2^{(k-k_0)\delta q}\bigg\|\sum_{i=0}^{\infty}\sum_{j}\chi_{I_{k,i,j}}\bigg\|_{p(\cdot)}^{q}\bigg)\Bigg]^{\frac{1}{\theta}}\\
    	=&\ \sum_{k_0\in\mathbb{Z}}2^{k_0q}\Bigg[\sum_{k=-\infty}^{k_0-1}2^{(k-k_0)\delta \theta q}\bigg\|\sum_{i=0}^{\infty}\sum_{j}\chi_{I_{k,i,j}}\bigg\|_{p(\cdot)}^{\theta q}\gamma_{b}^{\theta q}\bigg(2^{(k-k_0)\delta}\bigg\|\sum_{i=0}^{\infty}\sum_{j}\chi_{I_{k,i,j}}\bigg\|_{p(\cdot)}\bigg)\Bigg]^{\frac{1}{\theta}}.
    \end{align*}
    Let $0<\beta<\frac{\delta-1}{\delta}$. By applying H\"{o}lder's inequality with $1-\theta+\theta=1$, we have
    \begin{align*}
    	&\ \sum_{k_0\in\mathbb{Z}}2^{k_0q}\|\chi_{\{T_1>2^{k_0}\}}\|_{p(\cdot)}^{q}\gamma_b^{q}\big(\|\chi_{\{T_1>2^{k_0}\}}\|_{p(\cdot)}\big)\\
    	\lesssim&\ \sum_{k_0\in\mathbb{Z}}2^{k_0q}\Bigg[\sum_{k=-\infty}^{k_0-1}2^{(k-k_0)\delta \theta q\beta}2^{(k-k_0)(1-\beta)\delta\theta q}\bigg\|\sum_{i=0}^{\infty}\sum_{j}\chi_{I_{k,i,j}}\bigg\|_{p(\cdot)}^{\theta q}\\
    	&\ \ \ \ \ \ \ \ \ \ \ \ \ \ \ \ \ \ \ \ \ \ \ \ \ \ \ \ \ \ \ \ \ \ \ \ \ \ \ \ \ \ \ \ \ \ \ \ \ \ \ \ \ \times\gamma_{b}^{\theta q}\bigg(2^{(k-k_0)\delta}\bigg\|\sum_{i=0}^{\infty}\sum_{j}\chi_{I_{k,i,j}}\bigg\|_{p(\cdot)}\bigg)\Bigg]^{\frac{1}{\theta}}\\
    	\le&\ \sum_{k_0\in\mathbb{Z}}2^{k_0q}\bigg(\sum_{k=-\infty}^{k_0-1}2^{(k-k_0)\delta \theta q\beta/(1-\theta)}\bigg)^{\frac{1-\theta}{\theta}}\sum_{k=-\infty}^{k_0-1}2^{(k-k_0)(1-\beta)\delta q}\bigg\|\sum_{i=0}^{\infty}\sum_{j}\chi_{I_{k,i,j}}\bigg\|_{p(\cdot)}^{q}\\
    	&\ \ \ \ \ \ \ \ \ \ \ \ \ \ \ \ \ \ \ \ \ \ \ \ \ \ \ \ \ \ \ \ \ \ \ \ \ \ \ \ \ \ \ \ \ \ \ \ \ \ \ \ \ \times\gamma_{b}^{q}\bigg(2^{(k-k_0)\delta}\bigg\|\sum_{i=0}^{\infty}\sum_{j}\chi_{I_{k,i,j}}\bigg\|_{p(\cdot)}\bigg)\\
    	\lesssim&\ \sum_{k_0\in\mathbb{Z}}2^{k_0q}\sum_{k=-\infty}^{k_0-1}2^{(k-k_0)(1-\beta)\delta q}\bigg\|\sum_{i=0}^{\infty}\sum_{j}\chi_{I_{k,i,j}}\bigg\|_{p(\cdot)}^{q}\gamma_{b}^{q}\bigg(2^{(k-k_0)\delta}\bigg\|\sum_{i=0}^{\infty}\sum_{j}\chi_{I_{k,i,j}}\bigg\|_{p(\cdot)}\bigg).
    \end{align*}
    Setting $0<z<\frac{\delta-\delta\beta-1}{\delta}$, we obtain that for $k\le k_0-1$,
    \begin{align}\label{gs26}
    	&\ \gamma_b\bigg(2^{(k-k_0)\delta}\bigg\|\sum_{i=0}^{\infty}\sum_{j}\chi_{I_{k,i,j}}\bigg\|_{p(\cdot)}\bigg)\\\notag
    	=&\ \bigg(2^{(k-k_0)\delta}\bigg\|\sum_{i=0}^{\infty}\sum_{j}\chi_{I_{k,i,j}}\bigg\|_{p(\cdot)}\bigg)^{-z}\\\notag
    	&\ \ \ \ \ \ \ \ \ \ \ \ \times\bigg(2^{(k-k_0)\delta}\bigg\|\sum_{i=0}^{\infty}\sum_{j}\chi_{I_{k,i,j}}\bigg\|_{p(\cdot)}\bigg)^{z}\gamma_b\bigg(2^{(k-k_0)\delta}\bigg\|\sum_{i=0}^{\infty}\sum_{j}\chi_{I_{k,i,j}}\bigg\|_{p(\cdot)}\bigg)\\\notag
    	\lesssim&\ \bigg(2^{(k-k_0)\delta}\bigg\|\sum_{i=0}^{\infty}\sum_{j}\chi_{I_{k,i,j}}\bigg\|_{p(\cdot)}\bigg)^{-z}\bigg\|\sum_{i=0}^{\infty}\sum_{j}\chi_{I_{k,i,j}}\bigg\|_{p(\cdot)}^{z}\gamma_b\bigg(\bigg\|\sum_{i=0}^{\infty}\sum_{j}\chi_{I_{k,i,j}}\bigg\|_{p(\cdot)}\bigg)\\\notag
    	=&\ 2^{-(k-k_0)\delta z}\gamma_b\bigg(\bigg\|\sum_{i=0}^{\infty}\sum_{j}\chi_{I_{k,i,j}}\bigg\|_{p(\cdot)}\bigg).\notag
    \end{align}
    Hence it follows from Abel's transformation that
    \begin{align*}
    	&\ \sum_{k_0\in\mathbb{Z}}2^{k_0q}\|\chi_{\{T_1>2^{k_0}\}}\|_{p(\cdot)}^{q}\gamma_b^{q}\big(\|\chi_{\{T_1>2^{k_0}\}}\|_{p(\cdot)}\big)\\
    	\lesssim&\ \sum_{k_0\in\mathbb{Z}}2^{k_0q}\sum_{k=-\infty}^{k_0-1}2^{(k-k_0)(1-\beta)\delta q}\bigg\|\sum_{i=0}^{\infty}\sum_{j}\chi_{I_{k,i,j}}\bigg\|_{p(\cdot)}^{q}2^{-(k-k_0)\delta zq}\gamma_b^q\bigg(\bigg\|\sum_{i=0}^{\infty}\sum_{j}\chi_{I_{k,i,j}}\bigg\|_{p(\cdot)}\bigg)\\
    	=&\ \sum_{k\in\mathbb{Z}}2^{k(1-\beta-z)\delta q}\bigg\|\sum_{i=0}^{\infty}\sum_{j}\chi_{I_{k,i,j}}\bigg\|_{p(\cdot)}^{q}\gamma_b^q\bigg(\bigg\|\sum_{i=0}^{\infty}\sum_{j}\chi_{I_{k,i,j}}\bigg\|_{p(\cdot)}\bigg)\sum_{k_0=k+1}^{\infty}2^{k_0q[1+\delta(\beta-1)+\delta z]}\\
    	\lesssim&\ \sum_{k\in\mathbb{Z}}2^{kq}\bigg\|\sum_{i=0}^{\infty}\sum_{j}\chi_{I_{k,i,j}}\bigg\|_{p(\cdot)}^{q}\gamma_b^q\bigg(\bigg\|\sum_{i=0}^{\infty}\sum_{j}\chi_{I_{k,i,j}}\bigg\|_{p(\cdot)}\bigg).
    \end{align*}

    Next we discuss the case of $q=\infty$. 
    According to (\ref{gs25}) and H\"{o}lder's inequality with $\varepsilon+1-\varepsilon=1$, we obtain that
    \begin{align*}
    	&\ \|\chi_{\{T_1>2^{k_0}\}}\|_{p(\cdot)}\\
        \lesssim&\ 2^{k_0L(\sigma-1)}\Bigg(\sum_{k=-\infty}^{k_0-1}2^{k\varepsilon(L(1-\sigma)-\delta)/{(1-\varepsilon)}}\Bigg)^{(1-\varepsilon)/\varepsilon}\sum_{k=-\infty}^{k_0-1}2^{k\delta}\bigg\|\sum_{i=0}^{\infty}\sum_{j}\chi_{I_{k,i,j}}\bigg\|_{p(\cdot)}\\
    	\lesssim&\ 2^{-k_0\delta}\sum_{k=-\infty}^{k_0-1}2^{k\delta}\bigg\|\sum_{i=0}^{\infty}\sum_{j}\chi_{I_{k,i,j}}\bigg\|_{p(\cdot)}=\sum_{k=-\infty}^{k_0-1}2^{(k-k_0)\delta}\bigg\|\sum_{i=0}^{\infty}\sum_{j}\chi_{I_{k,i,j}}\bigg\|_{p(\cdot)}.
    \end{align*}
    Then it follows from Lemma \ref{wuqiongheb} and H\"{o}lder's inequality with $1-\theta+\theta=1$ that
    \begin{align*}
    	&\ \sup_{k_0\in\mathbb{Z}}2^{k_0}\|\chi_{\{T_1>2^{k_0}\}}\|_{p(\cdot)}\gamma_{b}(\|\chi_{\{T_1>2^{k_0}\}}\|_{p(\cdot)})\\
    	\lesssim&\ \sup_{k_0\in\mathbb{Z}}2^{k_0}\sum_{k=-\infty}^{k_0-1}2^{(k-k_0)\delta}\bigg\|\sum_{i=0}^{\infty}\sum_{j}\chi_{I_{k,i,j}}\bigg\|_{p(\cdot)}\gamma_{b}\Bigg(\sum_{k=-\infty}^{k_0-1}2^{(k-k_0)\delta}\bigg\|\sum_{i=0}^{\infty}\sum_{j}\chi_{I_{k,i,j}}\bigg\|_{p(\cdot)}\Bigg)\\
    	\lesssim&\ \sup_{k_0\in\mathbb{Z}}2^{k_0}\Bigg[\sum_{k=-\infty}^{k_0-1}2^{(k-k_0)\delta\theta}\bigg\|\sum_{i=0}^{\infty}\sum_{j}\chi_{I_{k,i,j}}\bigg\|_{p(\cdot)}^{\theta}\gamma_{b^{\theta}}\bigg(2^{(k-k_0)\delta}\bigg\|\sum_{i=0}^{\infty}\sum_{j}\chi_{I_{k,i,j}}\bigg\|_{p(\cdot)}\bigg)\Bigg]^{\frac{1}{\theta}}\\
    	=&\ \sup_{k_0\in\mathbb{Z}}2^{k_0}\Bigg[\sum_{k=-\infty}^{k_0-1}2^{\beta\delta\theta(k-k_0)}2^{(k-k_0)(1-\beta)\delta\theta}\bigg\|\sum_{i=0}^{\infty}\sum_{j}\chi_{I_{k,i,j}}\bigg\|_{p(\cdot)}^{\theta}\\
    	&\ \ \ \ \ \ \ \ \ \ \ \ \ \ \ \ \ \ \ \ \ \ \ \ \ \ \ \ \ \ \ \ \ \ \ \times\gamma_{b}^{\theta}\bigg(2^{(k-k_0)\delta}\bigg\|\sum_{i=0}^{\infty}\sum_{j}\chi_{I_{k,i,j}}\bigg\|_{p(\cdot)}\bigg)\Bigg]^{\frac{1}{\theta}}\\
    	\le&\ \sup_{k_0\in\mathbb{Z}}2^{k_0}\Bigg(\sum_{k=-\infty}^{k_0-1}2^{\beta\delta\theta(k-k_0)/(1-\theta)}\Bigg)^{\frac{1-\theta}{\theta}}\sum_{k=-\infty}^{k_0-1}2^{(k-k_0)(1-\beta)\delta}\\
    	&\ \ \ \ \ \ \ \ \ \ \ \ \ \ \ \ \ \ \ \ \ \ \ \ \ \ \ \ \ \ \ \ \ \ \ \times\bigg\|\sum_{i=0}^{\infty}\sum_{j}\chi_{I_{k,i,j}}\bigg\|_{p(\cdot)}\gamma_{b}\bigg(2^{(k-k_0)\delta}\bigg\|\sum_{i=0}^{\infty}\sum_{j}\chi_{I_{k,i,j}}\bigg\|_{p(\cdot)}\bigg)\\
    	\lesssim&\ \sup_{k_0\in\mathbb{Z}}2^{k_0}\sum_{k=-\infty}^{k_0-1}2^{(k-k_0)(1-\beta)\delta}\bigg\|\sum_{i=0}^{\infty}\sum_{j}\chi_{I_{k,i,j}}\bigg\|_{p(\cdot)}\gamma_{b}\bigg(2^{(k-k_0)\delta}\bigg\|\sum_{i=0}^{\infty}\sum_{j}\chi_{I_{k,i,j}}\bigg\|_{p(\cdot)}\bigg).
    \end{align*}
    Furthermore, by \eqref{gs26},
    \begin{align*}
    	&\ \sup_{k_0\in\mathbb{Z}}2^{k_0}\|\chi_{\{T_1>2^{k_0}\}}\|_{p(\cdot)}\gamma_{b}(\|\chi_{\{T_1>2^{k_0}\}}\|_{p(\cdot)})\\
    	\lesssim&\ \sup_{k_0\in\mathbb{Z}}2^{k_0}\sum_{k=-\infty}^{k_0-1}2^{(k-k_0)(1-\beta)\delta}\bigg\|\sum_{i=0}^{\infty}\sum_{j}\chi_{I_{k,i,j}}\bigg\|_{p(\cdot)}2^{-(k-k_0)\delta z}\gamma_b\bigg(\bigg\|\sum_{i=0}^{\infty}\sum_{j}\chi_{I_{k,i,j}}\bigg\|_{p(\cdot)}\bigg)\\
    	=&\ \sup_{k\in\mathbb{Z}}2^{k\delta(1-\beta-z)}\bigg\|\sum_{i=0}^{\infty}\sum_{j}\chi_{I_{k,i,j}}\bigg\|_{p(\cdot)}\gamma_b\bigg(\bigg\|\sum_{i=0}^{\infty}\sum_{j}\chi_{I_{k,i,j}}\bigg\|_{p(\cdot)}\bigg)\sum_{k_0=k+1}^{\infty}2^{k_0(1-\delta+\beta\delta+\delta z)}\\
    	\lesssim&\ \sup_{k\in\mathbb{Z}}2^{k}\bigg\|\sum_{i=0}^{\infty}\sum_{j}\chi_{I_{k,i,j}}\bigg\|_{p(\cdot)}\gamma_b\bigg(\bigg\|\sum_{i=0}^{\infty}\sum_{j}\chi_{I_{k,i,j}}\bigg\|_{p(\cdot)}\bigg).
    \end{align*}
    To sum up, for $0<q\le\infty$,
    $$\|T_1\|_{p(\cdot),q,b}\lesssim\Bigg\|\Bigg\{2^k\bigg\|\sum_{i=0}^{\infty}\sum_{j}\chi_{I_{k,i,j}}\bigg\|_{p(\cdot)}\gamma_b\bigg(\bigg\|\sum_{i=0}^{\infty}\sum_{j}\chi_{I_{k,i,j}}\bigg\|_{p(\cdot)}\bigg)\Bigg\}_{k\in\mathbb{Z}}\Bigg\|_{l_q}.$$

	Now, we estimate $T_2$. Since $s(a^{k,i,j})=0$ on $\{\tau_k=\infty\}$, we have
	 $$\{T_2>2^{k_0}\}\subset\{T_2>0\}\subset\bigcup_{k=k_0}^{\infty}\{s(a^{k,i,j})>0\}\subset\bigcup_{k=k_0}^{\infty}\bigcup_{i=0}^{\infty}\bigcup_{j}I_{k,i,j}.$$
	Let $0<m<p_-$ and $0<\varepsilon<\min\{\underline{p},q\}$. It follows from Lemma \ref{wuqiongbianb} that
	\begin{align*}
		&\ 2^{k_0m\varepsilon}\|\chi_{\{T_2>2^{k_0}\}}\|_{p(\cdot)}^m\gamma_b^m\big(\|\chi_{\{T_2>2^{k_0}\}}\|_{p(\cdot)}\big)\\\notag
		\lesssim&\ 2^{k_0m\varepsilon}\bigg\|\sum_{k=k_0}^{\infty}\sum_{i=0}^{\infty}\sum_{j}\chi_{I_{k,i,j}}\bigg\|_{p(\cdot)}^m\gamma_b^m\bigg(\bigg\|\sum_{k=k_0}^{\infty}\sum_{i=0}^{\infty}\sum_{j}\chi_{I_{k,i,j}}\bigg\|_{p(\cdot)}\bigg)\\\notag
		\lesssim&\ 2^{k_0m\varepsilon}\sum_{k=k_0}^{\infty}\bigg\|\sum_{i=0}^{\infty}\sum_{j}\chi_{I_{k,i,j}}\bigg\|_{p(\cdot)}^m\gamma_b^m\bigg(\bigg\|\sum_{i=0}^{\infty}\sum_{j}\chi_{I_{k,i,j}}\bigg\|_{p(\cdot)}\bigg)\\\notag
		\lesssim&\ \sum_{k=k_0}^{\infty}2^{km\varepsilon}\bigg\|\sum_{i=0}^{\infty}\sum_{j}\chi_{I_{k,i,j}}\bigg\|_{p(\cdot)}^m\gamma_b^m\bigg(\bigg\|\sum_{i=0}^{\infty}\sum_{j}\chi_{I_{k,i,j}}\bigg\|_{p(\cdot)}\bigg).
	\end{align*}
    Obviously, according to Lemma \ref{belong}, we obtain
    \begin{align*}
    	 \|T_2\|_{p(\cdot),q,b}\lesssim\Bigg\|\Bigg\{2^k\bigg\|\sum_{i=0}^{\infty}\sum_{j}\chi_{I_{k,i,j}}\bigg\|_{p(\cdot)}\gamma_b\bigg(\bigg\|\sum_{i=0}^{\infty}\sum_{j}\chi_{I_{k,i,j}}\bigg\|_{p(\cdot)}\bigg)\Bigg\}_{k\in\mathbb{Z}}\Bigg\|_{l_q}.
    \end{align*}
    Hence, it follows from Lemma \ref{fanshusanjiao} that
    \begin{align*}
   	    \|f\|_{H_{p(\cdot),q,b}^s}=&\|s(f)\|_{p(\cdot),q,b}\le\|T_1+T_2\|_{p(\cdot),q,b}\\
   	    \lesssim&\ \|T_1\|_{p(\cdot),q,b}+\|T_2\|_{p(\cdot),q,b}\\
   	    \lesssim&\ \Bigg\|\Bigg\{2^k\bigg\|\sum_{i=0}^{\infty}\sum_{j}\chi_{I_{k,i,j}}\bigg\|_{p(\cdot)}\gamma_b\bigg(\bigg\|\sum_{i=0}^{\infty}\sum_{j}\chi_{I_{k,i,j}}\bigg\|_{p(\cdot)}\bigg)\Bigg\}_{k\in\mathbb{Z}}\Bigg\|_{l_q}\\
   	    \approx&\ \Bigg\|\Bigg\{\bigg\|\sum_{i=0}^{\infty}\sum_{j}\mu_{k,i,j}\|\chi_{I_{k,i,j}}\|_{p(\cdot)}^{-1}\chi_{I_{k,i,j}}\bigg\|_{p(\cdot)}\gamma_b\bigg(\bigg\|\sum_{i=0}^{\infty}\sum_{j}\chi_{I_{k,i,j}}\bigg\|_{p(\cdot)}\bigg)\Bigg\}_{k\in\mathbb{Z}}\Bigg\|_{l_q}\\
   	    \approx&\ \|f\|_{H_{p(\cdot),q,b}^{s-at,r,1}}.
        \end{align*}
    The proof is complete.
\end{proof}

\begin{rem}\label{rem4.2}
	The sum $\sum\limits_{k=m}^{l}\sum\limits_{i=0}^{\infty}\sum\limits_{j}\mu_{k,i,j}a^{k,i,j}$ converges to $f$ in $H_{p(\cdot),q,b}^s$ as $m\rightarrow-\infty$,  $l\rightarrow\infty$, when $p(\cdot)\in\mathcal{P}(\cdot)$, $0<q<\infty$ and $b$ is a slowly varying function.
	Indeed,
	\begin{align*}
		 \sum_{k=m}^{l}\sum_{i=0}^{\infty}\sum_{j}\mu_{k,i,j}a^{k,i,j}=&\sum_{k=m}^{l}(f^{\tau_{k+1}}-f^{\tau_k})\sum_{i=0}^{\infty}\sum_{j}\chi_{I_{k,i,j}}\\
		=&\ \sum_{k=m}^{l}(f^{\tau_{k+1}}-f^{\tau_k})\chi_{\{\tau_k<\infty\}}=f^{\tau_{l+1}}-f^{\tau_m}.
	\end{align*}
	By the sublinearity of the conditional square operator $s$ and Lemma $\ref{fanshusanjiao}$, we have
	\begin{align*}
		&\ \bigg\|f-\sum_{k=m}^{l}\sum_{i=0}^{\infty}\sum_{j}\mu_{k,i,j}a^{k,i,j}\bigg\|_{H_{p(\cdot),q,b}^s}=\big\|s(f-f^{\tau_{l+1}}+f^{\tau_m})\big\|_{p(\cdot),q,b}\\
		\le&\ \big\|s(f-f^{\tau_{l+1}})+s(f^{\tau_m})\big\|_{p(\cdot),q,b}\lesssim\big\|s(f-f^{\tau_{l+1}})\big\|_{p(\cdot),q,b}+\big\|s(f^{\tau_m})\big\|_{p(\cdot),q,b}.
	\end{align*}
	Since $s(f-f^{\tau_{l+1}})^2=s(f)^2-s(f^{\tau_{l+1}})^2$, then $s(f-f^{\tau_{l+1}})\le s(f)$, $s(f^{\tau_{m}})\le s(f)$ and $s(f-f^{\tau_{l+1}})\rightarrow0,\ s(f^{\tau_{m}})\rightarrow0$ a.e. as $m\rightarrow-\infty, \; \,l\rightarrow\infty$. Hence, by dominated convergence theorem $($Lemma $\ref{dct})$, we have
	$$\big\|s(f-f^{\tau_{l+1}})\big\|_{p(\cdot),q,b}\rightarrow0,\ \ \big\|s(f^{\tau_m})\big\|_{p(\cdot),q,b}\rightarrow0,\ \ \text{as}\ m\rightarrow-\infty,\;\,l\rightarrow\infty,$$
	which means that $$\bigg\|f-\sum_{k=m}^{l}\sum_{i=0}^{\infty}\sum_{j}\mu_{k,i,j}a^{k,i,j}\bigg\|_{H_{p(\cdot),q,b}^s}\rightarrow0,\ \ \text{as}\ m\rightarrow-\infty,\;\,l\rightarrow\infty.$$
	Furthermore, for each $k\in\mathbb{Z}$, $a^{k,i,j}=(a_n^{k,i,j})_{n\ge0}$ is $L_2$-bounded. If $p_+<2$, we have $H_2^s\subset H_{p(\cdot),q,b}^s$. Thus $H_2^s=L_2$ is dense in $H_{p(\cdot),q,b}^s$ when $p_+<2$. Similarly, $\sum\limits_{k=m}^{l}\sum\limits_{i=0}^{\infty}\sum\limits_{j}\mu_{k,i,j}a^{k,i,j}$ converges to $f$ in $\mathcal{H}_{p(\cdot),\infty,b}^s$ as $m\rightarrow-\infty$, $l\rightarrow\infty$ and $H_2^s=L_2$ is dense in $\mathcal{H}_{p(\cdot),\infty,b}^s$ when $p_+<2$.
\end{rem}

Next, we discuss the atomic decomposition theorems in $H_{p(\cdot),q,b}^M, \; H_{p(\cdot),q,b}^S, \; \mathcal{P}_{p(\cdot),q,b}$ and $\mathcal{Q}_{p(\cdot),q,b}$, respectively.

\begin{thm}\label{ad2}
	Let $p(\cdot)\in\mathcal{P}(\Omega)$ satisfy $(\ref{gs3})$, $0<q\le\infty$, $max\{p_+,1\}<r\le\infty$ and let $b$ be a slowly varying function. If $\{\mathcal{F}_n\}_{n\ge0}$ is regular, then
	$$H_{p(\cdot),q,b}^S=H_{p(\cdot),q,b}^{s-at,r,2}\ \ \text{and}\ \ H_{p(\cdot),q,b}^M=H_{p(\cdot),q,b}^{s-at,r,3}$$
	with equivalent $($quasi$)$-norms.
\end{thm}
\begin{proof}
	We only give the proof of $H_{p(\cdot),q,b}^M$, since it can be proved similarly for $H_{p(\cdot),q,b}^S$. Take $f\in H_{p(\cdot),q,b}^M$. Define the stopping times with respect to $\{\mathcal{F}_n\}_{n\ge0}$ by
	$$v_k:=\inf\{n\in\mathbb{N}:|f_n|>2^k\},\ \ \ k\in\mathbb{Z}.$$
	Define $$F_n^k:=\Big\{\mathbb{E}_{n-1}\big(\chi_{\{v_k=n\}}\big)\ge\frac{1}{\mathcal{R}}\Big\}\in\mathcal{F}_{n-1},\ \ \ n\in\mathbb{N}.$$
	Since $\{\mathcal{F}_n\}_{n\ge0}$ is regular, we obtain that
	$$\{v_k=n\}\subset F_n^k\ \ \text{and}\ \ \mathbb{P}(F_n^k)\le\mathcal{R}\mathbb{P}(v_k=n).$$
	Define another set of stopping times by
	$$\tau_k(\omega):=\inf\{n\in\mathbb{N}:\omega\in F_{n+1}^k\}.$$
	It is obvious that $\{\tau_k\}_{k\in\mathbb{Z}}$ is a nondecreasing sequence of stopping times and $\tau_k(\omega)\le n-1$ when $v_k(\omega)=n$. In other words, $\tau_k<v_k$ on the set $\{v_k\neq\infty\}$.
	Moreover, by Lemma \ref{l11},
	\begin{align}\label{lem3.6}
		 \|\chi_{\{\tau_k<\infty\}}\|_{p(\cdot)}\lesssim\|\chi_{\{v_k<\infty\}}\|_{p(\cdot)}=\|\chi_{\{M(f)>2^k\}}\|_{p(\cdot)}\le2^{-k}\|M(f)\|_{p(\cdot)}.
	\end{align}
	Then $\|\chi_{\{\tau_k<\infty\}}\|_{p(\cdot)}\rightarrow0$ as $k\rightarrow\infty$, which implies that
	$$\lim\limits_{k\rightarrow\infty}\mathbb{P}(\tau_k=\infty)=1.$$
	Hence $\lim\limits_{k\rightarrow\infty}\tau_k=\infty\ a.e.$ and $\lim\limits_{k\rightarrow\infty}f_n^{\tau_k}=f_n\ a.e.$ for $n\in\mathbb{N}$. Define $\mu_{k,i,j}$ and $a_n^{k,i,j}$ in the same as in the proof of Theorem \ref{ad1}. By the definition of $\tau_k$, \eqref{e12} holds and
	 $$M\big((a_n^{k,i,j})_{n\ge0}\big)\le\frac{M(f_n^{\tau_{k+1}})+M(f_n^{\tau_k})}{\mu_{k,i,j}}\le\frac{2^{k+1}+2^k}{\mu_{k,i,j}}=\frac{1}{\|\chi_{I_{k,i,j}}\|_{p(\cdot)}}.$$
	Thus $(a_n^{k,i,j})_{n\ge0}$ is an $L_2$-bounded martingale. Moreover, there exists $a^{k,i,j}\in L_2$ such that $\mathbb{E}_n(a^{k,i,j})=a_n^{k,i,j}$ and
	 $$\|M(a^{k,i,j})\|_r\le\mathbb{P}(I_{k,i,j})^{1/r}\frac{1}{\|\chi_{I_{k,i,j}}\|_{p(\cdot)}}.$$
	Therefore, $a^{k,i,j}$ is a simple $(p(\cdot),r)^M$-atom and (\ref{deco-f}) holds.
	It is easy to get from (\ref{lem3.6}) that
   $$\bigg\|\sum_{i=0}^{\infty}\sum_{j}\chi_{I_{k,i,j}}\bigg\|_{p(\cdot)}\lesssim\bigg\|\sum_{i=0}^{\infty}\sum_{j}\chi_{\overline{I_{k,i,j}}}\bigg\|_{p(\cdot)/\varepsilon}^{1/\varepsilon}=\big\|\chi_{\{M(f)>2^k\}}\big\|_{p(\cdot)/\varepsilon}^{1/\varepsilon}=\big\|\chi_{\{M(f)>2^k\}}\big\|_{p(\cdot)}.$$
    Since $t\gamma_b(t)$ is equivalent to a nondecreasing function, for $0<q<\infty$, we obtain
    \begin{align*}
    	&\ \sum_{k\in\mathbb{Z}}\bigg\|\sum_{i=0}^{\infty}\sum_{j}\mu_{k,i,j}\|\chi_{I_{k,i,j}}\|_{p(\cdot)}^{-1}\chi_{I_{k,i,j}}\bigg\|_{p(\cdot)}^q\gamma_b^q\bigg(\bigg\|\sum_{i=0}^{\infty}\sum_{j}\chi_{I_{k,i,j}}\bigg\|_{p(\cdot)}\bigg)\\
    	=&\ 3^q\sum_{k\in\mathbb{Z}}2^{kq}\bigg\|\sum_{i=0}^{\infty}\sum_{j}\chi_{I_{k,i,j}}\bigg\|_{p(\cdot)}^q\gamma_b^q\bigg(\bigg\|\sum_{i=0}^{\infty}\sum_{j}\chi_{I_{k,i,j}}\bigg\|_{p(\cdot)}\bigg)\\
    	\lesssim&\ \sum_{k\in\mathbb{Z}}2^{kq}\|\chi_{\{M(f)>2^k\}}\|_{p(\cdot)}^q\gamma_b^q\big(\|\chi_{\{M(f)>2^k\}}\|_{p(\cdot)}\big)\\
    	\approx&\ \|M(f)\|_{p(\cdot),q,b}^q=\|f\|_{H_{p(\cdot),q,b}^M}^q.
    \end{align*}
    Similarly, for the case of $q=\infty$,
    \begin{align*}
    	&\ \bigg\|\sum_{i=0}^{\infty}\sum_{j}\mu_{k,i,j}\|\chi_{I_{k,i,j}}\|_{p(\cdot)}^{-1}\chi_{I_{k,i,j}}\bigg\|_{p(\cdot)}\gamma_b\bigg(\bigg\|\sum_{i=0}^{\infty}\sum_{j}\chi_{I_{k,i,j}}\bigg\|_{p(\cdot)}\bigg)\\
    	=&\ 3\cdot2^{k}\bigg\|\sum_{i=0}^{\infty}\sum_{j}\chi_{I_{k,i,j}}\bigg\|_{p(\cdot)}\gamma_b\bigg(\bigg\|\sum_{i=0}^{\infty}\sum_{j}\chi_{I_{k,i,j}}\bigg\|_{p(\cdot)}\bigg)\\
    	\lesssim&\ 2^k\|\chi_{\{M(f)>2^k\}}\|_{p(\cdot)}\gamma_b\big(\|\chi_{\{M(f)>2^k\}}\|_{p(\cdot)}\big)\\
    	\lesssim&\ \|M(f)\|_{p(\cdot),\infty,b}=\|f\|_{H_{p(\cdot),\infty,b}^M}.
    \end{align*}
    Hence, we have $\|f\|_{H_{p(\cdot),q,b}^{s-at,r,3}}\lesssim\|f\|_{H_{p(\cdot).q.b}^M}$.

    We omit the proof of the converse part, since it is similar to the proof of Theorem \ref{ad1}. Thus, the proof is complete.
\end{proof}

The next lemma is used in the proof of the previous theorem and we refer the readers to Lemma 3.6 in \cite{jwzw} for its proof.
\begin{lem}\label{l11}
	Let $p(\cdot)\in\mathcal{P}(\Omega)$ satisfy $(\ref{gs3})$ and $\{\mathcal{F}_n\}_{n\ge0}$ be regular. Taking the same stopping times $\tau_{k}$ and $v_k$ as in the proof of Theorem $\ref{ad2}$, we have
	 $$\|\chi_{\{\tau_k<\infty\}}\|_{p(\cdot)}\lesssim\|\chi_{\{v_k<\infty\}}\|_{p(\cdot)}.$$
\end{lem}

\begin{rem}
	Let $p(\cdot)\in\mathcal{P}(\cdot)$ and $b$ be a slowly varying function. Similarly to Remark $\ref{rem4.2}$, we know that the sum $\sum\limits_{k=m}^{l}\sum\limits_{i=0}^{\infty}\sum\limits_{j}\mu_{k,i,j}a^{k,i,j}$ converges to $f$ in $H_{p(\cdot),q,b}^M$ $($resp. $H_{p(\cdot),q,b}^S)$ if $0<q<\infty$ and in $\mathcal{H}^M_{p(\cdot),\infty,b}$ $($resp. $\mathcal{H}_{p(\cdot),\infty,b}^S)$, as $m\rightarrow-\infty$, $l\rightarrow\infty$.
\end{rem}

Next we will prove that the conclusion in Theorem \ref{ad1} is also valid for $\mathcal{Q}_{p(\cdot),q,b}$ and $\mathcal{P}_{p(\cdot),q,b}$ with simple $(p(\cdot),\infty)^S$-atoms and simple $(p(\cdot),\infty)^M$-atoms, respectively.

\begin{thm}\label{ad3}
	Let $p(\cdot)\in\mathcal{P}(\Omega)$ satisfy $(\ref{gs3})$, $0<q\le\infty$ and $b$ be a slowly varying function. Then
	$$\mathcal{Q}_{p(\cdot),q,b}=H_{p(\cdot),q,b}^{s-at,\infty,2}\ \ \text{and}\ \ \mathcal{P}_{p(\cdot),q,b}=H_{p(\cdot),q,b}^{s-at,\infty,3}$$
	with equivalent $($quasi$)$-norms.
\end{thm}
\begin{proof}
	The proof is similar to that of Theorem \ref{ad1}, so we only sketch it. Let $f=(f_n)_{n\ge0}\in\mathcal{Q}_{p(\cdot),q,b}$ (resp. $\mathcal{P}_{p(\cdot),q,b}$). For $k\in\mathbb{Z}$, define the stopping times
	$$\tau_k:=\inf\{n\in\mathbb{N}:\lambda_n>2^k\},$$
	where $(\lambda_n)_{n\ge0}$ is the sequence in the definition of $\mathcal{Q}_{p(\cdot),q,b}$ (resp. $\mathcal{P}_{p(\cdot),q,b}$). Let the definitions of $\mu_{k,i,j}$ and $a_n^{k,i,j}$ be the same as in the proof of Theorem \ref{ad1}. Then we get
	 $$f_n=\sum_{k\in\mathbb{Z}}\sum_{i=0}^{n-1}\sum_{j}\mu_{k,i,j}\mathbb{E}_n(a^{k,i,j}),$$
	where $(a^{k,i,j})_{k\in\mathbb{Z},i\in\mathbb{N},j}$ is a sequence of simple $(p(\cdot),\infty)^S$-atoms (resp. simple $(p(\cdot),\infty)^M$-atoms). Moreover,  $H_{p(\cdot),q,b}^{s-at,\infty,2}\supset \mathcal{Q}_{p(\cdot),q,b}\ \ (\text{resp.}\ H_{p(\cdot),q,b}^{s-at,\infty,3}\supset \mathcal{P}_{p(\cdot),q,b})$.
	
	For the converse part, let
	 $$\lambda_n=\sum_{k\in\mathbb{Z}}\sum_{i=0}^{n-1}\sum_{j}\mu_{k,i,j}\chi_{I_{k,i,j}}\|S(a^{k,i,j})\|_\infty$$
	$$\bigg(\mathrm{resp}.\ \lambda_n=\sum_{k\in\mathbb{Z}}\sum_{i=0}^{n-1}\sum_{j}\mu_{k,i,j}\chi_{I_{k,i,j}}\|M(a^{k,i,j})\|_\infty\bigg),$$
	where $(a^{k,i,j})_{k\in\mathbb{Z},i\in\mathbb{N},j}$ is a sequence of simple $(p(\cdot),\infty)^S$-atoms (resp. simple $(p(\cdot),\infty)^M$-atoms) associated with $(I_{k,i,j})_{k\in\mathbb{Z},i\in\mathbb{N},j}$ and $\mu_{k,i,j}=3\cdot2\|\chi_{I_{k,i,j}}\|_{p(\cdot)}$.
	Then $(\lambda_n)_{n\ge0}\in\Lambda_{p(\cdot),q,b}$ with $S_{n+1}(f)\le\lambda_n$ (resp. $M_{n+1}(f)\le\lambda_n$) for any $n\ge0$. For any given integer $k_0$, let
	$$\lambda_\infty=\lambda_\infty^{(1)}+\lambda_\infty^{(2)},$$
	where
	 $$\lambda_\infty^{(1)}=\sum_{k=-\infty}^{k_0-1}\sum_{i=0}^{\infty}\sum_{j}\mu_{k,i,j}\chi_{I_{k,i,j}}\|S(a^{k,i,j})\|_\infty,$$
	 $$\lambda_\infty^{(2)}=\sum_{k=k_0}^{\infty}\sum_{i=0}^{\infty}\sum_{j}\mu_{k,i,j}\chi_{I_{k,i,j}}\|S(a^{k,i,j})\|_\infty$$
	$$\bigg(\text{resp.}\ \lambda_\infty^{(1)}=\sum_{k=-\infty}^{k_0-1}\sum_{i=0}^{\infty}\sum_{j}\mu_{k,i,j}\chi_{I_{k,i,j}}\|M(a^{k,i,j})\|_\infty,$$
	$$\ \ \ \ \ \ \ \ \  \lambda_\infty^{(2)}=\sum_{k=k_0}^{\infty}\sum_{i=0}^{\infty}\sum_{j}\mu_{k,i,j}\chi_{I_{k,i,j}}\|M(a^{k,i,j})\|_\infty\bigg).$$
	By replacing $T_1$ and $T_2$ in the proof of Theorem \ref{ad1} with $\lambda_\infty^{(1)}$ and $\lambda_\infty^{(2)}$, respectively, we obtain $f\in \mathcal{Q}_{p(\cdot),q,b}$ $(\text{resp}.\ \mathcal{P}_{p(\cdot),q,b})$ and $\mathcal{Q}_{p(\cdot),q,b}\supset H_{p(\cdot),q,b}^{s-at,\infty,2}\ \ (\text{resp.}\ \mathcal{P}_{p(\cdot),q,b}\supset H_{p(\cdot),q,b}^{s-at,\infty,3})$.
\end{proof}

\begin{rem}
	Let $p(\cdot)\in\mathcal{P}(\cdot)$, $0<q<\infty$ and $b$ be a slowly varying function.
	The sum $\sum\limits_{k=m}^{l}\sum\limits_{i=0}^{\infty}\sum\limits_{j}\mu_{k,i,j}a^{k,i,j}$ converges to $f$ in $\mathcal{Q}_{p(\cdot),q,b}$ $(\text{resp.}\ \mathcal{P}_{p(\cdot),q,b})$ as $m\rightarrow-\infty$, $l\rightarrow\infty$. The proof is similar to Remark $\ref{rem4.2}$ and left to the reader.
\end{rem}

Especially, if we set $p(\cdot)\equiv p$ in Theorems \ref{ad1}, \ref{ad2} and \ref{ad3}, then the following result holds.
\begin{coro}\label{coro1}
	Let $0<p<\infty$, $0<q\le\infty$ and $b$ be a slowing varying function. We obtain the atomic decomposition theorems of martingale Hardy-Lorentz-Karamata spaces $H_{p,q,b}^s$, $\mathcal{Q}_{p,q,b}$ and $\mathcal{P}_{p,q,b}$, respectively. Moreover, if $\{\mathcal{F}_n\}_{n\ge0}$ is regular, we have the atomic decompositions of $H_{p,q,b}^M$ and $H_{p,q,b}^S$, respectively.
\end{coro}
\begin{rem}
	The atomic decomposition theorems of $H_{p,q,b}^s$, $\mathcal{Q}_{p,q,b}$ and $\mathcal{P}_{p,q,b}$ can be found in \cite{jxz} under the condition that $0<p<\infty$, $0<q\le\infty$ and $b$ is a nondecreasing slowly varying function. Notice that the slowly varying function $b$ is not necessarily nondecreasing in Corollary $\ref{coro1}$. Hence our results extend the atomic decomposition theorems of $H_{p,q,b}^s$, $\mathcal{Q}_{p,q,b}$ and $\mathcal{P}_{p,q,b}$ in \cite{jxz} and also give the atomic decomposition theorems of $H_{p,q,b}^M$ and $H_{p,q,b}^S$.
\end{rem}

Particularly, we have the following conclusion if $b\equiv1$ in Theorems \ref{ad1}, \ref{ad2} and \ref{ad3}.
\begin{coro}\label{c1}
	Let $p(\cdot)\in\mathcal{P}(\Omega)$ satisfy $(\ref{gs3})$ and $0<q\le\infty$. We get the atomic decomposition theorems of the variable martingale Hardy-Lorentz spaces $H_{p(\cdot),q}^s$, $\mathcal{P}_{p(\cdot),q}$ and $\mathcal{Q}_{p(\cdot),q}$, respectively. Furthermore, if $\{\mathcal{F}_n\}_{n\ge0}$ is regular, we have the atomic decomposition theorems of $H_{p(\cdot),q}^M$ and $H_{p(\cdot),q}^S$.
\end{coro}
\begin{rem}
	Corollary $\ref{c1}$ was proved in \cite{jwwz,jwzw}.
\end{rem}

\subsection{Atomic Decompositions via  Atoms}
In this subsection, we give a slightly different atomic decomposition for the five  variable martingale Hardy-Lorentz-Karamata spaces.

\begin{defi} \label{1defi1}
	Let $p(\cdot) \in \mathcal{P}(\Omega)$. A measurable function $a$ is called a $(p(\cdot), \infty)^s$-atom $($resp. $(p(\cdot),\infty)^S$-atom \text{or}\;$(p(\cdot),\infty)^M$-atom$)$ if there exists a stopping time $\tau$ such that
	
	$(1)$ $\mathbb{E}_{n}(a)=0$,  $\forall\; n\leq \tau$;
	
	$(2)$ $\|s(a)\|_\infty\ (resp.\ \|S(a)\|_\infty \;\text{or}\; \|M(a)\|_\infty)\leq\left\|\chi_{\{\tau<\infty\}}\right\|^{-1}_{p(\cdot)}.$
\end{defi}

\begin{defi}
	Let $p(\cdot)\in\mathcal{P}(\Omega)$, $0<q\leq \infty$ and $b$ be a slowly varying function.
	The atomic martingale Lorentz-Karamata space with variable exponent $H_{p(\cdot),q,b}^{at,\infty,1}$ $\big(\text{resp.}\; H_{p(\cdot),q,b}^{at,\infty,2}$ $\text{or}\; H_{p(\cdot),q,b}^{at,\infty,3}\big)$ is defined as the set of all $f=(f_n)_{n\geq 0} \in \mathcal{M} $ such that for each $n\geq0$,
	\begin{eqnarray*}
		f_n= \sum_{ k \in \mathbb{Z} } \mu_{k} \mathbb{E}_n (a^{k}),
	\end{eqnarray*}
    where $(a^k)_{k\in\mathbb{Z}}$ is a sequence of $(p(\cdot),\infty)^s$-atoms $\big(\text{resp.}\; (p(\cdot),\infty)^S\text{-atoms}\; \text{or}\; (p(\cdot),\infty)^M\text{-atoms}\big)$ associated with stopping times $(\tau_k)_{k\in\mathbb{Z}}$ and
    $$\Big\|\big\{\mu_k\gamma_b\big(\|\chi_{\{\tau_k<\infty\}}
\|_{p(\cdot)}\big)\big\}_{k\in\mathbb{Z}}\Big\|_{l_q}<\infty$$
  with $\mu_k=3\cdot2^k\|\chi_{\{\tau_k<\infty\}}\|_{p(\cdot)}$ for each $k\in\mathbb{Z}$.
	Endow $H_{p(\cdot),q,b}^{at,\infty,1}$ $\big(\text{resp.}\; H_{p(\cdot),q,b}^{at,\infty,2}\; \text{or}\; H_{p(\cdot),q,b}^{at,\infty,3}\big)$ with the $($quasi$)$-norm
	\begin{align*}
		\|f\|_{H_{p(\cdot),q,b}^{at,\infty,1}}\ \big(\text{resp.}\ \|f\|_{H_{p(\cdot),q,b}^{at,\infty,2}}\; \text{or}\; \|f\|_{H_{p(\cdot),q,b}^{at,\infty,3}}\big)=\inf
		 \Big\|\big\{\mu_k\gamma_b\big(\|\chi_{\{\tau_k<\infty\}}
\|_{p(\cdot)}\big)\big\}_{k\in\mathbb{Z}}\Big\|_{l_q},
	\end{align*}
	where the infimum is taken over all decompositions of $f$ as above.
\end{defi}

With the help of the atom $a$ defined in Definition \ref{1defi1}, we obtain another kind of atomic decompositions for $H_{p(\cdot),q,b}^s$, $\mathcal{Q}_{p(\cdot),q,b}$, $\mathcal{P}_{p(\cdot),q,b}$, $H_{p(\cdot),q,b}^M$ and $H_{p(\cdot),q,b}^S$. The proofs are similar to those of Theorems \ref{ad1}, \ref{ad2} and \ref{ad3}, so we leave them to the reader.

\begin{thm}\label{ad11}
	Let $p(\cdot)\in\mathcal{P}(\Omega)$ satisfy $(\ref{gs3})$, $0<q\le\infty$ and let $b$ be a slowly varying function. Then
	 $$H_{p(\cdot),q,b}^s=H_{p(\cdot),q,b}^{at,\infty,1},\quad\mathcal{Q}_{p(\cdot),q,b}=H_{p(\cdot),q,b}^{at,\infty,2}\quad\text{and}\quad \mathcal{P}_{p(\cdot),q,b}=H_{p(\cdot),q,b}^{at,\infty,3}$$
with equivalent $($quasi$)$-norms.
\end{thm}

\begin{thm}
	Let $p(\cdot)\in\mathcal{P}(\Omega)$ satisfy $(\ref{gs3})$, $0<q\le\infty$ and let $b$ be a slowly varying function. If $\{\mathcal{F}_n\}_{n\ge0}$ is regular, then
	$$H_{p(\cdot),q,b}^S=H_{p(\cdot),q,b}^{at,\infty,2}\quad \text{and}\quad H_{p(\cdot),q,b}^M=H_{p(\cdot),q,b}^{at,\infty,3}$$
	with equivalent $($quasi$)$-norms.
\end{thm}

\section{Martingale Inequalities} \label{s4}
In this section, we show some martingale inequalities between different variable martingale Hardy-Lorentz-Karamata spaces. The method we use here is to establish a sufficient condition for a $\sigma$-sublinear operator to be bounded from the variable martingale Hardy-Lorentz-Karamata spaces to the variable Lorentz-Karamata spaces.

Let us recall that an operator $T:X\rightarrow Y$ is said to be $\sigma$-sublinear, if for any constant $c$,
$$\bigg|T\bigg(\sum_{k=1}^{\infty}f_k\bigg)\bigg|\le\sum_{k=1}^{\infty}|T(f_k)|\ \ \text{and}\ \ |T(cf)|=|c||T(f)|,$$
where $X$ is a martingale space and $Y$ is a measurable function space.

\begin{lem}\label{5.1}
	Let $p(\cdot)\in\mathcal{P}(\Omega)$ satisfy $(\ref{gs3})$, $0<q\le\infty$ and let $b$ be a slowly varying function. Suppose that $\max\{p_+,1\}<r<\infty$. If $T:H^s_r\rightarrow L_r$ is a bounded $\sigma$-sublinear operator and
	$$\{|T(a)|>0\}\subset I$$
	for every simple $(p(\cdot),r)^s$-atom $a$ associated with $I$, then for $f\in H_{p(\cdot),q,b}^s$,
	$$\|T(f)\|_{p(\cdot),q,b}\lesssim\|f\|_{H_{p(\cdot),q,b}^s}.$$
\end{lem}
\begin{proof}
	Let $f\in H_{p(\cdot),q,b}^s$. According to Theorem \ref{ad1}, there exist a sequence of simple $(p(\cdot),r)^s$-atoms $(a^{k,i,j})_{k\in\mathbb{Z},i\in\mathbb{N},j}$ and $\mu_{k,i,j}=3\cdot2^k\|\chi_{I_{k,i,j}}\|_{p(\cdot)}$ such that
	 $$f=\sum_{k\in\mathbb{Z}}\sum_{i=0}^{\infty}\sum_{j}\mu_{k,i,j}a^{k,i,j}$$
	and
	 $$\Bigg\|\Bigg\{\bigg\|\sum_{i=0}^{\infty}\sum_{j}\mu_{k,i,j}\|\chi_{I_{k,i,j}}\|_{p(\cdot)}^{-1}\chi_{I_{k,i,j}}\bigg\|_{p(\cdot)}\gamma_b\bigg(\bigg\|\sum_{i=0}^{\infty}\sum_{j}\chi_{I_{k,i,j}}\bigg\|_{p(\cdot)}\bigg)\Bigg\}_{k\in\mathbb{Z}}\Bigg\|_{l_q}\lesssim\|f\|_{H_{p(\cdot),q,b}^s}.$$
	For an arbitrary integer $k_0$, set
	 $$D_1:=\sum_{k=-\infty}^{k_0-1}\sum_{i=0}^{\infty}\sum_{j}\mu_{k,i,j}|T(a^{k,i,j})|\quad\text{and}\quad D_2:=\sum_{k=k_0}^{\infty}\sum_{i=0}^{\infty}\sum_{j}\mu_{k,i,j}|T(a^{k,i,j})|.$$
	Then by the $\sigma$-sublinearity of $T$, there is
	 $$|T(f)|\le\sum_{k\in\mathbb{Z}}\sum_{i=0}^{\infty}\sum_{j}\mu_{k,i,j}|T(a^{k,i,j})|=D_1+D_2.$$
	Let $0<\varepsilon<\min\{\underline{p},q\}$, $1<L<\min\{\frac{r}{p_+},\frac{1}{\varepsilon}\}$ and $0<\sigma<1-\frac{1}{L}$. It follows from H\"{o}lder's inequality that
	\begin{align*}
		D_1=&\ \sum_{k=-\infty}^{k_0-1}\sum_{i=0}^{\infty}\sum_j\mu_{k,i,j}|T(a^{k,i,j})|\\
		\le&\ \bigg(\sum_{k=-\infty}^{k_0-1}2^{k\sigma L'}\bigg)^\frac{1}{L'}\bigg[\sum_{k=-\infty}^{k_0-1}2^{-k\sigma L}\bigg(\sum_{i=0}^{\infty}\sum_{j}\mu_{k,i,j}|T(a^{k,i,j})|\bigg)^L\bigg]^{\frac{1}{L}}\\
		\approx&\ 2^{k_0\sigma}\bigg[\sum_{k=-\infty}^{k_0-1}2^{-k\sigma L}\bigg(\sum_{i=0}^{\infty}\sum_{j}\mu_{k,i,j}|T(a^{k,i,j})|\bigg)^L\bigg]^{\frac{1}{L}},
	\end{align*}
	where $L'$ is the conjugate of $L$. According to Lemma \ref{lem2.1} and Remark \ref{property} (3), we have
	\begin{align*}
		\|\chi_{\{D_1>2^{k_0}\}}\|_{p(\cdot)}\le & 2^{-k_0L}\|D_1^L\|_{p(\cdot)}\\
		\lesssim&\ 2^{-k_0L}\Bigg\|2^{k_0\sigma L}\sum_{k=-\infty}^{k_0-1}2^{-k\sigma L}\bigg(\sum_{i=0}^{\infty}\sum_{j}\mu_{k,i,j}|T(a^{k,i,j})|\bigg)^L\Bigg\|_{p(\cdot)}\\
		\approx&\ 2^{k_0L(\sigma-1)}\Bigg\|\sum_{k=-\infty}^{k_0-1}2^{kL(1-\sigma)}\bigg(\sum_{i=0}^{\infty}\sum_{j}\|\chi_{I_{k,i,j}}\|_{p(\cdot)}|T(a^{k,i,j})|\bigg)^L\Bigg\|_{p(\cdot)}\\
		=&\ 2^{k_0L(\sigma-1)}\Bigg\|\bigg[\sum_{k=-\infty}^{k_0-1}2^{kL(1-\sigma)}\bigg(\sum_{i=0}^{\infty}\sum_{j}\|\chi_{I_{k,i,j}}\|_{p(\cdot)}|T(a^{k,i,j})|\bigg)^L\bigg]^\varepsilon\Bigg\|_{\frac{p(\cdot)}{\varepsilon}}^{\frac{1}{\varepsilon}}\\
		\le&\ 2^{k_0L(\sigma-1)}\Bigg\|\sum_{k=-\infty}^{k_0-1}2^{kL\varepsilon(1-\sigma)}\bigg(\sum_{i=0}^{\infty}\sum_{j}\|\chi_{I_{k,i,j}}\|_{p(\cdot)}|T(a^{k,i,j})|\bigg)^{L\varepsilon}\Bigg\|_{\frac{p(\cdot)}{\varepsilon}}^{\frac{1}{\varepsilon}}\\
		\le&\ 2^{k_0L(\sigma-1)}\Bigg\|\sum_{k=-\infty}^{k_0-1}2^{kL\varepsilon(1-\sigma)}\sum_{i=0}^{\infty}\sum_{j}\Big(\|\chi_{I_{k,i,j}}\|_{p(\cdot)}|T(a^{k,i,j})|\Big)^{L\varepsilon}\Bigg\|_{\frac{p(\cdot)}{\varepsilon}}^{\frac{1}{\varepsilon}}\\
		\lesssim&\ 2^{k_0L(\sigma-1)}\Bigg(\sum_{k=-\infty}^{k_0-1}2^{kL\varepsilon(1-\sigma)}\bigg\|\sum_{i=0}^{\infty}\sum_{j}\Big(\|\chi_{I_{k,i,j}}\|_{p(\cdot)}|T(a^{k,i,j})|\Big)^{L\varepsilon}\bigg\|_{\frac{p(\cdot)}{\varepsilon}}\Bigg)^{\frac{1}{\varepsilon}}.
	\end{align*}
	Since $T:H^s_r\rightarrow L_r$ is bounded, we have that for a simple $(p(\cdot),r)^s$-atom $a^{k,i,j}$, $$\|T(a^{k,i,j})\|_r\lesssim\|s(a^{k,i,j})\|_r\le\frac{\|\chi_{I_{k,i,j}}\|_r}{\|\chi_{I_{k,i,j}}\|_{p(\cdot)}}.$$
	Then using Lemma \ref{3.3}, we obtain
	 $$\bigg\|\sum_{i=0}^{\infty}\sum_{j}\Big(\|\chi_{I_{k,i,j}}\|_{p(\cdot)}|T(a^{k,i,j})|\Big)^{L\varepsilon}\bigg\|_{\frac{p(\cdot)}{\varepsilon}}\lesssim\Bigg\|\sum_{i=0}^{\infty}\sum_{j}\chi_{I_{k,i,j}}\Bigg\|_{\frac{p(\cdot)}{\varepsilon}}.$$
	Hence we get that
	\begin{align}\label{gs61}
		\|\chi_{\{D_1>2^{k_0}\}}\|_{p(\cdot)}\lesssim&\ 2^{k_0L(\sigma-1)}\Bigg(\sum_{k=-\infty}^{k_0-1}2^{kL\varepsilon(1-\sigma)}\bigg\|\sum_{i=0}^{\infty}\sum_{j}\chi_{I_{k,i,j}}\bigg\|_{\frac{p(\cdot)}{\varepsilon}}\Bigg)^{\frac{1}{\varepsilon}}\\\notag
		=&\ 2^{k_0L(\sigma-1)}\Bigg(\sum_{k=-\infty}^{k_0-1}2^{kL\varepsilon(1-\sigma)}\bigg\|\sum_{i=0}^{\infty}\sum_{j}\chi_{I_{k,i,j}}\bigg\|_{p(\cdot)}^{\varepsilon}\Bigg)^{\frac{1}{\varepsilon}}\\\notag
		=&\ 2^{k_0L(\sigma-1)}\Bigg(\sum_{k=-\infty}^{k_0-1}2^{k\varepsilon(L(1-\sigma)-\delta)}2^{k\varepsilon\delta}\bigg\|\sum_{i=0}^{\infty}\sum_{j}\chi_{I_{k,i,j}}\bigg\|_{p(\cdot)}^{\varepsilon}\Bigg)^{\frac{1}{\varepsilon}},\notag
	\end{align}
	where $\delta>0$ satisfies $1<\delta<L(1-\sigma)$.
	
	Firstly, we consider the case of $0<q<\infty$. By H\"{o}lder's inequality with $\frac{\varepsilon}{q}+\frac{q-\varepsilon}{q}=1$, we obtain
	\begin{align*}
		\|\chi_{\{D_1>2^{k_0}\}}\|_{p(\cdot)}\lesssim&\ 2^{k_0L(\sigma-1)}\Bigg(\sum_{k=-\infty}^{k_0-1}2^{k\varepsilon(L(1-\sigma)-\delta)\frac{q}{q-\varepsilon}}\Bigg)^{\frac{q-\varepsilon}{q\varepsilon}}\Bigg(\sum_{k=-\infty}^{k_0-1}2^{kq\delta}\bigg\|\sum_{i=0}^{\infty}\sum_{j}\chi_{I_{k,i,j}}\bigg\|_{p(\cdot)}^{q}\Bigg)^{\frac{1}{q}}\\
		\lesssim&\ 2^{-k_0\delta}\Bigg(\sum_{k=-\infty}^{k_0-1}2^{k\delta q}\bigg\|\sum_{i=0}^{\infty}\sum_{j}\chi_{I_{k,i,j}}\bigg\|_{p(\cdot)}^{q}\Bigg)^{\frac{1}{q}}\\
		=&\ \Bigg(\sum_{k=-\infty}^{k_0-1}2^{(k-k_0)\delta q}\bigg\|\sum_{i=0}^{\infty}\sum_{j}\chi_{I_{k,i,j}}\bigg\|_{p(\cdot)}^{q}\Bigg)^{\frac{1}{q}}.
	\end{align*}
	This yields that
	\begin{align*}
		&\ \sum_{k_0\in\mathbb{Z}}2^{k_0q}\|\chi_{\{D_1>2^{k_0}\}}\|_{p(\cdot)}^{q}\gamma_b^{q}\big(\|\chi_{\{D_1>2^{k_0}\}}\|_{p(\cdot)}\big)\\
		\lesssim&\ \sum_{k_0\in\mathbb{Z}}2^{k_0q}\sum_{k=-\infty}^{k_0-1}2^{(k-k_0)\delta q}\bigg\|\sum_{i=0}^{\infty}\sum_{j}\chi_{I_{k,i,j}}\bigg\|_{p(\cdot)}^{q}\gamma_b^{q}\Bigg[\Bigg(\sum_{k=-\infty}^{k_0-1}2^{(k-k_0)\delta q}\bigg\|\sum_{i=0}^{\infty}\sum_{j}\chi_{I_{k,i,j}}\bigg\|_{p(\cdot)}^{q}\Bigg)^{\frac{1}{q}}\Bigg].
	\end{align*}
	Define $b_1(t)=b\big(t^{\frac{1}{q}}\big)$ for $t\in[1,\infty)$ and let $0<\theta<1$. It follows from Lemma \ref{wuqiongheb} that
	\begin{align*}
		&\ \sum_{k_0\in\mathbb{Z}}2^{k_0q}\|\chi_{\{D_1>2^{k_0}\}}\|_{p(\cdot)}^{q}\gamma_b^{q}\big(\|\chi_{\{D_1>2^{k_0}\}}\|_{p(\cdot)}\big)\\
		\lesssim&\ \sum_{k_0\in\mathbb{Z}}2^{k_0q}\sum_{k=-\infty}^{k_0-1}2^{(k-k_0)\delta q}\bigg\|\sum_{i=0}^{\infty}\sum_{j}\chi_{I_{k,i,j}}\bigg\|_{p(\cdot)}^{q}\gamma_{b_1^{q}}\bigg(\sum_{k=-\infty}^{k_0-1}2^{(k-k_0)\delta q}\bigg\|\sum_{i=0}^{\infty}\sum_{j}\chi_{I_{k,i,j}}\bigg\|_{p(\cdot)}^{q}\bigg)\\
		=&\ \sum_{k_0\in\mathbb{Z}}2^{k_0q}\Bigg[\bigg(\sum_{k=-\infty}^{k_0-1}2^{(k-k_0)\delta q}\bigg\|\sum_{i=0}^{\infty}\sum_{j}\chi_{I_{k,i,j}}\bigg\|_{p(\cdot)}^{q}\bigg)^{\theta}\\
		&\ \ \ \ \ \ \ \ \ \ \ \ \ \ \ \ \ \ \ \ \ \ \ \ \ \ \ \ \ \ \ \ \ \ \ \ \ \ \ \ \ \times\gamma_{b_1^{\theta q}}\bigg(\sum_{k=-\infty}^{k_0-1}2^{(k-k_0)\delta q}\bigg\|\sum_{i=0}^{\infty}\sum_{j}\chi_{I_{k,i,j}}\bigg\|_{p(\cdot)}^{q}\bigg)\Bigg]^{\frac{1}{\theta}}\\
		\lesssim&\ \sum_{k_0\in\mathbb{Z}}2^{k_0q}\Bigg[\sum_{k=-\infty}^{k_0-1}2^{(k-k_0)\delta \theta q}\bigg\|\sum_{i=0}^{\infty}\sum_{j}\chi_{I_{k,i,j}}\bigg\|_{p(\cdot)}^{\theta q}\gamma_{b_1^{\theta q}}\bigg(2^{(k-k_0)\delta q}\bigg\|\sum_{i=0}^{\infty}\sum_{j}\chi_{I_{k,i,j}}\bigg\|_{p(\cdot)}^{q}\bigg)\Bigg]^{\frac{1}{\theta}}\\
		=&\ \sum_{k_0\in\mathbb{Z}}2^{k_0q}\Bigg[\sum_{k=-\infty}^{k_0-1}2^{(k-k_0)\delta \theta q}\bigg\|\sum_{i=0}^{\infty}\sum_{j}\chi_{I_{k,i,j}}\bigg\|_{p(\cdot)}^{\theta q}\gamma_{b}^{\theta q}\bigg(2^{(k-k_0)\delta}\bigg\|\sum_{i=0}^{\infty}\sum_{j}\chi_{I_{k,i,j}}\bigg\|_{p(\cdot)}\bigg)\Bigg]^{\frac{1}{\theta}}.
	\end{align*}
	Let $0<\beta<\frac{\delta-1}{\delta}$. By applying H\"{o}lder's inequality with $1-\theta+\theta=1$, we have
	\begin{align*}
		&\ \sum_{k_0\in\mathbb{Z}}2^{k_0q}\|\chi_{\{D_1>2^{k_0}\}}\|_{p(\cdot)}^{q}\gamma_b^{q}\big(\|\chi_{\{D_1>2^{k_0}\}}\|_{p(\cdot)}\big)\\
		\lesssim&\ \sum_{k_0\in\mathbb{Z}}2^{k_0q}\Bigg[\sum_{k=-\infty}^{k_0-1}2^{(k-k_0)\delta \theta q\beta}2^{(k-k_0)(1-\beta)\delta\theta q}\bigg\|\sum_{i=0}^{\infty}\sum_{j}\chi_{I_{k,i,j}}\bigg\|_{p(\cdot)}^{\theta q}\\
		&\ \ \ \ \ \ \ \ \ \ \ \ \ \ \ \ \ \ \ \ \ \ \ \ \ \ \ \ \ \ \ \ \ \ \ \ \ \times\gamma_{b}^{\theta q}\bigg(2^{(k-k_0)\delta}\bigg\|\sum_{i=0}^{\infty}\sum_{j}\chi_{I_{k,i,j}}\bigg\|_{p(\cdot)}\bigg)\Bigg]^{\frac{1}{\theta}}\\
		\le&\ \sum_{k_0\in\mathbb{Z}}2^{k_0q}\bigg(\sum_{k=-\infty}^{k_0-1}2^{(k-k_0)\delta \theta q\beta/(1-\theta)}\bigg)^{\frac{1-\theta}{\theta}}\sum_{k=-\infty}^{k_0-1}2^{(k-k_0)(1-\beta)\delta q}\bigg\|\sum_{i=0}^{\infty}\sum_{j}\chi_{I_{k,i,j}}\bigg\|_{p(\cdot)}^{q}\\
		&\ \ \ \ \ \ \ \ \ \ \ \ \ \ \ \ \ \ \ \ \ \ \ \ \ \ \ \ \ \ \ \ \ \ \ \ \ \times\gamma_{b}^{q}\bigg(2^{(k-k_0)\delta}\bigg\|\sum_{i=0}^{\infty}\sum_{j}\chi_{I_{k,i,j}}\bigg\|_{p(\cdot)}\bigg)\\
		\lesssim&\ \sum_{k_0\in\mathbb{Z}}2^{k_0q}\sum_{k=-\infty}^{k_0-1}2^{(k-k_0)(1-\beta)\delta q}\bigg\|\sum_{i=0}^{\infty}\sum_{j}\chi_{I_{k,i,j}}\bigg\|_{p(\cdot)}^{q}\gamma_{b}^{q}\bigg(2^{(k-k_0)\delta}\bigg\|\sum_{i=0}^{\infty}\sum_{j}\chi_{I_{k,i,j}}\bigg\|_{p(\cdot)}\bigg).
	\end{align*}
	Set $0<z<\frac{\delta-\delta\beta-1}{\delta}$, it follows from (\ref{gs26}) that
	\begin{align*}
		 \gamma_b\bigg(2^{(k-k_0)\delta}\bigg\|\sum_{i=0}^{\infty}\sum_{j}\chi_{I_{k,i,j}}\bigg\|_{p(\cdot)}\bigg)\lesssim2^{-(k-k_0)\delta z}\gamma_b\bigg(\bigg\|\sum_{i=0}^{\infty}\sum_{j}\chi_{I_{k,i,j}}\bigg\|_{p(\cdot)}\bigg).
	\end{align*}
	Hence we get the following by Abel's transformation:
	\begin{align*}
		&\ \sum_{k_0\in\mathbb{Z}}2^{k_0q}\|\chi_{\{D_1>2^{k_0}\}}\|_{p(\cdot)}^{q}\gamma_b^{q}\big(\|\chi_{\{D_1>2^{k_0}\}}\|_{p(\cdot)}\big)\\
		\lesssim&\ \sum_{k_0\in\mathbb{Z}}2^{k_0q}\sum_{k=-\infty}^{k_0-1}2^{(k-k_0)(1-\beta)\delta q}\bigg\|\sum_{i=0}^{\infty}\sum_{j}\chi_{I_{k,i,j}}\bigg\|_{p(\cdot)}^{q}2^{-(k-k_0)\delta zq}\gamma_b^q\bigg(\bigg\|\sum_{i=0}^{\infty}\sum_{j}\chi_{I_{k,i,j}}\bigg\|_{p(\cdot)}\bigg)\\
		=&\ \sum_{k\in\mathbb{Z}}2^{k(1-\beta-z)\delta q}\bigg\|\sum_{i=0}^{\infty}\sum_{j}\chi_{I_{k,i,j}}\bigg\|_{p(\cdot)}^{q}\gamma_b^q\bigg(\bigg\|\sum_{i=0}^{\infty}\sum_{j}\chi_{I_{k,i,j}}\bigg\|_{p(\cdot)}\bigg)\sum_{k_0=k+1}^{\infty}2^{k_0q[1+\delta(\beta-1)+\delta z]}\\
		\lesssim&\ \sum_{k\in\mathbb{Z}}2^{kq}\bigg\|\sum_{i=0}^{\infty}\sum_{j}\chi_{I_{k,i,j}}\bigg\|_{p(\cdot)}^{q}\gamma_b^q\bigg(\bigg\|\sum_{i=0}^{\infty}\sum_{j}\chi_{I_{k,i,j}}\bigg\|_{p(\cdot)}\bigg).
	\end{align*}

	Next we discuss the case of $q=\infty$. 
	According to (\ref{gs61}) and H\"{o}lder's inequality with $\varepsilon+1-\varepsilon=1$, we obtain that
	\begin{align*}
		\|\chi_{\{D_1>2^{k_0}\}}\|_{p(\cdot)}
		\lesssim&\ 2^{k_0L(\sigma-1)}\Bigg(\sum_{k=-\infty}^{k_0-1}2^{k\varepsilon(L(1-\sigma)-\delta)/{(1-\varepsilon)}}\Bigg)^{(1-\varepsilon)/\varepsilon}\sum_{k=-\infty}^{k_0-1}2^{k\delta}\bigg\|\sum_{i=0}^{\infty}\sum_{j}\chi_{I_{k,i,j}}\bigg\|_{p(\cdot)}\\
		\lesssim&\ 2^{-k_0\delta}\sum_{k=-\infty}^{k_0-1}2^{k\delta}\bigg\|\sum_{i=0}^{\infty}\sum_{j}\chi_{I_{k,i,j}}\bigg\|_{p(\cdot)}=\sum_{k=-\infty}^{k_0-1}2^{(k-k_0)\delta}\bigg\|\sum_{i=0}^{\infty}\sum_{j}\chi_{I_{k,i,j}}\bigg\|_{p(\cdot)}.
	\end{align*}
	Then it follows from Lemma \ref{wuqiongheb} and H\"{o}lder's inequality with $1-\theta+\theta=1$ that
	\begin{align*}
		&\ \sup_{k_0\in\mathbb{Z}}2^{k_0}\|\chi_{\{D_1>2^{k_0}\}}\|_{p(\cdot)}\gamma_{b}(\|\chi_{\{D_1>2^{k_0}\}}\|_{p(\cdot)})\\
		\lesssim&\ \sup_{k_0\in\mathbb{Z}}2^{k_0}\sum_{k=-\infty}^{k_0-1}2^{(k-k_0)\delta}\bigg\|\sum_{i=0}^{\infty}\sum_{j}\chi_{I_{k,i,j}}\bigg\|_{p(\cdot)}\gamma_{b}\Bigg(\sum_{k=-\infty}^{k_0-1}2^{(k-k_0)\delta}\bigg\|\sum_{i=0}^{\infty}\sum_{j}\chi_{I_{k,i,j}}\bigg\|_{p(\cdot)}\Bigg)\\
		=&\ \sup_{k_0\in\mathbb{Z}}2^{k_0}\Bigg[\Bigg(\sum_{k=-\infty}^{k_0-1}2^{(k-k_0)\delta}\bigg\|\sum_{i=0}^{\infty}\sum_{j}\chi_{I_{k,i,j}}\bigg\|_{p(\cdot)}\Bigg)^{\theta}\\
		&\ \ \ \ \ \ \ \ \ \ \ \ \ \ \ \ \ \ \ \ \ \ \ \ \ \ \ \ \ \ \ \ \ \ \ \times\gamma_{b^{\theta}}\Bigg(\sum_{k=-\infty}^{k_0-1}2^{(k-k_0)\delta}\bigg\|\sum_{i=0}^{\infty}\sum_{j}\chi_{I_{k,i,j}}\bigg\|_{p(\cdot)}\Bigg)\Bigg]^{\frac{1}{\theta}}\\
		\lesssim&\ \sup_{k_0\in\mathbb{Z}}2^{k_0}\Bigg[\sum_{k=-\infty}^{k_0-1}2^{(k-k_0)\delta\theta}\bigg\|\sum_{i=0}^{\infty}\sum_{j}\chi_{I_{k,i,j}}\bigg\|_{p(\cdot)}^{\theta}\gamma_{b^{\theta}}\bigg(2^{(k-k_0)\delta}\bigg\|\sum_{i=0}^{\infty}\sum_{j}\chi_{I_{k,i,j}}\bigg\|_{p(\cdot)}\bigg)\Bigg]^{\frac{1}{\theta}}\\
		=&\ \sup_{k_0\in\mathbb{Z}}2^{k_0}\Bigg[\sum_{k=-\infty}^{k_0-1}2^{\beta\delta\theta(k-k_0)}2^{(k-k_0)(1-\beta)\delta\theta}\bigg\|\sum_{i=0}^{\infty}\sum_{j}\chi_{I_{k,i,j}}\bigg\|_{p(\cdot)}^{\theta}\\
		&\ \ \ \ \ \ \ \ \ \ \ \ \ \ \ \ \ \ \ \ \ \ \ \ \ \ \ \ \ \ \ \ \ \ \ \times\gamma_{b}^{\theta}\bigg(2^{(k-k_0)\delta}\bigg\|\sum_{i=0}^{\infty}\sum_{j}\chi_{I_{k,i,j}}\bigg\|_{p(\cdot)}\bigg)\Bigg]^{\frac{1}{\theta}}\\
		\le&\ \sup_{k_0\in\mathbb{Z}}2^{k_0}\Bigg(\sum_{k=-\infty}^{k_0-1}2^{\beta\delta\theta(k-k_0)/(1-\theta)}\Bigg)^{\frac{1-\theta}{\theta}}\sum_{k=-\infty}^{k_0-1}2^{(k-k_0)(1-\beta)\delta}\\
		&\ \ \ \ \ \ \ \ \ \ \ \ \ \ \ \ \ \ \ \ \ \ \ \ \ \ \ \ \ \ \ \ \ \ \ \times\bigg\|\sum_{i=0}^{\infty}\sum_{j}\chi_{I_{k,i,j}}\bigg\|_{p(\cdot)}\gamma_{b}\bigg(2^{(k-k_0)\delta}\bigg\|\sum_{i=0}^{\infty}\sum_{j}\chi_{I_{k,i,j}}\bigg\|_{p(\cdot)}\bigg)\\
		\lesssim&\ \sup_{k_0\in\mathbb{Z}}2^{k_0}\sum_{k=-\infty}^{k_0-1}2^{(k-k_0)(1-\beta)\delta}\bigg\|\sum_{i=0}^{\infty}\sum_{j}\chi_{I_{k,i,j}}\bigg\|_{p(\cdot)}\gamma_{b}\bigg(2^{(k-k_0)\delta}\bigg\|\sum_{i=0}^{\infty}\sum_{j}\chi_{I_{k,i,j}}\bigg\|_{p(\cdot)}\bigg).
	\end{align*}
Inequality (\ref{gs26}) implies
	\begin{align*}
		&\ \sup_{k_0\in\mathbb{Z}}2^{k_0}\|\chi_{\{D_1>2^{k_0}\}}\|_{p(\cdot)}\gamma_{b}(\|\chi_{\{D_1>2^{k_0}\}}\|_{p(\cdot)})\\
		\lesssim&\ \sup_{k_0\in\mathbb{Z}}2^{k_0}\sum_{k=-\infty}^{k_0-1}2^{(k-k_0)(1-\beta)\delta}\bigg\|\sum_{i=0}^{\infty}\sum_{j}\chi_{I_{k,i,j}}\bigg\|_{p(\cdot)}2^{-(k-k_0)\delta z}\gamma_b\bigg(\bigg\|\sum_{i=0}^{\infty}\sum_{j}\chi_{I_{k,i,j}}\bigg\|_{p(\cdot)}\bigg)\\
		=&\ \sup_{k\in\mathbb{Z}}2^{k\delta(1-\beta-z)}\bigg\|\sum_{i=0}^{\infty}\sum_{j}\chi_{I_{k,i,j}}\bigg\|_{p(\cdot)}\gamma_b\bigg(\bigg\|\sum_{i=0}^{\infty}\sum_{j}\chi_{I_{k,i,j}}\bigg\|_{p(\cdot)}\bigg)\sum_{k_0=k+1}^{\infty}2^{k_0(1-\delta+\beta\delta+\delta z)}\\
		\lesssim&\ \sup_{k\in\mathbb{Z}}2^{k}\bigg\|\sum_{i=0}^{\infty}\sum_{j}\chi_{I_{k,i,j}}\bigg\|_{p(\cdot)}\gamma_b\bigg(\bigg\|\sum_{i=0}^{\infty}\sum_{j}\chi_{I_{k,i,j}}\bigg\|_{p(\cdot)}\bigg).
	\end{align*}
	Therefore, we conclude that
	 $$\|D_1\|_{p(\cdot),q,b}\lesssim\Bigg\|\Bigg\{2^k\bigg\|\sum_{i=0}^{\infty}\sum_{j}\chi_{I_{k,i,j}}\bigg\|_{p(\cdot)}\gamma_b\bigg(\bigg\|\sum_{i=0}^{\infty}\sum_{j}\chi_{I_{k,i,j}}\bigg\|_{p(\cdot)}\bigg)\Bigg\}_{k\in\mathbb{Z}}\Bigg\|_{l_q}.$$
	
	Now, we estimate $D_2$. Since $s(a^{k,i,j})=0$ on the set $\bigcup\limits_{i=0}^{\infty}\bigcup\limits_{j}I_{k,i,j}$, we have
	 $$\{D_2>2^{k_0}\}\subset\{D_2>0\}\subset\bigcup_{k=k_0}^{\infty}\{s(a^{k,i,j})>0\}\subset\bigcup_{k=k_0}^{\infty}\bigcup_{i=0}^{\infty}\bigcup_{j}I_{k,i,j}.$$
	Let $0<m<\underline{p}$ and $\varepsilon$ be the same as before. It follows from Lemma \ref{wuqiongbianb} that
	\begin{align*}
		&\ 2^{k_0m\varepsilon}\|\chi_{\{D_2>2^{k_0}\}}\|_{p(\cdot)}^m\gamma_b^m\big(\|\chi_{\{D_2>2^{k_0}\}}\|_{p(\cdot)}\big)\\\notag
		\lesssim&\ 2^{k_0m\varepsilon}\bigg\|\sum_{k=k_0}^{\infty}\sum_{i=0}^{\infty}\sum_{j}\chi_{I_{k,i,j}}\bigg\|_{p(\cdot)}^m\gamma_b^m\bigg(\bigg\|\sum_{k=k_0}^{\infty}\sum_{i=0}^{\infty}\sum_{j}\chi_{I_{k,i,j}}\bigg\|_{p(\cdot)}\bigg)\\\notag
		\lesssim&\ 2^{k_0m\varepsilon}\sum_{k=k_0}^{\infty}\bigg\|\sum_{i=0}^{\infty}\sum_{j}\chi_{I_{k,i,j}}\bigg\|_{p(\cdot)}^m\gamma_b^m\bigg(\bigg\|\sum_{i=0}^{\infty}\sum_{j}\chi_{I_{k,i,j}}\bigg\|_{p(\cdot)}\bigg)\\\notag
		\lesssim&\ \sum_{k=k_0}^{\infty}2^{km\varepsilon}\bigg\|\sum_{i=0}^{\infty}\sum_{j}\chi_{I_{k,i,j}}\bigg\|_{p(\cdot)}^m\gamma_b^m\bigg(\bigg\|\sum_{i=0}^{\infty}\sum_{j}\chi_{I_{k,i,j}}\bigg\|_{p(\cdot)}\bigg).
	\end{align*}
	Obviously, according to Lemma \ref{belong}, we obtain
	 $$\|D_2\|_{p(\cdot),q,b}\lesssim\Bigg\|\Bigg\{2^k\bigg\|\sum_{i=0}^{\infty}\sum_{j}\chi_{I_{k,i,j}}\bigg\|_{p(\cdot)}\gamma_b\bigg(\bigg\|\sum_{i=0}^{\infty}\sum_{j}\chi_{I_{k,i,j}}\bigg\|_{p(\cdot)}\bigg)\Bigg\}_{k\in\mathbb{Z}}\Bigg\|_{l_q}.$$
	Hence, by Lemma \ref{fanshusanjiao}, there is
	\begin{align*}
		\|T(f)\|_{p(\cdot),q,b}
		\lesssim&\ \|D_1\|_{p(\cdot),q,b}+\|D_2\|_{p(\cdot),q,b}\\
		\lesssim&\ \Bigg\|\Bigg\{2^{k+1}\bigg\|\sum_{i=0}^{\infty}\sum_{j}\chi_{I_{k,i,j}}\bigg\|_{p(\cdot)}\gamma_b\bigg(\bigg\|\sum_{i=0}^{\infty}\sum_{j}\chi_{I_{k,i,j}}\bigg\|_{p(\cdot)}\bigg)\Bigg\}_{k\in\mathbb{Z}}\Bigg\|_{l_q}\\
		\lesssim&\ \|f\|_{H_{p(\cdot),q,b}^s}.
	\end{align*}
	The proof is complete now.
\end{proof}

Similarly, with Theorems \ref{ad2} and \ref{ad3}, we get the following results.
\begin{lem}
	Let $p(\cdot)\in\mathcal{P}(\Omega)$ satisfy $(\ref{gs3})$, $0<q\le\infty$ and $b$ be a slowly varying function. Suppose that $\{\mathcal{F}_n\}_{n\ge0}$ is regular and $\max\{p_+,1\}<r<\infty$. If $T:H^S_r\rightarrow L_r \,(\text{resp.}\; T: H^M_r\rightarrow L_r)$ is a bounded $\sigma$-sublinear operator and
	$$\{|T(a)|>0\}\subset I$$
	for every simple $(p(\cdot),r)^S$-atom $(\text{resp.}$ simple $(p(\cdot),r)^M$-atom$)$ $a$ associated with $I$, then for $f\in H_{p(\cdot),q,b}^S$ $(\text{resp.}\ f\in H_{p(\cdot),q,b}^M)$,
	$$\|T(f)\|_{p(\cdot),q,b}\lesssim\|f\|_{H_{p(\cdot),q,b}^S}\ \ \big(\text{resp.}\ \|T(f)\|_{p(\cdot),q,b}\lesssim\|f\|_{H_{p(\cdot),q,b}^M}\big).$$
\end{lem}

\begin{lem}\label{5.3}
	Let $p(\cdot)\in\mathcal{P}(\Omega)$ satisfy $(\ref{gs3})$, $0<q\le\infty$ and let $b$ be a slowly varying function. Suppose that $\max\{p_+,1\}<r<\infty$. If $T:H^S_r\rightarrow L_r\ (\text{resp.} \; T: H^M_r\rightarrow L_r)$ is a bounded $\sigma$-sublinear operator and
	$$\{|T(a)|>0\}\subset I$$
	for every simple $(p(\cdot),r)^S$-atom $(\text{resp.}$ simple $(p(\cdot),r)^M$-atom$)$ $a$ associated with $I$, then for $f\in \mathcal{Q}_{p(\cdot),q,b}$ $(\text{resp.}\ f\in \mathcal{P}_{p(\cdot),q,b})$,
	 $$\|T(f)\|_{p(\cdot),q,b}\lesssim\|f\|_{\mathcal{Q}_{p(\cdot),q,b}}\quad  \big(\text{resp.}\ \|T(f)\|_{p(\cdot),q,b}\lesssim\|f\|_{\mathcal{P}_{p(\cdot),q,b}}\big).$$
\end{lem}

To get the martingale inequalities for variable martingale Hardy-Lorentz-Karamata spaces, we also require the following lemma.

\begin{lem}[\cite{g,w}]\label{sMS}
	Let $f$ be a martingale. Then
	$$\|M(f)\|_2\le2\|S(f)\|_2=2\|s(f)\|_2\le2\|M(f)\|_2;$$
	$$\|s(f)\|_r\le\sqrt{\frac{r}{2}}\|M(f)\|_r,\ \ \ \ r\ge2;$$
	$$\|s(f)\|_r\le\sqrt{\frac{r}{2}}\|S(f)\|_r,\ \ \ \ r\ge2.$$
	Moreover, if the stochastic basis $\{\mathcal{F}_n\}_{n\ge0}$ is regular, then
	\begin{align*}
		\|M(f)\|_r\approx\|S(f)\|_r\approx\|s(f)\|_r,\ \ \ \ 0<r<\infty.
	\end{align*}
\end{lem}

Now, we show the main conclusion of this section.

\begin{thm}\label{mi}
	Let $p(\cdot)\in\mathcal{P}(\Omega)$ satisfy $(\ref{gs3})$, $0<q\le\infty$ and let $b$ be a slowly varying function. Then
	\begin{align}\label{1}
		\|f\|_{H_{p(\cdot),q,b}^M}\lesssim\|f\|_{H_{p(\cdot),q,b}^s},\ \ \|f\|_{H_{p(\cdot),q,b}^S}\lesssim\|f\|_{H_{p(\cdot),q,b}^s},\ \ \  0<p_-\le p_+<2;
	\end{align}
	\begin{align}\label{2}
		 \|f\|_{H_{p(\cdot),q,b}^M}\lesssim\|f\|_{\mathcal{P}_{p(\cdot),q,b}},\ \ \|f\|_{H_{p(\cdot),q,b}^S}\lesssim\|f\|_{\mathcal{Q}_{p(\cdot),q,b}};
	\end{align}
	\begin{align}\label{3}
		 \|f\|_{H_{p(\cdot),q,b}^S}\lesssim\|f\|_{\mathcal{P}_{p(\cdot),q,b}},\ \ \|f\|_{H_{p(\cdot),q,b}^M}\lesssim\|f\|_{\mathcal{Q}_{p(\cdot),q,b}};
	\end{align}
	\begin{align}\label{4}
		 \|f\|_{H_{p(\cdot),q,b}^s}\lesssim\|f\|_{\mathcal{P}_{p(\cdot),q,b}},\ \ \|f\|_{H_{p(\cdot),q,b}^s}\lesssim\|f\|_{\mathcal{Q}_{p(\cdot),q,b}};
	\end{align}
	\begin{align}\label{5}
		 \|f\|_{\mathcal{P}_{p(\cdot),q,b}}\lesssim\|f\|_{\mathcal{Q}_{p(\cdot),q,b}}\lesssim\|f\|_{\mathcal{P}_{p(\cdot),q,b}}.
	\end{align}
	If $\{\mathcal{F}_n\}_{n\ge0}$ is regular, then
	 $$H_{p(\cdot),q,b}^S=\mathcal{Q}_{p(\cdot),q,b}=\mathcal{P}_{p(\cdot),q,b}=H_{p(\cdot),q,b}^M=H_{p(\cdot),q,b}^s.$$
\end{thm}
\begin{proof}
	It is obvious that the operators $s,M,S$ are $\sigma$-sublinear and
	$$\{s(a_1)>0\}\subset I_1,\ \ \{M(a_2)>0\}\subset I_2\ \ \text{and}\ \{S(a_3)>0\}\subset I_3,$$
	where $a_i\ (i=1,2,3)$ is $(p(\cdot),r)^*$-atom $(*=s,M,S)$ associated with $I_i$ $(i=1,2,3)$, respectively. Lemma \ref{sMS} gives that the operators $M:H_2^s\rightarrow L_2$ and $S:H_2^s\rightarrow L_2$ are bounded. Hence, using Lemma \ref{5.1}, we know that the inequalities in (\ref{1}) hold if $0<p_-\le p_+<2$.
	
 The inequalities in (\ref{2}) follow directly from the definitions of $\mathcal{P}_{p(\cdot),q,b}$ and $\mathcal{Q}_{p(\cdot),q,b}$.
	
	According to Burkholder-Davis-Gundy's inequality (see \cite[Theorem 2.12]{w}), namely,
	$$c_r\|f\|_{H_r^S}\le\|f\|_{H_r^M}\le C_r\|f\|_{H_r^S},\ \ 1\le r<\infty$$
	and Doob's maximal inequality, we obtain that the operators $M:H_r^S\rightarrow L_r$ and $S:H_r^M\rightarrow L_r$ are bounded. Hence, using Lemma \ref{5.3}, the inequalities in (\ref{3}) hold.
	
	Let $\max\{p_+,2\}<r<\infty$. It follows from Lemma \ref{sMS} that $s:H_r^M\rightarrow L_r$ and $s:H_r^S\rightarrow L_r$ are bounded. Then by Lemma \ref{5.3}, we obtain the inequalities in (\ref{4}).
	
	In order to prove (\ref{5}), if we take $f=(f_n)_{n\ge0}\in\mathcal{Q}_{p(\cdot),q,b}$, then there exists $(\lambda_n^{(1)})_{n\ge0}\in\Lambda_{p(\cdot),q,b}$ such that $S_n(f)\le\lambda_{n-1}^{(1)}$ with $\lambda_\infty^{(1)}\in L_{p(\cdot),q,b}$. Since
	$$|f_n|\le|f_n-f_{n-1}|+|f_{n-1}|\le S_n(f)+M_{n-1}(f)\le \lambda_{n-1}^{(1)}+M_{n-1}(f)$$
	and $(\lambda_{n}^{(1)}+M_{n}(f))_{n\ge0}\in\Lambda_{p(\cdot),q,b}$, by the second inequality of (\ref{3}), we get
	 $$\|f\|_{\mathcal{P}_{p(\cdot),q,b}}\le\|\lambda_\infty^{(1)}+M(f)\|_{p(\cdot),q,b}\lesssim\|M(f)\|_{H_{p(\cdot),q,b}}+\|\lambda_\infty^{(1)}\|_{p(\cdot),q,b}\lesssim\|f\|_{\mathcal{Q}_{p(\cdot),q,b}}.$$
	If $f=(f_n)_{n\ge0}\in\mathcal{P}_{p(\cdot),q,b}$, then there exists $(\lambda_n^{(2)})_{n\ge0}\in\Lambda_{p(\cdot),q,b}$ such that $|f_n|\le\lambda_{n-1}^{(2)}$ with $\lambda_\infty^{(2)}\in L_{p(\cdot),q,b}$.
	Since $$S_n(f)\le S_{n-1}(f)+|f_n-f_{n-1}|\le S_{n-1}(f)+2M_n(f)\le S_{n-1}(f)+2\lambda_n^{(2)}$$
	and $(\lambda_{n}^{(2)}+S_{n}(f))_{n\ge0}\in\Lambda_{p(\cdot),q,b}$, by the first inequality of (\ref{3}), we have
	 $$\|f\|_{\mathcal{Q}_{p(\cdot),q,b}}\le\|\lambda_\infty^{(2)}+S(f)\|_{p(\cdot),q,b}\lesssim\|S(f)\|_{H_{p(\cdot),q,b}}+\|\lambda_\infty^{(2)}\|_{p(\cdot),q,b}\lesssim\|f\|_{\mathcal{P}_{p(\cdot),q,b}}.$$
	Hence, the inequality (\ref{5}) holds and this yields that $\mathcal{P}_{p(\cdot),q,b}=\mathcal{Q}_{p(\cdot),q,b}$.
	
	If $\{\mathcal{F}_n\}_{n\ge0}$ is regular, combining Theorem \ref{ad2} with Theorem \ref{ad3}, we conclude that
	$$H_{p(\cdot),q,b}^M=\mathcal{P}_{p(\cdot),q,b}\ \ \text{and}\ \ H_{p(\cdot),q,b}^S=\mathcal{Q}_{p(\cdot),q,b}.$$
	According to \cite[p.33]{w}, there is
	$$S_n(f)\le \mathcal{R}^{\frac{1}{2}}s_n(f).$$
	Then by the definition of $\mathcal{Q}_{p(\cdot),q,b}$, we obtain
	 $$\|f\|_{\mathcal{Q}_{p(\cdot),q,b}}\lesssim\|s(f)\|_{p(\cdot),q,b}=\|f\|_{H_{p(\cdot),q,b}^s},$$
	since $s_n(f)\in\mathcal{F}_{n-1}$.
	Hence, it follows from (\ref{4}) and (\ref{5}) that
	 $$H_{p(\cdot),q,b}^s=\mathcal{Q}_{p(\cdot),q,b}=\mathcal{P}_{p(\cdot),q,b}.$$
	Therefore,
	 $$H_{p(\cdot),q,b}^S=\mathcal{Q}_{p(\cdot),q,b}=H_{p(\cdot),q,b}^s=\mathcal{P}_{p(\cdot),q,b}=H_{p(\cdot),q,b}^M$$
	and the proof of this theorem is complete.
\end{proof}

\begin{rem}
	When $p(\cdot)\equiv p$ is a constant, Theorem $\ref{mi}$ removes the condition that $b$ is nondecreasing in \cite[Theorems $1.2$ and $1.3$]{wzp}. When $b\equiv1$, Theorem $\ref{mi}$ reduces to \cite[Theorem $4.11$]{jwzw}. Moreover, if $p(\cdot)\equiv p$ is a constant and $b\equiv1$, we obtain the relation of five martingale Hardy-Lorentz spaces, see \cite{w}.
\end{rem}

Moreover, combining Theorem \ref{mi} with Theorem \ref{ad11}, we have the following result.
\begin{coro}\label{cc11}
	Let $p(\cdot)\in\mathcal{P}(\Omega)$ satisfy $(\ref{gs3})$, $0<q\le\infty$ and let $b$ be a slowly varying function. If $\{\mathcal{F}_n\}_{n\ge0}$ is regular, then
	 $$H_{p(\cdot),q,b}^S=\mathcal{Q}_{p(\cdot),q,b}=\mathcal{P}_{p(\cdot),q,b}
	 =H_{p(\cdot),q,b}^M=H_{p(\cdot),q,b}^s=H_{p(\cdot),q,b}^{at,\infty,i},$$
	where $i=1,2,3$.
\end{coro}

As another application of Theorem \ref{mi}, we can obtain the boundedness of the martingale transform for variable Lorentz-Karamata spaces. The martingale transform was introduced by Burkholder in \cite{b} and he obtained the boundedness of the martingale transform generated by uniformly bounded sequence on the Lebesgue spaces.
 Jiao et al. \cite{jwzw} extended this result to the variable exponents setting, and studied the boundedness of martingale transform on variable Lebesgue spaces and variable Lorentz spaces. Firstly, we recall the definition of martingale transform.
\begin{defi}
	For any martingale $f\in\mathcal{M}$, the martingale transform of $f$ is defined by
	$$(T_vf)_{n}=\sum_{k=0}^{n}v_{k-1}d_kf,\ \ \ n\in\mathbb{N},$$
	where for each $k$, $v_k$ is $\mathcal{F}_k$-measurable and $\|v_k\|_{\infty}<L$ for some $L>0$.
\end{defi}

\begin{thm}\label{mt}
	Let $p(\cdot)\in\mathcal{P}(\Omega)$ satisfy $(\ref{gs3})$ with $1<p_-\le p_+<\infty$, $0<q\le\infty$ and let $b$ be a slowly varying function. Then we have
	$$\|T_vf\|_{p(\cdot),q,b}\lesssim\|f\|_{p(\cdot),q,b}\quad \text{for any}\ f\in L_{p(\cdot),q,b}.$$
\end{thm}
\begin{proof}
	Since $\|v_k\|_{\infty}<L$ for each $k$ and some $L>0$, we have
	\begin{align*}
		S(T_vf)=&\ \bigg(\sum_{n=0}^{\infty}\big|d_n(T_vf)\big|^2\bigg)^{\frac{1}{2}}=\bigg(\sum_{n=0}^{\infty}|v_{n-1}d_nf|^2\bigg)^{\frac{1}{2}}\\
		<&\ L\bigg(\sum_{n=0}^{\infty}|d_nf|^2\bigg)^{\frac{1}{2}}=LS(f).
	\end{align*}
	Hence,  by using Theorem \ref{mi} and Lemma \ref{Doob}, we obtain that
	\begin{align*}
		\|T_vf\|_{p(\cdot),q,b}\lesssim&\ \|M(T_vf)\|_{p(\cdot),q,b}\lesssim\|S(T_vf)\|_{p(\cdot),q,b}\\
		\lesssim&\ \|S(f)\|_{p(\cdot),q,b}\lesssim\|M(f)\|_{p(\cdot),q,b}\lesssim\|f\|_{p(\cdot),q,b}.
	\end{align*}
	The proof is complete.
\end{proof}

\section{Dual Results} \label{s5}
In this section, our main objective is to characterize the dual spaces of variable martingale Hardy-Lorentz-Karamata spaces. To this end, we introduce the following generalized $BMO$ martingale spaces.

\begin{defi}\label{B1}
    Let $\alpha(\cdot)+1\in\mathcal{P}(\Omega)$, $1\le r<\infty$ and let $b$ be a slowly varying function. The space $BMO_{r,b}(\alpha(\cdot))$ consists of those functions $f\in L_r$ for which
    $$\|f\|_{BMO_{r,b}(\alpha(\cdot))}=\sup_{n\ge0}\sup_{I\in A(\mathcal{F}_n)}\|\chi_{I}\|_{\frac{1}{\alpha(\cdot)+1}}^{-1}\mathbb{P}(I)^{1-\frac{1}{r}}\gamma_b^{-1}\big(\|\chi_{I}\|_{\frac{1}{\alpha(\cdot)+1}}\big)\|(f-\mathbb{E}_n(f))\chi_I\|_r$$
    is finite.
\end{defi}

\begin{rem}
	If $b\equiv1$, then the above generalized martingale space is the same as $BMO_{r}(\alpha(\cdot))$ in \cite{jwwz}. $BMO_{r,b}(\alpha)$ given by Jiao et al. \cite{jxz} is a special case of Definition $\ref{B1}$ when $\alpha(\cdot)\equiv\alpha>0$ is a constant. Moreover, if $b\equiv1$ and $\alpha(\cdot)\equiv0$, then the generalized martingale space goes back to the classical martingale $BMO_r$ space in \cite{w}.
\end{rem}

Now we introduce another generalization of $BMO$ spaces.

\begin{defi}\label{B2}
	Let $\alpha(\cdot)+1\in\mathcal{P}(\Omega)$, $1\le r<\infty$, $0<q\le\infty$ and let $b$ be a slowly varying function. The generalized martingale space $BMO_{r,q,b}(\alpha(\cdot))$ is defined by
	$$BMO_{r,q,b}(\alpha(\cdot))=\big\{f\in L_r:\|f\|_{BMO_{r,q,b}(\alpha(\cdot))}<\infty\big\},$$
	where
	 $$\|f\|_{BMO_{r,q,b}(\alpha(\cdot))}=\sup\frac{\sum\limits_{k\in\mathbb{Z}}\sum\limits_{i=0}^{\infty}\sum\limits_{j}2^k\mathbb{P}(I_{k,i,j})^{1-\frac{1}{r}}\|(f-\mathbb{E}_i(f))\chi_{I_{k,i,j}}\|_r}{\bigg\|\bigg\{2^{k}\Big\|\sum\limits_{i=0}^{\infty}\sum\limits_{j}\chi_{I_{k,i,j}}\Big\|_{\frac{1}{\alpha(\cdot)+1}}\gamma_b\Big(\Big\|\sum\limits_{i=0}^{\infty}\sum\limits_{j}\chi_{I_{k,i,j}}\Big\|_{\frac{1}{\alpha(\cdot)+1}}\Big)\bigg\}_{k\in\mathbb{Z}}\bigg\|_{l_q}}$$
	and the supremum is taken over all sequences of atoms $(I_{k,i,j})_{k\in\mathbb{Z},i\in\mathbb{N},j}$ such that $I_{k,i,j}$ are disjoint if $k$ is fixed, $I_{k,i,j}$ belong to $\mathcal{F}_i$ and
	 $$\Bigg\{2^{k}\bigg\|\sum\limits_{i=0}^{\infty}\sum\limits_{j}\chi_{I_{k,i,j}}\bigg\|_{\frac{1}{\alpha(\cdot)+1}}\gamma_b\bigg(\bigg\|\sum\limits_{i=0}^{\infty}\sum\limits_{j}\chi_{I_{k,i,j}}\bigg\|_{\frac{1}{\alpha(\cdot)+1}}\bigg)\Bigg\}_{k\in\mathbb{Z}}\in l_q.$$
\end{defi}

\begin{rem}
	If $b\equiv1$, $BMO_{r,q,b}(\alpha(\cdot))$ reduces to $BMO_{r,q}(\alpha(\cdot))$ in \cite{jwwz}. $BMO_{r,q,b}(\alpha)$ defined by Jiao et al. \cite{jxz} is a special case of Definition $\ref{B2}$ when $\alpha(\cdot)\equiv\alpha>0$ is a constant. Moreover, if $b\equiv1$ and $\alpha(\cdot)\equiv0$, then the generalized martingale space goes back to classical martingale  $BMO_{r,q}$ space in \cite{w}.
\end{rem}

The following result gives the connection of $BMO_{r,b}(\alpha(\cdot))$ and $BMO_{r,q,b}(\alpha(\cdot))$.
\begin{pro}\label{prod}
	Let $\alpha(\cdot)+1\in\mathcal{P}(\Omega)$, $1\le r<\infty$, $0<q<\infty$ and let $b$ be a slowly varying function. Then
	\begin{align*}
		 \|f\|_{BMO_{r,b}(\alpha(\cdot))}\le\|f\|_{BMO_{r,q,b}(\alpha(\cdot))}.
	\end{align*}
	If we suppose that $\alpha_->0$ and $0<q\le1$, then
	\begin{align}\label{''}
		 \|f\|_{BMO_{r,b}(\alpha(\cdot))}\approx\|f\|_{BMO_{r,q,b}(\alpha(\cdot))}.
	\end{align}
	In addition, if $\alpha_-\ge0$ and $b$ is nonincreasing, we also have $(\ref{''})$.
\end{pro}
\begin{proof}
	If we take the supremum in the definition of $BMO_{r,q,b}(\alpha(\cdot))$ only for one atom, then we get back to the definition of $BMO_{r,b}(\alpha(\cdot))$. Hence the first inequality is obvious. For the second inequality, it clearly suffices to verify that $\|f\|_{BMO_{r,q,b}(\alpha(\cdot))}\lesssim\|f\|_{BMO_{r,b}(\alpha(\cdot))}.$
	
	For any $f\in BMO_{r,b}(\alpha(\cdot))$, by the definition of $BMO_{r,b}(\alpha(\cdot))$, there is
	$$\mathbb{P}(I)^{1-\frac{1}{r}}\|(f-\mathbb{E}_i(f))\chi_I\|_r\le \|\chi_{I}\|_{\frac{1}{\alpha(\cdot)+1}}\gamma_b\big(\|\chi_{I}\|_{\frac{1}{\alpha(\cdot)+1}}\big)\|f\|_{BMO_{r,b}(\alpha(\cdot))}$$
	for any $I\in A(\mathcal{F}_i).$
	Since $\alpha_->0$ and $0<q\le1$, then it follows from Lemma \ref{wuqiongbianb} that
	\begin{align*}
		&\ \frac{\sum\limits_{k\in\mathbb{Z}}\sum\limits_{i=0}^{\infty}\sum\limits_{j}2^k\mathbb{P}(I_{k,i,j})^{1-\frac{1}{r}}\|(f-\mathbb{E}_i(f))\chi_{I_{k,i,j}}\|_r}{\bigg(\sum\limits_{k\in\mathbb{Z}}2^{kq}\Big\|\sum\limits_{i=0}^{\infty}\sum\limits_{j}\chi_{I_{k,i,j}}\Big\|_{\frac{1}{\alpha(\cdot)+1}}^q\gamma_b^q\Big(\Big\|\sum\limits_{i=0}^{\infty}\sum\limits_{j}\chi_{I_{k,i,j}}\Big\|_{\frac{1}{\alpha(\cdot)+1}}\Big)\bigg)^\frac{1}{q}}\\
		\le&\ \frac{\sum\limits_{k\in\mathbb{Z}}\sum\limits_{i=0}^{\infty}\sum\limits_{j}2^k\|\chi_{I_{k,i,j}}\|_{\frac{1}{\alpha(\cdot)+1}}\gamma_b\big(\|\chi_{I_{k,i,j}}\|_{\frac{1}{\alpha(\cdot)+1}}\big)\|f\|_{BMO_{r,b}(\alpha(\cdot))}}{\bigg(\sum\limits_{k\in\mathbb{Z}}2^{kq}\Big\|\sum\limits_{i=0}^{\infty}\sum\limits_{j}\chi_{I_{k,i,j}}\Big\|_{\frac{1}{\alpha(\cdot)+1}}^q\gamma_b^q\Big(\Big\|\sum\limits_{i=0}^{\infty}\sum\limits_{j}\chi_{I_{k,i,j}}\Big\|_{\frac{1}{\alpha(\cdot)+1}}\Big)\bigg)^\frac{1}{q}}\\
		\le&\ \|f\|_{BMO_{r,b}(\alpha(\cdot))}\frac{\sum\limits_{k\in\mathbb{Z}}\sum\limits_{i=0}^{\infty}\sum\limits_{j}2^k\|\chi_{I_{k,i,j}}\|_{\frac{1}{\alpha(\cdot)+1}}\gamma_b\big(\|\chi_{I_{k,i,j}}\|_{\frac{1}{\alpha(\cdot)+1}}\big)}{\sum\limits_{k\in\mathbb{Z}}2^{k}\Big\|\sum\limits_{i=0}^{\infty}\sum\limits_{j}\chi_{I_{k,i,j}}\Big\|_{\frac{1}{\alpha(\cdot)+1}}\gamma_b\Big(\Big\|\sum\limits_{i=0}^{\infty}\sum\limits_{j}\chi_{I_{k,i,j}}\Big\|_{\frac{1}{\alpha(\cdot)+1}}\Big)}\\
		\lesssim&\ \|f\|_{BMO_{r,b}(\alpha(\cdot))}\frac{\sum\limits_{k\in\mathbb{Z}}2^{k}\Big\|\sum\limits_{i=0}^{\infty}\sum\limits_{j}\chi_{I_{k,i,j}}\Big\|_{\frac{1}{\alpha(\cdot)+1}}\gamma_b\Big(\Big\|\sum\limits_{i=0}^{\infty}\sum\limits_{j}\chi_{I_{k,i,j}}\Big\|_{\frac{1}{\alpha(\cdot)+1}}\Big)}{\sum\limits_{k\in\mathbb{Z}}2^{k}\Big\|\sum\limits_{i=0}^{\infty}\sum\limits_{j}\chi_{I_{k,i,j}}\Big\|_{\frac{1}{\alpha(\cdot)+1}}\gamma_b\Big(\Big\|\sum\limits_{i=0}^{\infty}\sum\limits_{j}\chi_{I_{k,i,j}}\Big\|_{\frac{1}{\alpha(\cdot)+1}}\Big)}\\
		=&\ \|f\|_{BMO_{r,b}(\alpha(\cdot))}.
	\end{align*}
    Hence, $\|f\|_{BMO_{r,q,b}(\alpha(\cdot))}\lesssim\|f\|_{BMO_{r,b}(\alpha(\cdot))}$ and this completes the proof of \eqref{''}.
    If $\alpha_-\ge0$ and $b$ is nonincreasing, the proof is similar.
\end{proof}

We give the main result of this section.

\begin{thm}\label{dual}
	Let $p(\cdot)\in\mathcal{P}(\Omega)$ satisfy $(\ref{gs3})$ with $0<p_-\le p_+<2$, $0<q<\infty$ and let $b$ be a slowly varying function. Then
	$$\big(H_{p(\cdot),q,b}^s\big)^*=BMO_{2,q,b}(\alpha(\cdot)),\quad \alpha(\cdot)=\frac{1}{p(\cdot)}-1.$$
\end{thm}
\begin{proof}
    Let $g\in BMO_{2,q,b}(\alpha(\cdot))\subset L_2$. Define the functional $\ell_g$ as
    $$\ell_g(f):=\mathbb{E}(fg),\quad \forall\ f\in L_2.$$
    We shall prove that $\ell_g$ is a bounded linear functional on $H_{p(\cdot),q,b}^s$. By Remark \ref{rem4.2}, we see that $L_2\subset H_{p(\cdot),q,b}^s$. Moreover, $$f=\sum_{k\in\mathbb{Z}}\sum_{i=0}^{\infty}\sum_{j}\mu_{k,i,j}a^{k,i,j}$$
    holds also in $L_2$, where $\mu_{k,i,j}=3\cdot2^{k}\|\chi_{I_{k,i,j}}\|_{p(\cdot)}$ and $(a^{k,i,j})_{k\in\mathbb{Z},i\in\mathbb{N},j}$ is a sequence of simple $(p(\cdot),2)^s$-atoms.
    In addition,
    $$\Bigg(\sum_{ k \in \mathbb{Z} }\bigg\|\sum_{i=0}^{\infty}\sum_{j}\mu_{k,i,j}\|\chi_{I_{k,i,j}}\|_{p(\cdot)}^{-1}\chi_{I_{k,i,j}}
    \bigg\|_{p(\cdot)}^q\gamma_b^q\bigg(\bigg\|\sum_{i=0}^{\infty}\sum_{j}\chi_{I_{k,i,j}}\bigg\|_{p(\cdot)}\bigg)\Bigg)^{\frac{1}{q}}\lesssim\|f\|_{H_{p(\cdot),q,b}^{s}}.$$
    Hence,
    $$\ell_g(f)=\mathbb{E}(fg)=\sum_{k\in\mathbb{Z}}\sum_{i=0}^{\infty}\sum_{j}\mu_{k,i,j}\mathbb{E}(a^{k,i,j}g).$$
    By the definition of $a^{k,i,j}$, there is
    $$\mathbb{E}(a^{k,i,j}g)=\mathbb{E}\big((a^{k,i,j}-\mathbb{E}_i(a^{k,i,j}))g\big)=\mathbb{E}\big(a^{k,i,j}(g-\mathbb{E}_i(g))\big).$$
    According to H\"{o}lder's inequality, we obtain that
    \begin{align*}
    	|\ell_g(f)|=&\ \bigg|\sum_{k\in\mathbb{Z}}\sum_{i=0}^{\infty}\sum_{j}\mu_{k,i,j}\mathbb{E}(a^{k,i,j}(g-\mathbb{E}_i(g)))\bigg|\\
    	\le&\ \sum_{k\in\mathbb{Z}}\sum_{i=0}^{\infty}\sum_{j}\mu_{k,i,j}\|a^{k,i,j}\|_2\|(g-\mathbb{E}_i(g))\chi_{I_{k,i,j}}\|_2\\
    	\lesssim&\ \sum_{k\in\mathbb{Z}}\sum_{i=0}^{\infty}\sum_{j}\mu_{k,i,j}\frac{\mathbb{P}(I_{k,i,j})^{1/2}}{\|\chi_{I_{k,i,j}}\|_{p(\cdot)}}\|(g-\mathbb{E}_i(g))\chi_{I_{k,i,j}}\|_2\\
    	=&\ 3\sum_{k\in\mathbb{Z}}\sum_{i=0}^{\infty}\sum_{j}2^{k}\mathbb{P}(I_{k,i,j})^{1/2}\|(g-\mathbb{E}_i(g))\chi_{I_{k,i,j}}\|_2\\
    	\lesssim&\ \Bigg(\sum_{k\in\mathbb{Z}}2^{kq}\bigg\|\sum_{i=0}^{\infty}\sum_{j}\chi_{I_{k,i,j}}\Big\|_{p(\cdot)}^q\gamma_b^q\bigg(\bigg\|\sum\limits_{i=0}^{\infty}\sum\limits_{j}\chi_{I_{k,i,j}}\bigg\|_{p(\cdot)}\bigg)\Bigg)^\frac{1}{q}\cdot\|g\|_{BMO_{2,q,b}(\alpha(\cdot))}\\
    	\lesssim&\ \|f\|_{H_{p(\cdot),q,b}^s}\cdot\|g\|_{BMO_{2,q,b}(\alpha(\cdot))}.
    \end{align*}
    Since $L_2$ is dense in $H_{p(\cdot),q,b}^s$, the linear functional $\ell_g$ can be uniquely extended to a bounded linear functional on $H_{p(\cdot),q,b}^s$.

    Conversely, let $\ell\in \big(H_{p(\cdot),q,b}^s\big)^*$. We shall show that there exists $g\in BMO_{2,q,b}(\alpha(\cdot))$ such that $\ell=\ell_g$ and $$\|g\|_{BMO_{2,q,b}(\alpha(\cdot))}\lesssim\|\ell\|.$$
    Since $L_2$ can be embedded continuously to $H_{p(\cdot),q,b}^s$, then there exists $g\in L_2$ such that
    $$\ell(f)=\mathbb{E}(fg),\quad \forall\ f\in L_2.$$
    Let $\{I_{k,i,j}\}_{k\in\mathbb{Z},i\in\mathbb{N},j}$ be an arbitrary sequence of atoms such that $I_{k,i,j}$ are disjoint if $k$ is fixed, $I_{k,i,j}$ belong to $\mathcal{F}_i$ and
    $$\Bigg\{2^{k}\bigg\|\sum\limits_{i=0}^{\infty}\sum\limits_{j}\chi_{I_{k,i,j}}\bigg\|_{\frac{1}{\alpha(\cdot)+1}}\gamma_b\bigg(\bigg\|\sum\limits_{i=0}^{\infty}\sum\limits_{j}\chi_{I_{k,i,j}}\bigg\|_{\frac{1}{\alpha(\cdot)+1}}\bigg)\Bigg\}_{k\in\mathbb{Z}}\in l_q.$$
    Set
    $$h_{k,i,j}=\frac{(g-\mathbb{E}_i(g))\chi_{I_{k,i,j}}\|\chi_{I_{k,i,j}}\|_2}{\|(g-\mathbb{E}_i(g))\chi_{I_{k,i,j}}\|_2\cdot\|\chi_{I_{k,i,j}}\|_{\frac{1}{\alpha(\cdot)+1}}}.$$
    Then $h_{k,i,j}$ is a simple $(p(\cdot),2)^s$-atom and
    \begin{align*}
    	 \mathbb{E}\big(h_{k,i,j}(g-\mathbb{E}_i(g))\big)=\frac{\|\chi_{I_{k,i,j}}\|_2\cdot\|(g-\mathbb{E}_i(g))\chi_{I_{k,i,j}}\|_2}{\|\chi_{I_{k,i,j}}\|_{p(\cdot)}}.
    \end{align*}
    By Theorem \ref{ad1}, we find that
    $$f=\sum_{k\in\mathbb{Z}}\sum_{i=0}^{\infty}\sum_{j}3\cdot2^k\|\chi_{I_{k,i,j}}\|_{p(\cdot)}h_{k,i,j}\in H_{p(\cdot),q,b}^s$$
    and
    \begin{align}\label{gs18}
    	 \|f\|_{H_{p(\cdot),q,b}^s}\lesssim\Bigg(\sum_{k\in\mathbb{Z}}2^{kq}\bigg\|\sum_{i=0}^{\infty}\sum_{j}\chi_{I_{k,i,j}}\bigg\|_{p(\cdot)}^q\gamma_b^q\bigg(\bigg\|\sum\limits_{i=0}^{\infty}\sum\limits_{j}\chi_{I_{k,i,j}}\bigg\|_{p(\cdot)}\bigg)\Bigg)^{1/q}.
    \end{align}
    Hence
    $$\ell(f)=\mathbb{E}(fg)=3\sum_{k\in\mathbb{Z}}\sum_{i=0}^{\infty}\sum_{j}2^k\|\chi_{I_{k,i,j}}\|_{p(\cdot)}\mathbb{E}(h_{k,i,j}g).$$
    Furthermore, we have
    \begin{align*}
    	&\ \sum_{k\in\mathbb{Z}}\sum_{i=0}^{\infty}\sum_{j}2^k\mathbb{P}(I_{k,i,j})^\frac{1}{2}\|(g-\mathbb{E}_i(g))\chi_{I_{k,i,j}}\|_2\\
    	=&\ \sum_{k\in\mathbb{Z}}\sum_{i=0}^{\infty}\sum_{j}2^k\|\chi_{I_{k,i,j}}\|_{p(\cdot)}\mathbb{E}\big(h_{k,i,j}(g-\mathbb{E}_i(g))\big)\\
    	=&\ \sum_{k\in\mathbb{Z}}\sum_{i=0}^{\infty}\sum_{j}2^k\|\chi_{I_{k,i,j}}\|_{p(\cdot)}\mathbb{E}(h_{k,i,j}g)\\
    	=&\ \frac{1}{3}\ell(f)\le\frac{1}{3}\|f\|_{H_{p(\cdot),q,b}^s}\cdot\|\ell\|.
    \end{align*}
    Thus, applying (\ref{gs18}) and the definition of $\|\cdot\|_{BMO_{2,q,b}(\alpha(\cdot))}$, we obtain
    $$\|g\|_{BMO_{2,q,b}(\alpha(\cdot))}\lesssim\|\ell\|.$$
\end{proof}

\begin{rem}
	For the dual theory of martingale Hardy spaces, the researchers always split the range of $q$ to $0<q\le1$ and $1<q<\infty$ and discuss these two cases separately, see \cite{hl,hoa,jwwz,jxz,w}. In this paper, we get the duality theorem when $0<q<\infty$, which unifies the two cases.
\end{rem}

As a deduction of Proposition \ref{prod} and Theorem \ref{dual}, we obtain the dual theorem associated with $BMO_{2,b}(\alpha(\cdot))$.
\begin{coro}
	Let $p(\cdot)\in\mathcal{P}(\Omega)$ satisfy $(\ref{gs3})$ with $0<p_-\le p_+<1$, $0<q\le1$ and let $b$ be a slowly varying function. Then
	\begin{align}\label{dd}
		\big(H_{p(\cdot),q,b}^s\big)^*=BMO_{2,b}(\alpha(\cdot)),\quad \alpha(\cdot)=\frac{1}{p(\cdot)}-1.
	\end{align}
Moreover, if $0<p_-\le p_+\le1$ and $b$ is nonincreasing, $(\ref{dd})$ also holds.
\end{coro}

Particularly, Theorem \ref{dual} reduces to the following result when $b\equiv1$.
\begin{coro}\label{coro5}
	Let $p(\cdot)\in\mathcal{P}(\Omega)$ satisfy $(\ref{gs3})$ with $0<p_-\le p_+<2$ and $0<q<\infty$. Then $$\big(H_{p(\cdot),q}^s\big)^*=BMO_{2,q}(\alpha(\cdot)),\quad \alpha(\cdot)=\frac{1}{p(\cdot)}-1.$$
\end{coro}
\begin{rem}
	Jiao et al. \cite{jwwz} got the dual theorem of $H_{p(\cdot),q}^s$ in the situation that $p(\cdot)\in\mathcal{P}(\Omega)$ satisfies $(\ref{gs3})$ with $0<p_-\le p_+<2$ and $1<q<\infty$. Notice that, the condition of $q$ is $0<q<\infty$ in Corollary $\ref{coro5}$, therefore, our result extends the dual theorem in \cite{jwwz}.
\end{rem}

Specially, if $\alpha(\cdot)\equiv\alpha$ is a constant, then we have the next conclusion.
\begin{coro}\label{coro6}
	Let $0<p<2$, $0<q<\infty$ and $b$ be a slowly varying function. Then $$\big(H_{p,q,b}^s\big)^*=BMO_{2,q,b}(\alpha),\quad \alpha=\frac{1}{p}-1.$$
\end{coro}

\begin{rem}
	We refer to Jiao et al. \cite{jxz} for the duality of $H_{p,q,b}^s$ if $0<p\le1$, $1<q<\infty$ and $b$ is a nondecreasing slowly varying function. We extend the range of $p$ and $q$ to $0<p<2$ and $0<q<\infty$, respectively. Moreover, the slowly varying function $b$ is not necessarily nondecreasing in Corollary $\ref{coro6}$. Hence, Corollary $\ref{coro6}$ improves the dual theorem in \cite{jxz}.
\end{rem}

Next we consider $q=\infty$. This case is different to $0<q<\infty$, since $L_p$ is not dense in $L_{p,\infty}$ for $0<p<\infty$. We refer to \cite{ww} for this fact. From Remark \ref{huaH} and Theorem \ref{ad1}, we know that $L_2$ is dense in $\mathcal{H}_{p(\cdot),\infty,b}^s$.
Hence, similarly to the proof of Theorem \ref{dual}, we get the dual space when $q=\infty$.
\begin{thm}
	Let $p(\cdot)\in\mathcal{P}(\Omega)$ satisfy $(\ref{gs3})$ with $0<p_-\le p_+<2$ and $b$ be a slowly varying function. Then
	 $$\big(\mathcal{H}_{p(\cdot),\infty,b}^s\big)^*=BMO_{2,\infty,b}(\alpha(\cdot)),\quad \alpha(\cdot)=\frac{1}{p(\cdot)}-1.$$
\end{thm}

When $b\equiv1$, we obtain the dual spaces of $\mathcal{H}_{p(\cdot),\infty}^s$ as follows.
\begin{coro}\label{coro7}
	Let $p(\cdot)\in\mathcal{P}(\Omega)$ satisfy $(\ref{gs3})$ with $0<p_-\le p_+<2$. Then
	 $$\big(\mathcal{H}_{p(\cdot),\infty}^s\big)^*=BMO_{2,\infty}(\alpha(\cdot)),\quad \alpha(\cdot)=\frac{1}{p(\cdot)}-1.$$
\end{coro}
\begin{rem}
	Obviously, Corollary $\ref{coro7}$ can reduce to the result in \cite{jwwz}.
\end{rem}

If $\alpha(\cdot)\equiv\alpha$ is a constant, we know that the next result holds.
\begin{coro}\label{coro8}
	Let $0<p<2$ and $b$ be a slowly varying function. Then
	 $$\big(\mathcal{H}_{p,\infty,b}^s\big)^*=BMO_{2,\infty,b}(\alpha),\quad \alpha=\frac{1}{p}-1.$$
\end{coro}
\begin{rem}
	The duality of $\mathcal{H}_{p,\infty,b}^s$ was proved by Liu and Zhou in \cite{lz} if $0<p\le1$ and $b$ is a nondecreasing slowly varying function. Notice that Corollary $\ref{coro8}$ improves the range of $p$ to $0<p<2$ and shows that the nondecreasing condition of $b$ is unnecessary. Hence, Corollary $\ref{coro8}$ extends the dual theorem in \cite{lz}.
\end{rem}

\section{John-Nirenberg Theorem} \label{s6}

In this section, we prove the John-Nirenberg theorem for variable Lorentz-Karamata spaces when the stochastic basis $\{\mathcal{F}_n\}_{n\geq0}$ is regular.


\begin{thm}\label{jn1}
Let $p(\cdot)\in\mathcal{P}(\Omega)$ satisfy $(\ref{gs3})$ with $0<p_-\le p_+\le1$, $0<q<\infty$, $1<r< \infty$ and let $ b$ be a slowly varying function.
If the stochastic basis $\{\mathcal{F}_n\}_{n\geq0}$ is regular, then
$$\big(H^s_{p(\cdot),q,b}\big)^*=BMO_{r,q,b}(\alpha(\cdot)),\quad \alpha(\cdot)=\frac{1}{p(\cdot)}-1.$$
\end{thm}
\begin{proof}
The proof is very similar to that of Theorem \ref{dual}, so we only sketch it.
Let $r'$ be the conjugate number of $r$. Firstly we claim that
$L_{r'} \subset H^s_{p(\cdot),q,b}$. Indeed, for any $f\in L_{r'}$, it follows from Lemma \ref{3.4.48} and Theorem \ref{mi} that
$$\|f\|_{H^s_{p(\cdot),q,b}}=\|s(f)\|_{p(\cdot),q,b}\lesssim \|s(f)\|_{1,q,b} \lesssim
\|s(f)\|_{r',r',1}=\|s(f)\|_{r'}\approx\|f\|_{r'}
.$$
Hence, for any $f\in L_{r'}$, on the basis of Theorem \ref{ad1}, there exist a sequence of simple $(p(\cdot),r)^s$-atoms $(a^{k,i,j})_{k\in\mathbb{Z},i\in\mathbb{N},j}$ and $\mu_{k,i,j}=3\cdot2^k\|\chi_{I_{k,i,j}}\|_{p(\cdot)}$ such that
\begin{align*}
  f =\sum_{k\in \mathbb{Z}}\sum_{i=0}^{\infty}\sum_j \mu_{k,i,j}a^{k,i,j}
\end{align*}
in $L_{r'}$ and
\begin{align*}
	\Bigg(\sum_{ k \in \mathbb{Z} }\bigg\|\sum_{i=0}^{\infty}\sum_{j}\mu_{k,i,j}\|\chi_{I_{k,i,j}}\|_{p(\cdot)}^{-1}\chi_{I_{k,i,j}}
	 \bigg\|_{p(\cdot)}^q\gamma_b^q\bigg(\bigg\|\sum_{i=0}^{\infty}\sum_{j}\chi_{I_{k,i,j}}\bigg\|_{p(\cdot)}\bigg)\Bigg)^{\frac{1}{q}}\lesssim\|f\|_{H_{p(\cdot),q,b}^{s}}.
\end{align*}
For any given $g\in BMO_{r,q,b}\subset L_r$, define
$$\ell_g(f):=\mathbb{E}(fg),\quad\forall\ f\in L_{r'}.$$
In the same way as in the proof of Theorem \ref{dual}, we get that $\ell_g$ can be uniquely extended to a bounded linear functional on $H_{p(\cdot),q,b}^s$.

To prove the converse, let $\ell\in \big( H_{p(\cdot),q,b}^s\big)^*$. Since
$$L_{r'}=H_{r'}^s\subset H_{p(\cdot),q,b}^s = H_{p(\cdot),q,b}^{s-at,\infty,1} \subset H_{p(\cdot),q,b}^{s-at, r', 1},$$
we have
$$\big(H_{p(\cdot),q,b}^{s-at, r', 1}\big)^* \subset \big( H_{p(\cdot),q,b}^s\big)^*\subset L_r.$$
Hence
there exists $g \in L_r$ such that
$$\ell(f)=\ell_g(f)=\mathbb{E}(fg),\quad\quad\forall \;f\in L_{r'}.$$
Let $\{I_{k,i,j}\}_{k\in\mathbb{Z},i\in\mathbb{N},j}$ be an arbitrary sequence of atoms such that $I_{k,i,j}$ are disjoint if $k$ is fixed, $I_{k,i,j}$ belong to $\mathcal{F}_i$ and
$$\Bigg\{2^{k}\bigg\|\sum\limits_{i=0}^{\infty}\sum\limits_{j}\chi_{I_{k,i,j}}\bigg\|_{\frac{1}{\alpha(\cdot)+1}}\gamma_b\bigg(\bigg\|\sum\limits_{i=0}^{\infty}\sum\limits_{j}\chi_{I_{k,i,j}}\bigg\|_{\frac{1}{\alpha(\cdot)+1}}\bigg)\Bigg\}_{k\in\mathbb{Z}}\in l_q.$$
Set
$$h_{k,i,j}=\frac{|g-\mathbb{E}_i(g)|^{r-1}\operatorname{sgn}(g-\mathbb{E}_i(g))\chi_{I_{k,i,j}}\|\chi_{I_{k,i,j}}\|_{r'}}{\|(g-\mathbb{E}_i(g))\chi_{I_{k,i,j}}\|_r^{r-1}\cdot\|\chi_{I_{k,i,j}}\|_{p(\cdot)}}.$$
Then $h_{k,i,j}$ is a simple $(p(\cdot),r)^s$-atom and
\begin{align*}
	 \mathbb{E}\big(h_{k,i,j}(g-\mathbb{E}_i(g))\big)=\frac{\|\chi_{I_{k,i,j}}\|_{r'}\cdot\|(g-\mathbb{E}_i(g))\chi_{I_{k,i,j}}\|_r}{\|\chi_{I_{k,i,j}}\|_{p(\cdot)}}.
\end{align*}
By replacing $h_{k,i,j}$ in the proof of Theorem \ref{dual} by this new definition, we can obtain the conclusion and the proof is complete.
\end{proof}

For the case of $r=1$, we need some new insight. Let the dual space of $\mathcal{P}_{p(\cdot),q,b}$ be $\mathcal{P}_{p(\cdot),q,b}^*$. Denote by $\big(\mathcal{P}_{p(\cdot),q,b}^*\big)_1$ those elements $\ell\in\mathcal{P}_{p(\cdot),q,b}^*$ for which there exists $g\in L_1$ such that $\ell(f)=\mathbb{E}(fg)$, $f\in L_\infty$. That is,
$$\big(\mathcal{P}_{p(\cdot),q,b}^*\big)_1:=\big\{\ell\in \mathcal{P}_{p(\cdot),q,b}^*:\ \exists\ g\in L_1\ \text{s.t.}\ \ell(f)=\mathbb{E}(fg),\ \forall\ f\in L_\infty\big\}.$$

\begin{thm}\label{jn2}
	Let $p(\cdot)\in\mathcal{P}(\Omega)$ satisfy $(\ref{gs3})$ with $0<p_-\le p_+\le1$, $0<q< \infty$ and let $ b$ be a slowly varying function. Then $$\big(\mathcal{P}_{p(\cdot),q,b}^*\big)_1=BMO_{1,q,b}(\alpha(\cdot)),\quad \alpha(\cdot)=\frac{1}{p(\cdot)}-1.$$
\end{thm}
\begin{proof}
	Let $g\in BMO_{1,q,b}(\alpha(\cdot))\subset L_1$. Define the functional $\ell_g$ as $$\ell_g(f):=\mathbb{E}(fg),\quad \forall\ f\in L_{\infty}.$$
	By Theorem \ref{ad2}, there exist a sequence of simple $(p(\cdot),\infty)^M$-atoms $(a^{k,i,j})_{k\in\mathbb{Z},i\in\mathbb{N},j}$ and $\mu_{k,i,j}=3\cdot2^k\|\chi_{I_{k,i,j}}\|_{p(\cdot)}$ such that $$f=\sum\limits_{k\in\mathbb{Z}}\sum\limits_{i=0}^{\infty}\sum\limits_{j}\mu_{k,i,j}a^{k,i,j}$$
	and $$\Bigg(\sum\limits_{k\in\mathbb{Z}}\bigg\|\sum\limits_{i=0}^{\infty}\sum\limits_{j}\mu_{k,i,j}\|\chi_{I_{k,i,j}}\|_{p(\cdot)}^{-1}\chi_{I_{k,i,j}}\bigg\|_{p(\cdot)}^{q}\gamma_b^q\bigg(\bigg\|\sum\limits_{i=0}^{\infty}\sum\limits_{j}\chi_{I_{k,i,j}}\bigg\|_{p(\cdot)}\bigg)\Bigg)^\frac{1}{q}\lesssim \|f\|_{\mathcal{P}_{p(\cdot),q,b}}.$$
	Hence, it follows from dominated convergence theorem and H\"{o}lder's inequality that
	\begin{align*}
		|\ell_g(f)|&\leq \sum\limits_{k\in\mathbb{Z}}\sum\limits_{i=0}^{\infty}\sum\limits_{j}\mu_{k,i,j}\mathbb{E}(a^{k,i,j}g)
		\\
		&\leq
		 \sum\limits_{k\in\mathbb{Z}}\sum\limits_{i=0}^{\infty}\sum\limits_{j}\mu_{k,i,j}\|a^{k,i,j}\|_{\infty}\|(g-\mathbb{E}_i(g))\chi_{I_{k,i,j}}\|_1
		\\&=
		 3\sum\limits_{k\in\mathbb{Z}}\sum\limits_{i=0}^{\infty}\sum\limits_{j}2^k\|(g-\mathbb{E}_i(g))\chi_{I_{k,i,j}}\|_1
		\\&\lesssim
		 \Bigg(\sum\limits_{k\in\mathbb{Z}}2^{kq}\bigg\|\sum\limits_{i=0}^{\infty}\sum\limits_{j}\chi_{I_{k,i,j}}\bigg\|_{p(\cdot)}^q\gamma_b^q\bigg(\bigg\|\sum\limits_{i=0}^{\infty}\sum\limits_{j}\chi_{I_{k,i,j}}\bigg\|_{p(\cdot)}\bigg)\Bigg)^{\frac{1}{q}}\|g\|_{BMO_{1,q,b}(\alpha(\cdot))}
		\\&\lesssim
		 \|f\|_{\mathcal{P}_{p(\cdot),q,b}}\|g\|_{BMO_{1,q,b}(\alpha(\cdot))}.
	\end{align*}
	Then we get that $\ell_g$ can be extended to a bounded linear functional on $\mathcal{P}_{p(\cdot),q,b}$ and $\ell_g\in \big(\mathcal{P}_{p(\cdot),q,b}^{*}\big)_1$.
	
	Conversely, let $\ell\in\big(\mathcal{P}_{p(\cdot),q,b}^{*}\big)_1$. Then there exists $g\in L_1$ such that
	$$\ell(f)=\mathbb{E}(fg),\quad \forall\ f\in L_{\infty}.$$
	Let $(I_{k,i,j})_{k\in\mathbb{Z},i\in\mathbb{N},j}$ be an arbitrary sequence of atoms such that $I_{k,i,j}$ are disjoint if $k$ is fixed, $I_{k,i,j}$ belong to $\mathcal{F}_i$ and
	\begin{align*}
		 \Bigg\{2^k\bigg\|\sum\limits_{i=0}^{\infty}\sum\limits_{j}\chi_{I_{k,i,j}}\bigg\|_{p(\cdot)}\gamma_b\bigg(\bigg\|\sum\limits_{i=0}^{\infty}\sum\limits_{j}\chi_{I_{k,i,j}}\bigg\|_{p(\cdot)}\bigg)\Bigg\}_{k\in\mathbb{Z}}\in l_q.
	\end{align*}
	Let
	\[
		\varphi_{k,i,j} = \mbox{sign} (g-\mathbb{E}_i(g)) \chi_{I_{k,i,j}}
	\]
	and
	$$
	 h_{k,i,j}=\frac{\varphi_{k,i,j}-E_i(\varphi_{k,i,j})}{2\|\chi_{I_{k,i,j}}\|_{p(\cdot)}}.
	$$
	Then $h_{k,i,j}$ is a simple $(p(\cdot),\infty)^M$-atom and
    $$
    \mathbb{E}\big(h_{k,i,j}(g-\mathbb{E}_i(g))\big)=\frac{\|(g-\mathbb{E}_i(g))\chi_{I_{k,i,j}}\|_1}{2\|\chi_{I_{k,i,j}}\|_{p(\cdot)}}.
    $$
    Setting $\mu_{k,i,j}=3\cdot2^k\|\chi_{I_{k,i,j}}\|_{p(\cdot)}$, it follows from Theorem \ref{ad3} that
	 $$f=\sum\limits_{k\in\mathbb{Z}}\sum\limits_{i=0}^{\infty}\sum\limits_{j}\mu_{k,i,j}h_{k,i,j}\in\mathcal{P}_{p(\cdot),q,b}$$
    and $$\|f\|_{\mathcal{P}_{p(\cdot),q,b}}\lesssim\Bigg(\sum\limits_{k\in\mathbb{Z}}2^{kq}\bigg\|\sum\limits_{i=0}^{\infty}\sum\limits_{j}\chi_{I_{k,i,j}}\bigg\|_{p(\cdot)}^q\gamma_b^q\bigg(\bigg\|\sum\limits_{i=0}^{\infty}\sum\limits_{j}\chi_{I_{k,i,j}}\bigg\|_{p(\cdot)}\bigg)\Bigg)^{\frac{1}{q}}.$$
    For an arbitrary positive integer $N$, set
    $$(f)^N=\sum_{k=-N}^{N}\sum\limits_{i=0}^{\infty}\sum\limits_{j}\mu_{k,i,j}h_{k,i,j}.$$
    Then
    \begin{align*}
    	&\ \sum_{k=-N}^{N}\sum\limits_{i=0}^{\infty}\sum\limits_{j}2^k\|(g-\mathbb{E}_i(g))\chi_{I_{k,i,j}}\|_1\\
    	=&\ 2 \sum_{k=-N}^{N}\sum\limits_{i=0}^{\infty}\sum\limits_{j}2^k\|\chi_{I_{k,i,j}}\|_{p(\cdot)}\mathbb{E}\big(h_{k,i,j}(g-\mathbb{E}_i(g))\big)\\
    	=&\ 2\sum_{k=-N}^{N}\sum\limits_{i=0}^{\infty}\sum\limits_{j}2^k\|\chi_{I_{k,i,j}}\|_{p(\cdot)}\mathbb{E}(h_{k,i,j}g)\\
    	=&\ \frac{2}{3}\mathbb{E}((f)^Ng)= \frac{2}{3} \ell((f)^N)\le \frac{2}{3} \|\ell\|\|(f)^N\|_{\mathcal{P}_{p(\cdot),q,b}}.
    \end{align*}
	By the definition of $BMO_{1,q,b}(\alpha(\cdot))$ and by taking the supremum in $N$, we get that $$\|g\|_{BMO_{1,q,b}(\alpha(\cdot))}\lesssim \|\ell\|.$$
	The proof is finished now.
\end{proof}

\begin{thm}\label{jn3}
	Let $p(\cdot)\in\mathcal{P}(\Omega)$ satisfy {\rm(\ref{gs3})} with $0<p_-\le p_{+}\le1$, $0<q<\infty$ and let $b$ be a slowly varying function. If the stochastic basic $\{\mathcal{F}_n\}_{n\geq0}$ is regular, then $$\big(\mathcal{P}_{p(\cdot),q,b}^{*}\big)_1=\mathcal{P}_{p(\cdot),q,b}^{*}.$$
\end{thm}
\begin{proof}
	Since $\{\mathcal{F}_n\}_{n\geq0}$ is regular, we have $\mathcal{P}_{p(\cdot),q,b}\approx H_{p(\cdot),q,b}^s$.
	Thus $L_2$ can also be embedded continuously in $\mathcal{P}_{p(\cdot),q,b}$.
	Then $\mathcal{P}_{p(\cdot),q,b}^{*}\subset L_2^*=L_2$. For $\ell\in\mathcal{P}_{p(\cdot),q,b}^{*}$, there exists $g\in L_2\subset L_1$ such that $\ell=\ell_g$. It is obvious that $\ell\in \big(\mathcal{P}_{p(\cdot),q,b}^{*}\big)_1$, which means $\mathcal{P}_{p(\cdot),q,b}^{*}\subset \big(\mathcal{P}_{p(\cdot),q,b}^{*}\big)_1$. By the definition of $\big(\mathcal{P}_{p(\cdot),q,b}^{*}\big)_1$, we have $\big(\mathcal{P}_{p(\cdot),q,b}^{*}\big)_1\subset \mathcal{P}_{p(\cdot),q,b}^{*}$. Hence $\big(\mathcal{P}_{p(\cdot),q,b}^{*}\big)_1=\mathcal{P}_{p(\cdot),q,b}^{*}$.
\end{proof}

Combining Theorem $\ref{jn2}$ with Theorem $\ref{jn3}$, we have the dual theorem of $\mathcal{P}_{p(\cdot),q,b}$.
\begin{coro}\label{;}
	 Let $p(\cdot)\in\mathcal{P}(\Omega)$ satisfy $(\ref{gs3})$ with $0<p_-\le p_+\le1$, $0<q< \infty$ and let $ b$ be a slowly varying function. If the stochastic basic $\{\mathcal{F}_n\}_{n\geq0}$ is regular, then $$\mathcal{P}_{p(\cdot),q,b}^*=BMO_{1,q,b}(\alpha(\cdot)),\quad \alpha(\cdot)=\frac{1}{p(\cdot)}-1.$$
\end{coro}

Now, we state the John-Nirenberg theorem for variable Lorentz-Karamata spaces.
\begin{thm}
    Let $\alpha(\cdot)+1\in\mathcal{P}(\Omega)$ satisfy $(\ref{gs3})$ with $\alpha_-\ge0$, $0<q<\infty$ and let $b$ be a slowly varying function.
    If the stochastic basis $\{\mathcal{F}_n\}_{n\geq0}$ is regular, then
    $$BMO_{r,q,b}(\alpha(\cdot))=BMO_{2,q,b}(\alpha(\cdot))$$
    with equivalent norms for all $1\le r<\infty$.
\end{thm}
\begin{proof}
	When $1<r<\infty$, it follows from Theorems \ref{dual} and \ref{jn1} that
	$$BMO_{r,q,b}(\alpha(\cdot))=BMO_{2,q,b}(\alpha(\cdot)).$$
	Next we consider the case of $r=1$. Since the stochastic basis $\{\mathcal{F}_n\}_{n\geq0}$ is regular, we have $\mathcal{P}_{p(\cdot),q,b}\approx H_{p(\cdot),q,b}^s$, where $p(\cdot)=\frac{1}{\alpha(\cdot)+1}$. Hence, according to Theorem \ref{dual} and Corollary \ref{;}, we obtain that
	$$BMO_{1,q,b}(\alpha(\cdot))=BMO_{2,q,b}(\alpha(\cdot)).$$
\end{proof}

When $\alpha(\cdot)\equiv\alpha$ is a constant, we have the following result.
\begin{coro}\label{coro0}
	Let $\alpha\ge0$, $0<q<\infty$ and let $b$ be a slowly varying function.
	If the stochastic basis $\{\mathcal{F}_n\}_{n\geq0}$ is regular, then
	$$BMO_{r,q,b}(\alpha)=BMO_{2,q,b}(\alpha)$$
	with equivalent norms for all $1\le r<\infty$.
\end{coro}
\begin{rem}
    In $2015$, Jiao et al.  \cite{jxz} have showed the John-Nirenberg theorem on Lorentz-Karamata spaces. However, they need the condition that the slowly varying function $b$ is nondecreasing. Hence, Corollary $\ref{coro0}$ extends \cite[Theorem 1.6]{jxz}.
\end{rem}

When $b\equiv1$, we obtain the John-Nirenberg theorem associated with $BMO_{r,q}(\alpha(\cdot))$.
\begin{coro}
	Let $\alpha(\cdot)+1\in\mathcal{P}(\Omega)$ satisfy $(\ref{gs3})$ with $\alpha_-\ge0$ and $0<q<\infty$.
	If the stochastic basis $\{\mathcal{F}_n\}_{n\geq0}$ is regular, then
	$$BMO_{r,q}(\alpha(\cdot))=BMO_{2,q}(\alpha(\cdot))$$
	with equivalent norms for all $1\le r<\infty$.
\end{coro}

\section{Boundedness of fractional integrals} \label{s7}
In this section, without loss of generality, we always suppose that the constant in (\ref{R}) satisfies $\mathcal{R}\ge2$. Firstly, we introduce the definition of the fractional integral and give some necessary results.

\begin{defi}[\cite{nsm,sf}]
	For $f=(f_n)_{n\ge0}\in\mathcal{M}$ and $\alpha>0$, the fractional integral $I_\alpha f=\big((I_\alpha f)_n\big)_{n\ge0}$ of $f$ is defined as
	$$(I_\alpha f)_n=\sum_{k=1}^{n}b_{k-1}^\alpha d_kf,$$
	where $b_k$ is an $\mathcal{F}_k$-measurable function such that for all $B\in A(\mathcal{F}_k)$ and $\omega\in B$, $b_k(\omega)=\mathbb{P}(B)$.
\end{defi}

\begin{rem}
	We point out that $I_\alpha f$ is a martingale and $(I_\alpha f)_n$ is a martingale transform introduced by Burkholder. Especially, if $\Omega=[0,1]$ and the $\sigma$-algebra $\mathcal{F}_n$ is generated by the dyadic intervals of $[0,1]$, then $I_\alpha f$ is closely related to a class of multiplier transformations of Walsh-Fourier series, see \cite{wwwww}. We refer to \cite{skm} for the classical fractional integrals.
\end{rem}

\begin{lem}[\cite{hj,jzhc}]\label{4.2}
	Let $p(\cdot),q(\cdot)\in\mathcal{P}(\Omega)$ satisfy $(\ref{gs3})$. For any set $A\in \bigcup\limits_{n\ge0}A(\mathcal{F}_n)$, we have
	 $$\|\chi_A\|_{r(\cdot)}\approx\|\chi_A\|_{p(\cdot)}\|\chi_A\|_{q(\cdot)},$$
	where
	$$\frac{1}{r(\omega)}=\frac{1}{p(\omega)}+\frac{1}{q(\omega)},\quad \forall\  \omega\in\Omega.$$
\end{lem}

\begin{lem}[\cite{sf}]\label{4.1}
	Let $\{\mathcal{F}_n\}_{n\ge0}$ be regular, $f\in\mathcal{M}$, $\alpha>0$ and let $\mathcal{R}$ be the constant in $(\ref{R})$. If there exists $B\in\mathcal{F}$ such that $M(f)\le\chi_B$, then there exists a positive constant $C_\alpha=2+\frac{\mathcal{R}+1}{1-(1+\frac{1}{\mathcal{R}})^{\alpha-1}}$ independent of $f$ and $B$ such that
	$$M(I_\alpha f)\le C_\alpha\mathbb{P}(B)^\alpha\chi_B.$$
\end{lem}

\begin{lem}\label{6.3}
	Let $p_1(\cdot),p_2(\cdot)\in\mathcal{P}(\Omega)$ satisfy $(\ref{gs3})$, $\alpha>0$, $\{\mathcal{F}_n\}_{n\ge0}$ be regular, $b$ be a nondecreasing slowly varying function and let $\mathcal{R}$ be the constant in $(\ref{R})$. If $0<p_1(\cdot)<p_2(\cdot)<\infty$, $\alpha\ge\sup\limits_{\omega\in\Omega}\big(\frac{1}{p_1(\omega)}-\frac{1}{p_2(\omega)}\big)$, $0<q\le\infty$ and $a$ is a simple $(p_1(\cdot),\infty)^M$-atom, then
	$$\|I_\alpha a\|_{H_{p_2(\cdot),q,b}^M}\lesssim C_\alpha\gamma_b(\|\chi_{I}\|_{p_1(\cdot)}),$$
	where $C_\alpha$ is the same as in Lemma $\ref{4.1}$ and $I\in A(\mathcal{F}_n)$ is associated with $a$.
\end{lem}
\begin{proof}
	Let $I\in A(\mathcal{F}_n)$ be associated with the simple $(p_1(\cdot),\infty)^M$-atom $a$. We have $M(a)\le\|\chi_I\|_{p_1(\cdot)}^{-1}\chi_I$ and
	 $$M\big(\|\chi_I\|_{p_1(\cdot)}a\big)=\|\chi_I\|_{p_1(\cdot)}M(a)\le\chi_I.$$
	It follows from Lemma \ref{4.1} that
	$$M\big(I_\alpha(\|\chi_I\|_{p_1(\cdot)}a)\big)\le C_\alpha\mathbb{P}(I)^\alpha\chi_I.$$
	For $p_1(\cdot),p_2(\cdot)\in\mathcal{P}(\Omega)$ with $p_1(\cdot)<p_2(\cdot)$, we can find a variable exponent $r(\cdot)\in\mathcal{P}(\Omega)$ such that
	 $$\frac{1}{r(\omega)}=\frac{1}{p_1(\omega)}-\frac{1}{p_2(\omega)},\quad \forall\ \omega\in\Omega.$$
	Then $\sup\limits_{\omega\in\Omega}\frac{1}{r(\omega)}\le\alpha$, which means
	$r_-(\Omega)\ge\frac{1}{\alpha}$.
	By applying Lemma \ref{4.2}, we have
	\begin{align*}
		 \|\chi_I\|_{p_1(\cdot)}\approx\|\chi_I\|_{p_2(\cdot)}\|\chi_I\|_{r(\cdot)}\ge\|\chi_I\|_{p_2(\cdot)}\|\chi_I\|_{\frac{1}{\alpha}}=\|\chi_I\|_{p_2(\cdot)}\mathbb{P}(I)^\alpha.
	\end{align*}
    Hence there is
    \begin{align*}
    	M(I_\alpha a)\le C_\alpha\mathbb{P}(I)^\alpha\|\chi_I\|_{p_1(\cdot)}^{-1}\chi_I\lesssim C_\alpha\|\chi_I\|_{p_2(\cdot)}^{-1}\chi_I.
    \end{align*}
    Since $b$ is nondecreasing, we know that $\gamma_b$ is nonincreasing on $(0,1]$. Then it follows from Lemma \ref{2.4} that, for $p_1(\cdot)<p_2(\cdot)$,
    $$\gamma_b\big(\|\chi_{I}\|_{p_2(\cdot)}\big)\lesssim\gamma_b\big(\|\chi_{I}\|_{p_1(\cdot)}\big).$$
    It yields that
    \begin{align*}
    	\|I_\alpha a\|_{H_{p_2(\cdot),q,b}^M}=&\ \|M(I_\alpha a)\|_{p_2(\cdot),q,b}\\
    	\lesssim&\ C_\alpha\|\chi_I\|_{p_2(\cdot)}^{-1}\|\chi_I\|_{p_2(\cdot),q,b}\\
    	\approx&\ C_\alpha\|\chi_I\|_{p_2(\cdot)}^{-1}\|\chi_I\|_{p_2(\cdot)}\gamma_b\big(\|\chi_I\|_{p_2(\cdot)}\big)\\
    	\lesssim&\ C_\alpha\gamma_b\big(\|\chi_I\|_{p_1(\cdot)}\big).
    \end{align*}
    Therefore, the proof is complete.
\end{proof}

Now, we show the main conclusion of this section.

\begin{thm}\label{thmfi}
Let $\{\mathcal{F}_n\}_{n\ge0}$ be regular, $p(\cdot),q(\cdot)\in\mathcal{P}(\Omega)$ satisfy $(\ref{gs3})$, $0<s\le\infty$, $\alpha>0$ and $b_1,b_2$ be slowly varying functions. Assume that $\mathcal{N}$ is the constant in Remark $\ref{art}$. If $0<p(\cdot)<q(\cdot)<\infty$, $\alpha\ge\sup\limits_{\omega\in\Omega}\big(\frac{1}{p(\omega)}-\frac{1}{q(\omega)}\big)$, $b_1$ is nondecreasing and $\sup\limits_{1\le t<\infty}\frac{b_2(t)}{b_1(t)}<\infty$, then for any constant $r$ with $p_+<r\le\min\{s,\mathcal{N}\}$, there is
	$$
	\|I_\alpha f\|_{H_{q(\cdot),s,b_2}^M}\lesssim\|f\|_{H_{p(\cdot),r,b_1}^M}
	$$
	for all $f\in{H_{p(\cdot),r,b_1}^M}$.
\end{thm}
\begin{proof}
	Let $f\in H_{p(\cdot),r,b_1}^M$. Since $\{\mathcal{F}_n\}_{n\ge0}$ is regular,  there exist a sequence of simple $(p(\cdot),\infty)^M$-atoms $(a^{k,i,j})_{k\in\mathbb{Z},i\in\mathbb{N},j}$ and $\mu_{k,i,j}=3\cdot2^k\|\chi_{I_{k,i,j}}\|_{p(\cdot)}$ such that for all $n\in\mathbb{N}$,
	 $$f=\sum_{k\in\mathbb{Z}}\sum_{i=0}^{\infty}\sum_{j}\mu_{k,i,j}a^{k,i,j}\qquad a.e.$$
	and
	 $$\Bigg\|\Bigg\{\bigg\|\sum_{i=0}^{\infty}\sum_{j}\mu_{k,i,j}\|\chi_{I_{k,i,j}}\|_{p(\cdot)}^{-1}\chi_{I_{k,i,j}}\bigg\|_{p(\cdot)}\gamma_{b_1}\bigg(\bigg\|\sum_{i=0}^{\infty}\sum_{j}\chi_{I_{k,i,j}}\bigg\|_{p(\cdot)}\bigg)\Bigg\}_{k\in\mathbb{Z}}\Bigg\|_{l_{r}}\lesssim\|f\|_{H_{p(\cdot),r,b_1}^M}.$$
	According to the sublinearity of $M$, we have
	\begin{align*}
		\|I_\alpha f\|_{H_{q(\cdot),s,b_2}^M}=&\ \|M(I_\alpha f)\|_{q(\cdot),s,b_2}=\Bigg\|M\Bigg(I_\alpha\bigg(\sum_{k\in\mathbb{Z}}\sum_{i=0}^{\infty}\sum_{j}\mu_{k,i,j}a^{k,i,j}\bigg)\Bigg)\Bigg\|_{q(\cdot),s,b_2}\\
		\le&\ \bigg\|\sum_{k\in\mathbb{Z}}\sum_{i=0}^{\infty}\sum_{j}|\mu_{k,i,j}|M\big(I_\alpha(a^{k,i,j})\big)\bigg\|_{q(\cdot),s,b_2}.
	\end{align*}
	It follows from Remark \ref{art} that for any constant $r$ with $p_+<r\le\min\{s,\mathcal{N}\}$,
	 $$\bigg\|\sum_{k\in\mathbb{Z}}\sum_{i=0}^{\infty}\sum_{j}|\mu_{k,i,j}|M\big(I_\alpha(a^{k,i,j})\big)\bigg\|_{q(\cdot),s,b_2}^{r}\le4^{\frac{r}{\mathcal{N}}}\sum_{k\in\mathbb{Z}}\sum_{i=0}^{\infty}\sum_{j}\big\||\mu_{k,i,j}|M\big(I_\alpha(a^{k,i,j})\big)
    \big\|_{q(\cdot),s,b_2}^{r}.$$
	Hence, according to Lemmas \ref{baohan}, \ref{6.3} and \ref{wuqiongbianb}, we have
    \begin{align*}
		\|I_\alpha f\|_{H_{q(\cdot),s,b_2}^M}^{r}\le&\ 4\sum_{k\in\mathbb{Z}}\sum_{i=0}^{\infty}\sum_{j}|\mu_{k,i,j}|^{r}\big\|M\big(I_\alpha(a^{k,i,j})\big)\big\|_{q(\cdot),s,b_2}^{r}\\
		\lesssim&\ 4\sum_{k\in\mathbb{Z}}\sum_{i=0}^{\infty}\sum_{j}|\mu_{k,i,j}|^{r}\big\|M\big(I_\alpha(a^{k,i,j})\big)\big\|_{q(\cdot),l,b_1}^{r}\\
		\lesssim&\ C_\alpha^r\sum_{k\in\mathbb{Z}}\sum_{i=0}^{\infty}\sum_{j}|\mu_{k,i,j}|^{r}\gamma_{b_1}^{r}(\|\chi_{I_{k,i,j}}\|_{p(\cdot)})\\
		=&\ C_\alpha^r\sum_{k\in\mathbb{Z}}3^r\cdot2^{kr}\sum_{i=0}^{\infty}\sum_{j}\|\chi_{I_{k,i,j}}\|_{p(\cdot)}^{r}\gamma_{b_1}^{r}(\|\chi_{I_{k,i,j}}\|_{p(\cdot)})\\
		\lesssim&\ \sum_{k\in\mathbb{Z}}3^r\cdot2^{kr}\bigg\|\sum_{i=0}^{\infty}\sum_{j}\chi_{I_{k,i,j}}\bigg\|_{p(\cdot)}^{r}\gamma_{b_1}^{r}\bigg(\bigg\|\sum_{i=0}^{\infty}\sum_{j}\chi_{I_{k,i,j}}\bigg\|_{p(\cdot)}\bigg)\\
		=&\ \Bigg\|\Bigg\{\bigg\|\sum_{i=0}^{\infty}\sum_{j}\mu_{k,i,j}\|\chi_{I_{k,i,j}}\|_{p(\cdot)}^{-1}\chi_{I_{k,i,j}}\bigg\|_{p(\cdot)}\gamma_{b_1}\bigg(\bigg\|\sum_{i=0}^{\infty}\sum_{j}\chi_{I_{k,i,j}}\bigg\|_{p(\cdot)}\bigg)\Bigg\}_{k\in\mathbb{Z}}\Bigg\|_{l_{r}}^{r}\\
		\lesssim&\ \|f\|_{H_{p(\cdot),r,b_1}^M}^{r}.
	\end{align*}
    The proof of this theorem is complete now.
\end{proof}

When $p(\cdot)\equiv p$ and $q(\cdot)\equiv q$, Theorem \ref{thmfi} reduces to \cite[Theorem 4.4]{lz}.
When $b_1\equiv b_2\equiv1$, we obtain the boundedness of fractional integral on variable Hardy-Lorentz spaces.
\begin{coro}
	Let $\{\mathcal{F}_n\}_{n\ge0}$ be regular, $p_1(\cdot),p_2(\cdot)\in\mathcal{P}(\Omega)$ satisfy $(\ref{gs3})$, $0<q_2\le\infty$ and $\alpha>0$. Suppose that $L_{p_2(\cdot),q_2}$ is an $\mathcal{N}$-normed space. If $0<p_1(\cdot)<p_2(\cdot)<\infty$ and $\alpha\ge\sup\limits_{\omega\in\Omega}\big(\frac{1}{p_1(\omega)}-\frac{1}{p_2(\omega)}\big)$, then for any constant $q_1$ with $(p_{1})_+<q_1\le\min\{q_2,\mathcal{N}\}$, there is
	$$\|I_\alpha f\|_{H_{p_2(\cdot),q_2}^M}\lesssim\|f\|_{H_{p_1(\cdot),q_1}^M}
	$$
	for all $f\in{H_{p_1(\cdot),q_1}^M}$.
\end{coro}

\section{Applications in Fourier Analysis} \label{s8}
As an application of the previous results in Fourier analysis, we will investigate the boundedness of the maximal Fej\'{e}r operator on variable Hardy-Lorentz-Karamata spaces.
We consider the probability space $([0,1),\mathcal{F},dx)$, where $\mathcal{F}$ denotes the Lebesgue measurable sets. A dyadic interval means an interval of the form $[k2^{-n},(k+1)2^{-n})$ for some $k,n\in\mathbb{N}$, $0\le k<2^{n}$. Given any $n\in\mathbb{N}$ and $x\in[0,1)$, denote $I_n(x)$ the dyadic interval of length $2^{-n}$ which contains $x$. The $\sigma$-algebra $\mathcal{F}_n$ generated by the dyadic intervals $\{I_n(x):x\in[0,1)\}$ is called the $n$th dyadic $\sigma$-algebra. In this section, we always assume that $\{\mathcal{F}_n\}_{n\ge0}$ is a sequence of dyadic $\sigma$-algebras. Obviously, $\{\mathcal{F}_n\}_{n\ge0}$ is regular. According to Theorem $\ref{mi}$, the five variable martingale Hardy-Lorentz-Karamata spaces are equivalent. Hence, denote by $H_{p(\cdot),q,b}$ one of them.

\subsection{Walsh system and partial sums}\label{9.1}

For any $n\in\mathbb{N}$ and $x\in[0,1)$, let
$$r_n(x):=\text{sgn}\sin(2^n\pi x).$$
The set of $\{r_n(x)\}_{n\in\mathbb{N}}$ is called to be the system of Rademacher functions. The product system generated by the Rademacher functions is the Walsh system:
$$w_n:=\mathop{\prod}\limits_{k=0}^{\infty}r_k^{n_k},\quad n\in\mathbb{N},$$
where
\begin{eqnarray}\label{8.1}
	n=\sum_{k=0}^{\infty}n_k2^k,\quad n_k=0,1.
\end{eqnarray}

If $f\in L_1$, then $\widehat{f}(n):=\mathbb{E}(fw_n)\ (n\in\mathbb{N})$ is said to be the $n$-th Walsh-Fourier coefficient of $f$. We remember that $\mathop{\lim}\limits_{k\rightarrow\infty}\mathbb{E}_k(f)=f$ in the $L_1$-norm. Hence,
$$\widehat{f}(n)=\lim_{k\rightarrow\infty}\mathbb{E}\big((\mathbb{E}_kf)w_n\big),\quad n\in\mathbb{N}.$$
If $f=(f_k)_{k\ge0}$ is a martingale, then the Walsh-Fourier coefficients of $f$ are defined by
$$\widehat{f}(n):=\lim_{k\rightarrow\infty}\mathbb{E}(f_kw_n),\quad n\in\mathbb{N}.$$
Since $w_n$ is $\mathcal{F}_k$-measurable for $n<2^k$, it is easy to see that this limit does exist. Moreover, the Walsh-Fourier coefficients of $f\in L_1$ are the same as those of the martingale $\big(\mathbb{E}_k(f)\big)_{k\ge0}$ obtained from $f$.

From \cite{f}, we recall that the Walsh-Dirichlet kernels
$$D_n:=\sum\limits_{k=0}^{n-1}w_k,\quad n\in\mathbb{N}
$$
satisfy
$$D_{2^n}(x)=\left\{
\begin{aligned}
	&2^n,\ &\mathrm{if}\ x\in[0,2^{-n}),\\
	&0,\ &\mathrm{if}\ x\in[2^{-n},1),
\end{aligned}\right.\quad n\in\mathbb{N}.$$

Denote by $s_nf$ the $n$-th partial sum of the Walsh-Fourier series of a martingale $f$, that is
$$s_nf:=\mathop{\sum}\limits_{k=0}^{n-1}\widehat{f}(k)w_k.$$
If $f\in L_1$, then
$$s_nf=\int_{0}^{1}f(t)D_n(x\dot{+}t)dt,\quad n\in\mathbb{N},$$
where $\dot{+}$ denotes the dyadic addition (see \cite{swsp} or \cite{ws}). We get immediately that
$$s_{2^n}f=f_n,\quad n\in\mathbb{N}.$$
By martingale results, when $f\in L_p$ with $1\le p<\infty$, then
$$\lim_{n\rightarrow\infty}s_{2^n}f=f\quad \mathrm{in\ the}\ L_p\mbox{-}\mathrm{norm}.$$
This result was extended by Schipp et al. \cite{swsp} to the partial sums $s_nf$: when $f\in L_p$ with $1<p<\infty$, then
$$\lim_{n\rightarrow\infty}s_{n}f=f\quad \mathrm{in\ the}\ L_p\mbox{-}\mathrm{norm}.$$
Recently, Jiao et al. \cite{jwzw} generalized this result to variable Lebesgue spaces and variable Lorentz spaces. In this subsection, we extend these conclusions to variable Lorentz-Karamata spaces.
\begin{thm}
	Let $p(\cdot)\in\mathcal{P}(\Omega)$ satisfy $(\ref{gs3})$ with $1<p_-\le p_+<\infty$, $0<q\le\infty$ and let $b$ be a slowly varying function. If $f\in L_{p(\cdot),q,b}$, then
	 $$\sup_{n\in\mathbb{N}}\|s_nf\|_{p(\cdot),q,b}\lesssim\|f\|_{p(\cdot),q,b}.$$
\end{thm}
\begin{proof}
	Set $$T_0f:=\sum_{k=1}^{\infty}n_{k-1}d_kf,$$
	where the binary coefficients $n_k$ are defined in (\ref{8.1}). It follows from \cite{swsp} or \cite{w} that
	$$s_nf=w_nT_0(fw_n).$$
	Obviously, $T_0$ is a martingale transform and $S(T_0f)\le S(f)$. Since $\{\mathcal{F}_n\}_{n\ge0}$ is regular, it follows from Theorem \ref{mt} that
	 $$\|s_nf\|_{p(\cdot),q,b}=\|T_0f\|_{p(\cdot),q,b}\lesssim\|f\|_{p(\cdot),q,b}.$$
\end{proof}

\begin{coro}
	Let $p(\cdot)\in\mathcal{P}(\Omega)$ satisfy $(\ref{gs3})$ with $1<p_-\le p_+<\infty$, $0<q<\infty$ and let $b$ be a slowly varying function. If $f\in L_{p(\cdot),q,b}$, then
	$$\lim_{n\rightarrow\infty}s_{n}f=f\quad in\ the\ L_{p(\cdot),q,b}\mbox{-}norm.$$
\end{coro}
\begin{proof}
	Since the Walsh polynomials are dense in $L_{p(\cdot),q,b}$ by Lemma \ref{dct}, the proof follows from the usual density argument.
\end{proof}

Similarly, we have the following result.
\begin{coro}
	Let $p(\cdot)\in\mathcal{P}(\Omega)$ satisfy $(\ref{gs3})$ with $1<p_-\le p_+<\infty$ and let $b$ be a slowly varying function. If $f\in \mathcal{L}_{p(\cdot),\infty,b}$, then
	$$\lim_{n\rightarrow\infty}s_{n}f=f\quad in\ the\ \mathcal{L}_{p(\cdot),\infty,b}\mbox{-}\mathrm{norm}.$$
\end{coro}


\subsection{The maximal Fej\'{e}r operator $\sigma_*$}
Notice that for $p_-=1$, the results in Subsection \ref{9.1} are not true anymore. To be able to extend these theorems, we introduce the Fej\'{e}r summability method. The Fej\'{e}r mean of order $n\in\mathbb{N}$ of the Walsh-Fourier series of a martingale $f$ is given by
$$\sigma_nf:=\frac{1}{n}\sum_{k=1}^{n}s_kf.$$
Compared to $s_kf$, $\sigma_nf$ has better convergence properties.
The maximal operator $\sigma_*$ is defined by
$$\sigma_*f:=\sup_{n\in\mathbb{N}}|\sigma_nf|.$$
In this part, we mainly discuss the boundedness of $\sigma_*$ from variable Hardy-Lorentz-Karamata spaces to variable Lorentz-Karamata spaces. For classical Hardy spaces, Fujii \cite{fa} found that $\sigma_*$ is bounded from $H_1$ to $L_1$, see also \cite{ss}. In 1996, Weisz \cite{wc} generalized the above result and proved that $\sigma_*$ is bounded from $H_p$ to $L_p$ with $\frac{1}{2}<p<\infty$ and from $H_{p,q}$ to $L_{p,q}$ with $\frac{1}{2}<p<\infty$, $0<q\le\infty$. Recently, Jiao et al. \cite{jwzw} considered the maximal Fej\'{e}r operator on $H_{p(\cdot)}$ and $H_{p(\cdot),q}$.
We refer to \cite{fa,gg,ss,s,sw,wc} for more details.

\begin{thm}\label{T2}
	Let $p(\cdot)\in\mathcal{P}(\Omega)$ satisfy $(\ref{gs3})$, $0<q\le\infty$, $b$ be a slowly varying function and let $\max\{p_+,1\}<r\le\infty$. Suppose that $T:L_r\rightarrow L_r$ is a bounded $\sigma$-sublinear operator and for some $0<\beta<1$ and every $(p(\cdot),\infty)^M$-atom $a$ associated with stopping time $\tau$, there is
	\begin{align}\label{gs8.6}
		 \big\||T(a)|^\beta\chi_{\{\tau=\infty\}}\big\|_{p(\cdot)}\lesssim\|\chi_{\{\tau<\infty\}}\|_{p(\cdot)}^{1-\beta}.
	\end{align}
	Then for $f\in H_{p(\cdot),q,b}$,
	$$\|T(f)\|_{p(\cdot),q,b}\lesssim\|f\|_{H_{p(\cdot),q,b}}.$$
\end{thm}
\begin{proof}
	Let $f\in H_{p(\cdot),q,b}$. According to Theorem \ref{ad11} and Corollary \ref{cc11}, there exist a sequence $(a^k)_{k\in\mathbb{Z}}$ of $(p(\cdot),\infty)^M$-atoms associated with the stopping times $(\tau_k)_{k\in\mathbb{Z}}$ and a sequence $(\mu_k)_{k\in\mathbb{Z}}=\big(3\cdot 2^k\|\chi_{\{\tau_k<\infty\}}\|_{p(\cdot)}\big)_{k\in\mathbb{Z}}$ of positive numbers such that
	$$f=\sum_{k\in\mathbb{Z}}\mu_{k}a^{k}.$$
	For an arbitrary integer $k_0$, set
	$$D_1:=\sum_{k=-\infty}^{k_0-1}\mu_{k}|T(a^{k})|\quad \text{and}\quad D_2:=\sum_{k=k_0}^{\infty}\mu_{k}|T(a^{k})|.$$
	Then by the $\sigma$-sublinearity of $T$, there is
	$$|T(f)|\le\sum_{k\in\mathbb{Z}}\mu_{k}|T(a^{k})|=D_1+D_2.$$
	
	Firstly, we consider the case of $r=\infty$. Since $T:L_\infty\rightarrow L_\infty$ is bounded, we have
	\begin{align*}
		\|D_1\|_{\infty}\le&\ \sum_{k=-\infty}^{k_0-1}\mu_{k}\|T(a^{k})\|_\infty\le\sum_{k=-\infty}^{k_0-1}\mu_{k}\|M(a^{k})\|_\infty\\
		\le&\ \sum_{k=-\infty}^{k_0-1}\mu_{k}\|\chi_{\{\tau_k<\infty\}}\|_{p(\cdot)}^{-1}\le3\cdot2^{k_0}.
	\end{align*}
 We decompose $D_2$ into $X_1+X_2$, where
    $$X_1:=\sum_{k=k_0}^{\infty}\mu_{k}|T(a^{k})|\chi_{\{\tau_k<\infty\}}\ \ \text{and}\ \ X_2:=\sum_{k=k_0}^{\infty}\mu_{k}|T(a^{k})|\chi_{\{\tau_k=\infty\}}.$$
    It is obvious that
    $$\{|X_1|>2^{k_0}\}\subset\{|X_1|>0\}\subset\bigcup_{k=k_0}^{\infty}\{\tau_k<\infty\}.$$
    Hence we have
    $$\|\chi_{\{|X_1|>2^{k_0}\}}\|_{p(\cdot)}\le\bigg\|\sum_{k=k_0}^{\infty}\chi_{\{\tau_k<\infty\}}\bigg\|_{p(\cdot)}.$$
    Let $0<m<\underline{p}$ and $0<\varepsilon<1$. It follows from Lemma \ref{wuqiongbianb} that
    \begin{align*}
    	&\ 2^{k_0m\varepsilon}\|\chi_{\{|X_1|>2^{k_0}\}}\|_{p(\cdot)}^m\gamma_b^m\big(\|\chi_{\{|X_1|>2^{k_0}\}}\|_{p(\cdot)}\big)\\\notag
    	\lesssim&\ 2^{k_0m\varepsilon}\bigg\|\sum_{k=k_0}^{\infty}\chi_{\{\tau_k<\infty\}}\bigg\|_{p(\cdot)}^m\gamma_b^m\bigg(\bigg\|\sum_{k=k_0}^{\infty}\chi_{\{\tau_k<\infty\}}\bigg\|_{p(\cdot)}\bigg)\\\notag
    	\lesssim&\ 2^{k_0m\varepsilon}\sum_{k=k_0}^{\infty}\|\chi_{\{\tau_k<\infty\}}\|_{p(\cdot)}^m\gamma_b^m\big(\|\chi_{\{\tau_k<\infty\}}\|_{p(\cdot)}\big)\\\notag
    	\lesssim&\ \sum_{k=k_0}^{\infty}2^{km\varepsilon}\|\chi_{\{\tau_k<\infty\}}\|_{p(\cdot)}^m\gamma_b^m\big(\|\chi_{\{\tau_k<\infty\}}\|_{p(\cdot)}\big).
    \end{align*}
    Obviously, according to Lemma \ref{belong} and Corollary \ref{cc11}, we obtain
    \begin{align*}
    	\|D_1+X_1\|_{p(\cdot),q,b}\lesssim&\ \Big\|\big\{2^k\|\chi_{\{\tau_k<\infty\}}\|_{p(\cdot)}\gamma_b\big(\|\chi_{\{\tau_k<\infty\}}\|_{p(\cdot)}\big)\big\}_{k\in\mathbb{Z}}\Big\|_{l_q}\lesssim\|f\|_{H_{p(\cdot),q,b}}.
    \end{align*}
    Next, we investigate $X_2$.
    Let $0<\xi<\min\{\underline{p},q\}$. It follows from (\ref{gs8.6}) that for any $0<\beta<1$,
    \begin{align}\label{e200}
    	\|\chi_{\{|X_2|>2^{k_0}\}}\|_{p(\cdot)}\le&\ \frac{1}{2^{\beta k_0}}\big\||X_2|^{\beta}\big\|_{p(\cdot)}\\\notag
    	\lesssim&\ 2^{-k_0\beta}\bigg\|\sum_{k=k_0}^{\infty}\mu_{k}^{\beta}|T(a^{k})|^{\beta}\chi_{\{\tau_k=\infty\}}\bigg\|_{p(\cdot)}^{\xi\frac{1}{\xi}}\\\notag
    	\lesssim&\ 2^{-k_0\beta}\bigg(\sum_{k=k_0}^{\infty}\mu_{k}^{\beta\xi}\big\||T(a^{k})|^{\beta}\chi_{\{\tau_k=\infty\}}\big\|_{p(\cdot)}^\xi\bigg)^{\frac{1}{\xi}}\\\notag
    	\lesssim&\ 2^{-k_0\beta}\bigg(\sum_{k=k_0}^{\infty}\mu_{k}^{\beta\xi}\|\chi_{\{\tau_k<\infty\}}\|_{p(\cdot)}^{\xi-\beta\xi}\bigg)^{\frac{1}{\xi}}\\\notag
    	\approx&\ 2^{-k_0\beta}\bigg(\sum_{k=k_0}^{\infty}2^{k\beta\xi}\|\chi_{\{\tau_k<\infty\}}\|_{p(\cdot)}^{\xi}\bigg)^{\frac{1}{\xi}}.\notag
    \end{align}
    Firstly, we consider the case of $0<q<\infty$. Let $\beta<\alpha<1$. According to H\"{o}lder's inequality, we obtain
    \begin{align*}
    	\|\chi_{\{|X_2|>2^{k_0}\}}\|_{p(\cdot)}\lesssim&\ 2^{-k_0\beta}
    	 \bigg(\sum_{k=k_0}^{\infty}2^{k(\beta-\alpha)\xi\frac{q}{q-\xi}}\bigg)^{\frac{q-\xi}{q\xi}}\bigg(\sum_{k=k_0}^{\infty}2^{kq\alpha}\|\chi_{\{\tau_k<\infty\}}\|_{p(\cdot)}^q\bigg)^{\frac{1}{q}}\\
    	\lesssim&\ 2^{-k_0\alpha}\bigg(\sum_{k=k_0}^{\infty}2^{kq\alpha}\|\chi_{\{\tau_k<\infty\}}\|_{p(\cdot)}^q\bigg)^{\frac{1}{q}}.
    \end{align*}
    Then we get
    \begin{align*}
    	&\ \|X_2\|_{p(\cdot),q,b}^q\approx\sum_{k_0\in\mathbb{Z}}2^{k_0q}\|\chi_{\{|X_2|>2^{k_0}\}}\|_{p(\cdot)}^q\gamma_{b}^q\big(\|\chi_{\{|X_2|>2^{k_0}\}}\|_{p(\cdot)}\big)\\\notag
    	\lesssim&\ \sum_{k_0\in\mathbb{Z}}2^{k_0q}\sum_{k=k_0}^{\infty}2^{(k-k_0)q\alpha}\|\chi_{\{\tau_k<\infty\}}\|_{p(\cdot)}^q\gamma_{b}^q\bigg[\bigg(\sum_{k=k_0}^{\infty}2^{(k-k_0)q\alpha}\|\chi_{\{\tau_k<\infty\}}\|_{p(\cdot)}^q\bigg)^{\frac{1}{q}}\bigg].\notag
    \end{align*}
    Define $b_1(t)=b\big(t^{\frac{1}{q}}\big)$ for $t\in[1,\infty)$. Set $0<\theta<1$, it follows from Lemma \ref{wuqiongheb} that
    \begin{align*}
    	&\ \|X_2\|_{p(\cdot),q,b}^q\\\notag
    	\lesssim&\ \sum_{k_0\in\mathbb{Z}}2^{k_0q}\sum_{k=k_0}^{\infty}2^{(k-k_0)\alpha q}\|\chi_{\{\tau_k<\infty\}}\|_{p(\cdot)}^{q}\gamma_{b_1^{q}}\bigg(\sum_{k=k_0}^{\infty}2^{(k-k_0)\alpha q}\|\chi_{\{\tau_k<\infty\}}\|_{p(\cdot)}^{q}\bigg)\\
    	=&\ \sum_{k_0\in\mathbb{Z}}2^{k_0q}\Bigg[\bigg(\sum_{k=k_0}^{\infty}2^{(k-k_0)\alpha q}\|\chi_{\{\tau_k<\infty\}}\|_{p(\cdot)}^{q}\bigg)^{\theta}\gamma_{b_1^{\theta q}}\bigg(\sum_{k=k_0}^{\infty}2^{(k-k_0)\alpha q}\|\chi_{\{\tau_k<\infty\}}\|_{p(\cdot)}^{q}\bigg)\Bigg]^{\frac{1}{\theta}}\\
    	\lesssim&\ \sum_{k_0\in\mathbb{Z}}2^{k_0q}\Bigg(\sum_{k=k_0}^{\infty}2^{(k-k_0)\alpha \theta q}\|\chi_{\{\tau_k<\infty\}}\|_{p(\cdot)}^{\theta q}\gamma_{b_1^{\theta q}}\big(2^{(k-k_0)\alpha q}\|\chi_{\{\tau_k<\infty\}}\|_{p(\cdot)}^{q}\big)\Bigg)^{\frac{1}{\theta}}\\
    	=&\ \sum_{k_0\in\mathbb{Z}}2^{k_0q}\Bigg(\sum_{k=k_0}^{\infty}2^{(k-k_0)\alpha \theta q}\|\chi_{\{\tau_k<\infty\}}\|_{p(\cdot)}^{\theta q}\gamma_{b}^{\theta q}\big(2^{(k-k_0)\alpha}\|\chi_{\{\tau_k<\infty\}}\|_{p(\cdot)}\big)\Bigg)^{\frac{1}{\theta}}.
    \end{align*}
    Let $0<z<\frac{1-\alpha}{\alpha}$. By the same method as inequality (\ref{gs26}), we obtain that for $k\ge k_0$,
    $$\gamma_{b}\big(2^{(k-k_0)\alpha}\|\chi_{\{\tau_k<\infty\}}\|_{p(\cdot)}\big)\lesssim2^{(k-k_0)\alpha z}\gamma_b\big(\|\chi_{\{\tau_k<\infty\}}\|_{p(\cdot)}\big).$$
    Hence, there is
    \begin{align*}
    	 \|X_2\|_{p(\cdot),q,b}^q\lesssim\sum_{k_0\in\mathbb{Z}}2^{k_0q}\Bigg[\sum_{k=k_0}^{\infty}2^{(k-k_0)\alpha \theta q(1+z)}\|\chi_{\{\tau_k<\infty\}}\|_{p(\cdot)}^{\theta q}\gamma_b^{\theta q}\big(\|\chi_{\{\tau_k<\infty\}}\|_{p(\cdot)}\big)\Bigg]^{\frac{1}{\theta}}.
    \end{align*}
    Set $0<\zeta<\frac{1-\alpha-\alpha z}{\alpha}$. By applying H\"{o}lder's inequality with $1-\theta+\theta=1$ again, we have
    \begin{align*}
    	&\ \|X_2\|_{p(\cdot),q,b}^q\\
    	\lesssim&\ \sum_{k_0\in\mathbb{Z}}2^{k_0q}\Bigg[\sum_{k=k_0}^{\infty}2^{-(k-k_0)\alpha \theta q\zeta}2^{(k-k_0)\alpha\theta q(1+z+\zeta)}\|\chi_{\{\tau_k<\infty\}}\|_{p(\cdot)}^{\theta q}\gamma_b^{\theta q}\big(\|\chi_{\{\tau_k<\infty\}}\|_{p(\cdot)}\big)\Bigg]^{\frac{1}{\theta}}\\
    	\le&\ \sum_{k_0\in\mathbb{Z}}2^{k_0q}\bigg(\sum_{k=k_0}^{\infty}2^{-(k-k_0)\alpha \theta q\zeta/(1-\theta)}\bigg)^{\frac{1-\theta}{\theta}}\sum_{k=k_0}^{\infty}2^{(k-k_0)\alpha q(1+z+\zeta)}\|\chi_{\{\tau_k<\infty\}}\|_{p(\cdot)}^{ q}\gamma_b^{ q}\big(\|\chi_{\{\tau_k<\infty\}}\|_{p(\cdot)}\big)\\
    	\lesssim&\ \sum_{k_0\in\mathbb{Z}}2^{k_0q}\sum_{k=k_0}^{\infty}2^{(k-k_0)\alpha q(1+z+\zeta)}\|\chi_{\{\tau_k<\infty\}}\|_{p(\cdot)}^{ q}\gamma_b^{ q}\big(\|\chi_{\{\tau_k<\infty\}}\|_{p(\cdot)}\big).
    \end{align*}
    Hence, it follows from Abel's transformation and Theorem \ref{ad11} that
    \begin{align*}
    	\|X_2\|_{p(\cdot),q,b}^q \approx&\ \sum_{k_0\in\mathbb{Z}}2^{k_0q}\|\chi_{\{X_2>2^{k_0}\}}\|_{p(\cdot)}^{q}\gamma_b^{q}\big(\|\chi_{\{X_2>2^{k_0}\}}\|_{p(\cdot)}\big)\\
    	=&\ \sum_{k\in\mathbb{Z}}2^{k\alpha q(1+z+\zeta)}\|\chi_{\{\tau_k<\infty\}}\|_{p(\cdot)}^{ q}\gamma_b^{ q}\big(\|\chi_{\{\tau_k<\infty\}}\|_{p(\cdot)}\big)\sum_{k_0=-\infty}^{k}2^{k_0q[1-\alpha(1+z+\zeta)]}\\
    	\lesssim&\ \sum_{k\in\mathbb{Z}}2^{kq}\|\chi_{\{\tau_k<\infty\}}\|_{p(\cdot)}^{ q}\gamma_b^{ q}\big(\|\chi_{\{\tau_k<\infty\}}\|_{p(\cdot)}\big)\\
    	\lesssim&\ \|f\|_{H_{p(\cdot),q,b}}^q.
    \end{align*}

    Next, we discuss the case of $q=\infty$. According to (\ref{e200}) and H\"{o}lder's inequality, we have
    \begin{align*}
    	\|\chi_{\{|X_2|>2^{k_0}\}}\|_{p(\cdot)}
    	\lesssim&\ 2^{-k_0\beta}\bigg(\sum_{k=k_0}^{\infty}2^{k(\beta-\alpha)\varepsilon}2^{k\varepsilon\alpha}\|\chi_{\{\tau_k<\infty\}}\|_{p(\cdot)}^\varepsilon\bigg)^{1/\varepsilon}\\
        \le&\ 2^{-k_0\beta}\bigg(\sum_{k=k_0}^{\infty}2^{k(\beta-\alpha)\varepsilon}\bigg)^{1/\varepsilon}\sup_{k\ge k_0}2^{k\alpha}\|\chi_{\{\tau_k<\infty\}}\|_{p(\cdot)}\\
    	=&\ \sup_{k\ge k_0}2^{(k-k_0)\alpha}\|\chi_{\{\tau_k<\infty\}}\|_{p(\cdot)},
    \end{align*}
    where $0<\beta<\alpha<1$.
    Since $t^{-\nu}\gamma_{b}(t)$ is equivalent to a nonincreasing function for $\nu>0$ and $t\in(0,\infty)$, we have that for any $l$ satisfying $l\ge k_0$,
    \begin{align*}
    	&\ \|\chi_{\{|X_2|>2^{k_0}\}}\|_{p(\cdot)}\gamma_{b}\big(\|\chi_{\{|X_2|>2^{k_0}\}}\|_{p(\cdot)}\big)\\
    	\lesssim&\ \sup_{k\ge k_0}2^{(k-k_0)\alpha}\|\chi_{\{\tau_k<\infty\}}\|_{p(\cdot)}\gamma_{b}\Big(\sup_{k\ge k_0}2^{(k-k_0)\alpha}\|\chi_{\{\tau_k<\infty\}}\|_{p(\cdot)}\Big)\\
    	=&\ \Big(\sup_{l\ge k_0}2^{(l-k_0)\alpha}\|\chi_{\{\tau_l<\infty\}}\|_{p(\cdot)}\Big)^{1+\nu}\\
    	&\ \ \ \ \ \ \ \ \ \ \ \ \ \ \ \ \ \times\Big(\sup_{k\ge k_0}2^{(k-k_0)\alpha}\|\chi_{\{\tau_k<\infty\}}\|_{p(\cdot)}\Big)^{-\nu}\gamma_{b}\Big(\sup_{k\ge k_0}2^{(k-k_0)\alpha}\|\chi_{\{\tau_k<\infty\}}\|_{p(\cdot)}\Big)\\
    	\lesssim&\ \sup_{l\ge k_0}2^{(1+\nu)(l-k_0)\alpha}\|\chi_{\{\tau_l<\infty\}}\|_{p(\cdot)}^{1+\nu}\cdot2^{-(l-k_0)\alpha\nu}\|\chi_{\{\tau_l<\infty\}}\|_{p(\cdot)}^{-\nu}\gamma_{b}\big(2^{(l-k_0)\alpha}\|\chi_{\{\tau_l<\infty\}}\|_{p(\cdot)}\big)\\
    	=&\ \sup_{l\ge k_0}2^{(l-k_0)\alpha}\|\chi_{\{\tau_l<\infty\}}\|_{p(\cdot)}\gamma_{b}\big(2^{(l-k_0)\alpha}\|\chi_{\{\tau_l<\infty\}}\|_{p(\cdot)}\big).
    \end{align*}
    With the help of Corollary \ref{cc11}, we obtain that
    \begin{align*}
    	 \|X_2\|_{p(\cdot),\infty,b}\approx&\sup_{k_0\in\mathbb{Z}}2^{k_0}\|\chi_{\{|X_2|>2^{k_0}\}}\|_{p(\cdot)}\gamma_{b}\big(\|\chi_{\{|X_2|>2^{k_0}\}}\|_{p(\cdot)}\big)\\
    	\lesssim&\ \sup_{k_0\in\mathbb{Z}}2^{k_0}\cdot\sup_{l\ge k_0}2^{(l-k_0)\alpha}\|\chi_{\{\tau_l<\infty\}}\|_{p(\cdot)}\gamma_{b}\big(2^{(l-k_0)\alpha}\|\chi_{\{\tau_l<\infty\}}\|_{p(\cdot)}\big)\\
    	\lesssim&\ \sup_{k_0\in\mathbb{Z}}2^{k_0}\cdot\sup_{l\ge k_0}2^{(l-k_0)\alpha(1+z)}\|\chi_{\{\tau_l<\infty\}}\|_{p(\cdot)}\gamma_b\big(\|\chi_{\{\tau_l<\infty\}}\|_{p(\cdot)}\big)\\
    	=&\ \sup_{l\in\mathbb{Z}}2^{l\alpha(1+z)}\|\chi_{\{\tau_l<\infty\}}\|_{p(\cdot)}\gamma_b\big(\|\chi_{\{\tau_l<\infty\}}\|_{p(\cdot)}\big)\sup_{k_0\le l}2^{k_0(1-\alpha(1+z))}\\
    	\lesssim&\ \sup_{l\in\mathbb{Z}}2^{l}\|\chi_{\{\tau_l<\infty\}}\|_{p(\cdot)}\gamma_b\big(\|\chi_{\{\tau_l<\infty\}}\|_{p(\cdot)}\big)\\
    	\lesssim&\ \|f\|_{H_{p(\cdot),\infty,b}}.
    \end{align*}
    Combining the above inequalities, we conclude that
    $$\|T(f)\|_{p(\cdot),q,b}\lesssim\|D_1+X_1\|_{p(\cdot),q,b}+\|X_2\|_{p(\cdot),q,b}\lesssim\|f\|_{H_{p(\cdot),q,b}}.$$

    Now consider $\max{\{p_+,1\}}<r<\infty$. The estimation of $D_2$ is the same as above. For $D_1$, similarly to the proof of atomic decomposition theorem, we obtain
    $$\|D_1\|_{p(\cdot),q,b}\lesssim\|f\|_{H_{p(\cdot),q,b}}.$$
    Hence, the proof of this theorem is complete.
\end{proof}

Theorem 5.39 in \cite{jwzw} shows that the operator $\sigma_*$ satisfies the condition of the above theorem. More exactly,

\begin{lem}
	Let $p(\cdot)\in\mathcal{P}(\Omega)$ satisfy $(\ref{gs3})$.
	If $\frac{1}{2}<p_-\le p_+<\infty$ and
	\begin{align}\label{,}
		\frac{1}{p_-}-\frac{1}{p_+}<1.
	\end{align}
	For some $0<\beta<1$ and every $(p(\cdot),\infty)^M$-atom $a$ associated with stopping time $\tau$, there is
	\begin{align*}
		 \big\||\sigma_*a|^\beta\chi_{\{\tau=\infty\}}\big\|_{p(\cdot)}\lesssim\|\chi_{\{\tau<\infty\}}\|_{p(\cdot)}^{1-\beta}.
	\end{align*}
\end{lem}

Now, the boundedness of $\sigma_*$ from variable Hardy-Lorentz-Karamata spaces to variable Lorentz-Karamata spaces follows from Theorem \ref{T2}.

\begin{coro}\label{sig}
	Let $p(\cdot)\in\mathcal{P}(\Omega)$ satisfy $(\ref{gs3})$ and $(\ref{,})$, $0<q\le\infty$ and let $b$ be a slowly varying function. If $\frac{1}{2}<p_-\le p_+<\infty$, then
	$$\|\sigma_*f\|_{p(\cdot),q,b}\lesssim\|f\|_{H_{p(\cdot),q,b}}.$$
\end{coro}

\begin{rem}
   	$(1)$ If $b\equiv1$, then $(\ref{,})$ is also a sufficient condition $($see \cite{jwzw}$)$.
   	
   	$(2)$ When $b\equiv1$, for $p(\cdot)\equiv p$ with $p\le\frac{1}{2}$ or with $p=q\le\frac{1}{2}$, Theorem $\ref{sig}$ does not hold any more, see \cite{gg,s,sw}.
\end{rem}

By virtue of Theorem \ref{sig}, we consider the convergence of $\sigma_nf$. The proofs are omitted, since they are similar to \cite{jwzw}.

\begin{coro}\label{c11}
	Let $p(\cdot)\in\mathcal{P}(\Omega)$ satisfy $(\ref{gs3})$ and $(\ref{,})$, $0<q\le\infty$ and let $b$ be a slowly varying function. If $\frac{1}{2}<p_-\le p_+<\infty$ and $f\in H_{p(\cdot),q,b}$, then $\sigma_nf$ converges almost everywhere on $[0.1)$ and in the $L_{p(\cdot),q,b}$-norm.
\end{coro}

Let $I$ be an atom of $\mathcal{F}_k$. Define the restriction of a martingale $f$ to the atom $I$ by
$$f\chi_I:=(\mathbb{E}_n(f)\chi_I,n\ge k).$$

\begin{coro}\label{c22}
	Let $p(\cdot)\in\mathcal{P}(\Omega)$ satisfy $(\ref{gs3})$ and $(\ref{,})$, $0<q\le\infty$ and let $b$ be a slowly varying function. Suppose that $\frac{1}{2}<p_-\le p_+<\infty$ and $f\in H_{p(\cdot),q,b}$. If there exists a dyadic interval $I$ such that $f\chi_I\in L_1(I)$, then $\sigma_nf$ converges to $f$ almost everywhere on $I$ and in the $L_{p(\cdot),q,b}(I)$-norm.
\end{coro}

Moreover, for the case of $1\le p_-<\infty$, we have the next result.

\begin{coro}\label{c33}
	Let $p(\cdot)\in\mathcal{P}(\Omega)$ satisfy $(\ref{gs3})$ and $(\ref{,})$, $0<q\le\infty$ and let $b$ be a slowly varying function. If $1\le p_-\le p_+<\infty$ and $f\in H_{p(\cdot),q,b}$, then $\sigma_nf$ converges almost everywhere on $[0,1)$ and in the $L_{p(\cdot),q,b}$-norm.
\end{coro}

For the operator $\sigma_{2^n}$, we do not need the condition of $\frac{1}{2}<p_-\le p_+<\infty$.

\begin{lem}[\cite{jwzw}]
	Let $p(\cdot)\in\mathcal{P}(\Omega)$ satisfy $(\ref{gs3})$ and $(\ref{,})$.
	Then for some $0<\beta<1$ and every $(p(\cdot),\infty)^M$-atom $a$ associated with stopping time $\tau$, there is
	\begin{align*}
		 \big\|\sup_{n\in\mathbb{N}}|\sigma_{2^n}a|^\beta\chi_{\{\tau=\infty\}}\big\|_{p(\cdot)}\lesssim\|\chi_{\{\tau<\infty\}}\|_{p(\cdot)}^{1-\beta}.
	\end{align*}
\end{lem}

Similarly, we deduce the boundedness of $\sigma_{2^n}$ and some properties of convergence can be showed as follows.

\begin{coro}
	Let $p(\cdot)\in\mathcal{P}(\Omega)$ satisfy $(\ref{gs3})$ and $(\ref{,})$, $0<q\le\infty$ and let $b$ be a slowly varying function. Then for $f\in H_{p(\cdot),q,b}$,
	 $$\|\sup_{n\in\mathbb{N}}|\sigma_{2^n}f|\|_{p(\cdot),q,b}\lesssim\|f\|_{H_{p(\cdot),q,b}}.$$
\end{coro}

\begin{coro}\label{c44}
	Let $p(\cdot)\in\mathcal{P}(\Omega)$ satisfy $(\ref{gs3})$ and $(\ref{,})$, $0<q\le\infty$ and let $b$ be a slowly varying function. If $f\in H_{p(\cdot),q,b}$, then $\sigma_{2^n}f$ converges almost everywhere on $[0,1)$ and in the $L_{p(\cdot),q,b}$-norm.
\end{coro}

\begin{coro}\label{c55}
	Let $p(\cdot)\in\mathcal{P}(\Omega)$ satisfy $(\ref{gs3})$ and $(\ref{,})$, $0<q\le\infty$ and let $b$ be a slowly varying function. If $f\in H_{p(\cdot),q,b}$ and there exists a dyadic interval $I$ such that $f\chi_I\in L_1(I)$, then $\sigma_{2^n}f$ converges to $f$ almost everywhere on $I$ and in the $L_{p(\cdot),q,b}(I)$-norm.
\end{coro}

\begin{rem}
	Note that the convergence results in Corollaries $\ref{c11}$, $\ref{c22}$, $\ref{c33}$, $\ref{c44}$ and $\ref{c55}$ hold also for $\mathcal{H}_{p,\infty,b}$ and to $\mathcal{L}_{p,\infty,b}$.
\end{rem}

\begin{rem}
	The corresponding conclusions of the above theorems and lemmas for variable Hardy spaces and variable Hardy-Lorentz spaces can be found in Jiao et al. \cite{jwzw}, while they are new for Hardy-Lorentz-Karamata spaces.
\end{rem}

\section*{Declarations} 
\label{sec:declarations}


\noindent {\bf Ethical approval} Not applicable. \\

\noindent {\bf Competing interests} The authors declare that there is no competing interests. \\

\noindent {\bf Author Contributions} All authors wrote the main manuscript text and also reviewed the manuscript. \\

\noindent {\bf Funding} Zhiwei Hao is supported by the NSFC (No. 11801001) and Hunan Provincial Natural Science
Foundation (No. 2022JJ40145), Libo Li is supported by the NSFC (No. 12101223) and Hunan
Provincial Natural Science Foundation (No. 2022JJ40146). \\

\noindent {\bf Availability of data and materials} Data sharing not applicable to this article as no datasets were generated or analysed during the current study.\\

\end{document}